\documentclass[11pt]{article}
\usepackage{amssymb,amsmath,enumerate}
\newcommand{\BOX}{\ensuremath\Box}

\newenvironment{proof}{{\vskip\baselineskip\noindent\textbf{Proof:}}}%
{\hspace*{.1pt}\hspace*{\fill}\BOX\vskip\baselineskip}

\newenvironment{proofx}[1]%
{\vskip\baselineskip\noindent\textbf{Proof of {#1}:}}%
{\hspace*{.1pt}\hspace*{\fill}\BOX\vskip\baselineskip}
{\vskip\baselineskip\noindent\textbf{Proof of Theorem \protect\ref{#1}:}}%
{\hspace*{.1pt}\hspace*{\fill}\BOX\vskip\baselineskip}
{\vskip\baselineskip\noindent\textbf{Proof of Theorems \protect\ref{#1} --
\protect\ref{#2}:}}%
{\hspace*{.1pt}\hspace*{\fill}\BOX\vskip\baselineskip}

\newcommand{\C}{\mathbb{C}}
\newcommand{\N}{\mathbb{N}}
\newcommand{\R}{\mathbb{R}}

\newcommand{\dd}{\,{\rm d}}

\newcommand{\pa}{\partial}
\newcommand{\na}{\nabla}
\newcommand{\eps}{\epsilon}

\renewcommand{\Re}{\mathop{\mathrm{Re}}}
\renewcommand{\Im}{\mathop{\mathrm{Im}}}

\numberwithin{equation}{subsection}
\newtheorem{thm}{Theorem}[section]

\newtheorem{prop}[thm]{Proposition}
\newtheorem{lem}[thm]{Lemma}
\newtheorem{rem}[thm]{Remark}
\newtheorem{cor}[thm]{Corollary}

\begin{document}

\title{Gevrey stability of Prandtl expansions for $2$D Navier-Stokes flows}

\author{   \footnote{Universit\'e Paris Diderot, Sorbonne Paris Cit\'e, Institut de Math\'ematiques de Jussieu-Paris Rive Gauche, UMR 7586, F- 75205 Paris, France} David Gerard-Varet \and \footnote{Department of Mathematics, Graduate School of Science, Kyoto University
Kitashirakawa Oiwake-cho, Sakyo-ku, Kyoto 606-8502, Japan } Yasunori Maekawa \and 
\footnote{Courant Institute of Mathematical Sciences, 251 Mercer
   Street, New York, NY 10012, USA. } Nader Masmoudi}



\maketitle


\abstract
We investigate the stability of boundary layer solutions of the two-dimensional incompressible Navier-Stokes equations. We consider shear flow solutions of Prandtl type : 
$$ u^\nu(t,x,y) \, = \, \big (U^E(t,y) + U^{BL}(t,\frac{y}{\sqrt{\nu}})\,, \,  0 \big )\, , \quad 0<\nu \ll 1\,.  $$
We show that if $U^{BL}$ is monotonic and concave in $Y = y /\sqrt{\nu}$ then $u^\nu$ is stable over some time interval $(0,T)$, $T$ independent of $\nu$, under perturbations with Gevrey regularity in $x$ and Sobolev regularity in $y$. We improve in this way the classical stability results of Sammartino and Caflisch \cite{SaCa1,SaCa2} in analytic class (both in $x$ and $y$). Moreover, in the case where $U^{BL}$ is steady and strictly concave, our Gevrey exponent for stability is optimal. The proof relies on new and sharp resolvent estimates for the linearized Orr-Sommerfeld operator. 

\section{Introduction}\label{sec.intro}
One challenging mathematical problem from fluid mechanics is to understand the limit of solutions 
$u^\nu$  of the incompressible Navier-Stokes equations: 
\begin{equation} \label{NS}
 \partial_t u^\nu + u^\nu \cdot \na u^\nu + \na p^\nu - \nu \Delta u^\nu \, = \, f\,, \quad  {\rm div} \, u^\nu \, = \, 0 \, \quad \text{ in } \Omega\,,  \quad u^\nu\vert_{\pa \Omega} \, = \, 0\, , 
\end{equation}
when the parameter $\nu$ goes to zero, for domains $\Omega \subset \R^d$
 with boundaries. Indeed,  as the Euler solution $u^0$ does not satisfy the Dirichlet condition, convergence of $u^\nu$ to $u^0$ can not hold in strong topology (say $H^1(\Omega)$). The obstacle to convergence  is a concentration phenomenon near the boundary, in the so-called {\em boundary layer}. The mathematical understanding of this boundary layer is a difficult problem. The difficulty is emphasized by Kato's criterion \cite{Kat}:  roughly, it  says  that for smooth bounded domains $\Omega$, the  Leray solutions $u^\nu$  of Navier-Stokes converge to a smooth solution $u^0$  of Euler in $L^\infty(0,T; L^2(\Omega))$  if and only if convergence holds in $L^2(\Omega)$ at time $0$ and 
$$ \nu \int_0^T \int_{d(x, \pa \Omega) \le \nu} | \na u^\nu |^2 \, \rightarrow 0, \quad \nu \rightarrow 0$$  
that is the production of enstrophy in a layer of size $\nu$ near the boundary goes to zero with $\nu$. See \cite{Kel} for refinements. 

\medskip
\noindent
The most  popular model for the boundary layer was introduced by Ludwig Prandtl in 1904, and is inspired by the heat part of the Navier-Stokes equation. In the simple half-plane case $\Omega = \R \times \R_+$ or $\Omega = \mathbb{T} \times \R_+$, the Prandtl model corresponds to the following asymptotics for $u^\nu$: 
\begin{equation} \label{prandtl_exp}
u^\nu(t,x,y) \approx (U^E, V^E)(t,x,y) + (U^{BL}, \sqrt{\nu} V^{BL})(t,x,\frac{y}{\sqrt{\nu}})) \,,
\end{equation}
where $u^0 = (U^E, V^E)$ is the Euler solution, and $(U^{BL}, V^{BL})$ is a corrector localized near the boundary, at typical scale $\sqrt{\nu}$, allowing $u^\nu$ to satisfy  the Dirichlet condition. After a few manipulations, one obtains the so-called Prandtl equation on 
$$U^{P}(t,x,Y) \, = \, U^E(t,x,0) + U^{BL}(t,x,Y)\,, \quad V^{P}(t,x,Y) \, = \, \pa_Y V^E(t,x,0) Y  + V^{BL}(t,x,Y)$$ 
that is 
\begin{equation} \label{P}
\begin{aligned}
&\pa_t U^P + U^P \pa_x U^P + V^P \pa_Y U^P - \pa_Y^2 U^P  \, = \, - \pa_x P^E|_{y=0}\,, \\
& \pa_x U^P + \pa_Y V^P  \, = \, 0\, , \\
& U^P\vert_{Y = 0} \, = \, V^P\vert_{Y=0} \, = \, 0\, , \quad \lim_{Y \rightarrow +\infty} U^P \, = \, U^E|_{y=0}\,. 
\end{aligned}
\end{equation}
Notation $P^E$ refers to the Euler pressure field. 

\medskip
\noindent
When trying to justify the Prandtl expansion \eqref{prandtl_exp}, two difficulties are in order: 
\begin{itemize}
\item Justifying the well-posedness of the reduced model \eqref{P}, so as to construct the expansion.
\item Justifying that it approximates well the Navier-Stokes solution over some reasonable time scale,  starting from close initial data. Note that the Prandtl expansion can almost be seen as a solution of Navier-Stokes:  up to the addition of lower amplitude corrections in the expansion, the error source term can be made arbitrarily small. In that sense, the validation of the boundary layer asymptotics can be seen as a stability problem for a special class of solutions of Navier-Stokes. 
\end{itemize}
As regards the local in time well-posedness theory for \eqref{P}, the situation is by now rather well-understood. Two important factors are the monotonicity of the initial data $U^P(0,x,Y)$ with respect to $Y$, and its regularity with respect to $x$. In physical terms, the monotonicity assumption prevents (at least for small time) the separation of the boundary layer. Hence:  
\begin{itemize} 
\item Under a monotonicity assumption in $Y$ (plus other regularity requirements), local in time existence of smooth ($C^k$) solutions was proved by Oleinik \cite{Ole}. The proof is based on a change of  variables and unknowns called Crocco transform. Recently, the well-posedness theory for $Y$-increasing data was revisited in the Eulerian form \eqref{P} and Sobolev setting \cite{AlWaXuTa,MaWo}. 
\item Without the monotonicity assumption, the situation is much less favorable. Existence of local in time analytic solutions, for analytic initial data, was shown  by Sammartino and Caflisch in \cite{SaCa1}, see also \cite{LoCaSa,KuVi}. Article \cite{GeDo} by the first author and Dormy established the ill-posedness of  \eqref{P} in the Sobolev setting. More recently, local Gevrey well-posedness for data with non-degenerate critical points was shown in \cite{GeMa}. We refer to \cite{GeDo2,GeNg,KuMaViWo,LiWaYa} for more on the study of the Prandtl equation itself. 
\end{itemize}
  
The present paper is devoted to the second issue, namely the stability of the Prandtl expansion within the Navier-Stokes evolution.  Again, one may expect the stability/instability to depend on the monotonicity properties of the boundary layer flow $U^P$, and on the regularity of the perturbations under consideration.  

\medskip
On the positive side, local in time convergence of the Prandtl asymptotics was achieved in \cite{SaCa2}  in the analytic setting. Recently, the second author relaxed that result, treating the case of Sobolev data with vorticity away from the boundary \cite{Mae}. One can also mention the works \cite{LoMaNuTa,MaTa}, where convergence is shown to hold under low regularity, but  under stringent structural conditions on both  the boundary layer solution and the perturbations.    

\medskip
The main point with analyticity is that perturbations with frequencies $n$ in $x$, $n \gg 1,$  decay like $\exp(-\delta n)$ for some $\delta > 0$. This somehow eliminates any exponential instability of high frequencies. As soon as Sobolev perturbations are allowed, highly oscillatory perturbations can destabilize the Prandtl expansion, and make the exact Navier-Stokes and the boundary layer approximation diverge in very short time (typically a time $t_\nu$ going to zero with $\nu$). This kind of argument was first used by Grenier in article \cite{Gre}. The focus of \cite{Gre} is  put on Prandtl expansions of shear flow type, that read 
$$ u^\nu(t,x,y) \, =\,  (U^E(t,y), 0) + (U^{BL}(t,\frac{y}{\sqrt{\nu}}), 0)\,, \quad 0<\nu \ll 1\,.  $$
More precisely, Grenier considers the special case where $U^E = 0$ and $U^{BL} = U^{BL}(t,Y)$ solves the one-dimensional heat equation with Dirichlet condition. Through the change of variable  $(\tau,X,Y) = (t,x,y)/\sqrt{\nu}$, the stability of the boundary layer expansion amounts to the stability of $(U^{BL}(\sqrt{\nu} \tau, Y),0)$ as a solution of the rescaled Navier-Stokes equation
\begin{equation} \label{rescaled_NS}
 \pa_\tau u  + u \cdot \na_{X,Y}  u + \na_{X,Y} p - \sqrt{\nu} \Delta_{X,Y} u \, = \, 0\, . 
 \end{equation}
In particular, if the initial shear flow $(U^{BL}(0, Y),0)$ is linearly  unstable for the Euler equation, it is likely that this instability persists for $\nu \ll 1$.  Indeed, starting from such unstable shear flows, it is  shown in \cite{Gre} that equation \eqref{rescaled_NS} admits  solutions of the type 
$$ u(t,x,Y) \approx (U^{BL}(\sqrt{\nu} \tau, Y),0)  + \nu^N e^{i \alpha (X - c\tau)} v^\nu(Y)  $$ 
 where $\alpha \in \R^*$, $\alpha \Im c > 0$, , and $N$  large enough. Back to the original variables, it yields a solution that oscillates like $\exp(i \alpha x/\sqrt{\nu})$ and  separates from the Prandtl solution  as $\exp(\delta t /\sqrt{\nu})$ (over times of order $t = O(\sqrt{\nu} |\ln \nu|$). 
 This causes instability of these Prandtl expansions in any Sobolev space. 
More precisely, the analysis of \cite{Gre} exhibits perturbations with high frequencies $n = 1/\sqrt{\nu} \gg 1$ that grow like $e^{(\Im c) n t}$. Thus, it tends to indicates that stability over times of order $1$ is only possible in the analytic setting: hence, the general  convergence result of Sammartino and Caflisch in analytic regularity is likely to be optimal.

\medskip
Let us stress that, for the shear flow to be linearly unstable in the Euler equation, the profile $Y \rightarrow U^{BL}(0,Y)$ must have an inflexion point: this is the famous Rayleigh criterion,  {\it cf} the discussion in \cite{DrRe}. In experiments or numerics, the appearance of an inflexion point often goes along with a loss of monotonicity in $Y$, for which we have seen that strong instabilities develop in the reduced Prandtl model itself, in link with boundary layer separation. From this point of view, the instability result in \cite{Gre} may appear less surprising. 

\medskip
Hence, a natural problem is the consideration of Prandtl approximations of shear flow type, when $U^{BL}(0,Y)$ is increasing and concave in $Y$. It corresponds to a type of data for which the Prandtl system is well-posed, so that only the justification of the expansion has to be investigated. Following the approach  above, a key step is the linear stability analysis of monotonic and concave shear flows for the rescaled Navier-Stokes equation \eqref{rescaled_NS}. This problem has of course a long history :  we refer to \cite{DrRe} for an overview and a list of references. It has been revisited in an accurate and rigorous way recently, in remarkable works by Grenier, Guo and Nguyen \cite{GrGuNg,GGN2014}. A main conclusion  is that a shear flow which is linearly stable for Euler may be linearly unstable for Navier-Stokes: viscosity has a destabilizing effect.  In particular, paper \cite{GGN2014} studies the linearized Navier-Stokes equations around shear flows that are steady, strictly concave, monotonic, analytic near $Y = 0$. They construct explicit growing solutions  in the form $e^{i \alpha (X - c \tau)} v_\nu(Y)$, where 
$$ \alpha\sim \mathcal{O}(\nu^\frac18)\,, \quad \Im c \sim \mathcal{O}(\nu^\frac18). $$
More details will be given later on. Back to the original scales, we find that some perturbations with frequency $n \sim \mathcal{O}(\nu^{-\frac38})$ in $x$  are amplified by an exponential factor $\exp(\delta  n^{2/3} t)  \sim \mathcal{O}( \exp(\delta \nu^{-\frac14}) )$. Hence, contrary to the previous non-monotonic setting, stability of the boundary layer over times $t \sim \mathcal{O}(1)$ does not seem to require analyticity, but rather Gevrey regularity (with exponent at most $3/2$ for the case considered in \cite{GGN2014}). 
The goal of the present paper is to prove such Gevrey stability, in the nonlinear framework. Precise statements will be given in the next section.

\section{Statement of the main result} \label{sec.result}
Let $U^E = U^E(t,y)$ and $U^P=U^P(t,Y)$ two scalar functions on $\R_+\times \R_+$, satisfying the conditions   
\begin{align}\label{sec2.condition}
U^E (t,0)> 0, \qquad U^{P}(t,0) = 0\,, ~ \lim_{Y \rightarrow +\infty} U^{P}(t,Y) = U^E(t,0)~~{\rm for}~t\geq 0 \,.
\end{align}
Let 
\begin{align}
U^{BL}(t,Y) \, = \, U^P(t,Y) - U^E(t,0)\,.\label{sec2.U^{BL}}
\end{align}
Then, the shear flow $\big( U^E(t,y) + U^{BL}(t,y/\sqrt{\nu})\big)  {\bf e}_1 $ is of boundary layer type, and the goal is to investigate its stability. Therefore, we write the Navier-Stokes equation \eqref{NS} in perturbative form: 
$$u^\nu(t,x,y) \, = \, \big(U^E(t,y) + U^{BL}(t,\frac{y}{\sqrt{\nu}})\big) {\bf e}_1 + u(t,x,y)$$
with 
\begin{equation}\label{eq.perturb.intro}
\left\{
\begin{aligned}
\partial_t u  + \mathbb{A}_{\nu} (t) u & \, = \, - \mathbb{P} \big ( u \cdot \nabla u \big ) \,, \qquad t>0\,,\\
u|_{t=0} & \, = \, a\,.
\end{aligned}\right.
\end{equation}

\noindent 
Here $u=(u_1(t,x,y),u_2(t,x,y))$, $(x,y)\in \Omega = \mathbb{T} \times \R_+$ ($2\pi$-periodic in $x$), and $\mathbb{P}: L^2 (\Omega)^2 \rightarrow L^2_\sigma (\Omega)$ is the Helmholtz-Leray projection, where
$$L^2_\sigma (\Omega) = \{ f\in L^2(\Omega)^2 \ | \  {\rm div}\, f = 0  \ {\rm in }  \ \Omega \,,  f_2=0 \ {\rm on} \ y=0 \}.$$
The linear operator  $\mathbb{A}_\nu (t)$ is defined as 
\begin{align}\label{def.A_t}
\mathbb{A}_\nu (t) u \, = \, -\nu \mathbb{P} \Delta u + \mathbb{P} \bigg ( \big (U^E (t) + U^{BL} (t) \big ) \partial_x u + u_2 \partial_y \big (U^E (t) + U^{BL} (t) \big ) {\bf e}_1 \bigg )\,,
\end{align}
with the domain 
$$D (\mathbb{A}_\nu (t) ) = W^{2,2}(\Omega)^2 \cap W^{1,2}_{0} (\Omega)^2\cap L^2_\sigma (\Omega)$$
 and $W^{1,2}_0 (\Omega)=\{ f\in W^{1,2} (\Omega) \ | \ f=0  \ {\rm on} \ y=0 \}$.
Let us quote that to go from \eqref{NS} to \eqref{eq.perturb.intro}, we implicitly assumed that 
\begin{align*}
 f(t,x,y) & \, = \, \big(\pa_t U^{BL}(t,\frac{y}{\sqrt{\nu}}) - \pa^2_Y U^{BL}(t,\frac{y}{\sqrt{\nu}}) + \partial_t U^E (t,y) - \nu \pa_y^2 U^E (t,y) \big)  {\bf e}_1 \\
 & \, = \,  \big(\pa_t U^{P}(t,\frac{y}{\sqrt{\nu}}) - \pa^2_Y U^{P}(t,\frac{y}{\sqrt{\nu}})  +\partial_t U^E (t,y)- \partial_t U^E (t,0)- \nu \pa_y^2 U^E (t,y) \big)  {\bf e}_1\\
\end{align*}
Two cases deserve special attention. On one hand, if  $U^E (t,y ) \, = \, \mathcal{U}^E (y) + \mathcal{R}^E (t,y)$, where  $\mathcal{U}^E$ is given  and $\mathcal{R}^E$ solves the heat equation
\begin{align}\label{eq.R^E}
\partial_t \mathcal{R}^E - \nu \partial_y^2 \mathcal{R}^E \, = \, \nu \partial_y^2 \mathcal{U}^E \,, \quad t>0\,, ~ y>0\,, \qquad \mathcal{R}^E |_{y=0} \, =\, 0\,, \quad \mathcal{R}^E |_{t=0} \, =\, 0,
\end{align} 
and if moreover $U^P$ solves the heat equation 
\begin{equation} \label{heat_Up}
\pa_t U^P - \pa^2_Y U^P \, = \, 0, \quad t>0\,, ~Y>0\,, \qquad U^P\vert_{Y = 0} \, = \, 0\,, \lim_{Y \rightarrow +\infty}   U^P \, = \, U^E|_{y=0}\,,
\end{equation}
then the boundary layer expansion can be seen as an exact solution of the homogeneous Navier-Stokes equation. On the other hand, when  $U^P(T,Y) = U(Y)$ and $U^E(t,y) = U^E(y)$ are time independent, our stability issue connects to the long history of the stability of steady shear flows \cite{DrRe}. 

\medskip
\noindent
As discussed in the introduction, the goal is to prove the stability of such Prandtl boundary layer expansions within the Navier-Stokes evolution, under appropriate monotonicity and concavity conditions on $U^P$. Such stability result will be achieved in a Gevrey framework. Roughly, we will show that if a solution $u$ of \eqref{eq.perturb.intro} is initially $O(\nu^{\frac32})$  in some space of Gevrey regularity in $x$, $L^2$ in $y$, it will remain so over a time interval $(0,T)$ independent of $\nu$. The Gevrey exponent that we obtain is better when we consider steady $U^P$ rather than time dependent. In the case where $U^P$ is steady and strictly concave, our Gevrey exponent is optimal, taking into account recent results of Grenier, Guo and Nguyen \cite{GGN2014}. This will be explained in Remark \ref{rem.parameter}.

\medskip
\noindent
To state our main theorem, we need to introduce a few notations.  Let
\begin{align}\label{def.projection.P}
\begin{split}
\big ( \mathcal{P}_n f \big ) (y)  & \, = \, f_n (y) e^{i n x} \,, \qquad f_n  (y)  \, = \, \frac{1}{2\pi} \int_0^{2\pi} f (x, y) e^{-i n x} \dd x \,, \qquad n \in \mathbb{Z}\,,
\end{split}
\end{align} 
the projection on the Fourier mode $n$ in $x$, $n \in \mathbb{Z}$. We then introduce, for $\gamma\in (0,1]$, $d\geq 0$, and $K>0$ the Banach space $X_{d,\gamma,K}$ as 
\begin{align}
X_{d,\gamma, K} \, = \, &\{ f\in L^2_\sigma (\Omega) ~|~ \| f \|_{X_{d,\gamma,K}} = \sup_{n\in \mathbb{Z}} \, (1+|n|^d) e^{K \theta_{\gamma,n}} \| \mathcal{P}_n f \|_{L^2 (\Omega)} <\infty \}\,,\\
\mbox{} ~~~~~\qquad \theta_{\gamma,n} & \, = \, |n|^\gamma \big ( 1+ (1-\gamma)\log (1+|n|) \big ) \,.\label{sec2.def.theta}
\end{align}
Fields in this space have $L^2$ regularity in $y$, and Gevrey regularity in $x$, of class $s$ for any $s\geq \frac{1}{\gamma}$. When the boundary layer profile is monotonic and concave in $Y$, we will prove stability on times of order $1$ for initial perturbations that are $O(\nu^{l})$ in $X_{d,\gamma, K}$ for suitable $l\in (\frac12,\frac32)$ and $\gamma \in (0,1)$ (hence, below analytic regularity). The value of the exponent $\gamma$ will depend on the type of concavity that we impose. We  distinguish between three settings. 

\medskip
\noindent
{\bf Steady flow, weakly concave case}. We consider here $U^E(t,y)=U^E(y)$ and $U^P(t,Y) = U(Y)$, with 
\begin{align}\label{sec2.bound.U}
\begin{split}
& \|U\|\, := \, \sum_{k=0,1,2} \sup_{Y \geq 0}  \, (1+Y)^k |\pa_Y^k U(Y) | \, <\, \infty\,.
\end{split}
\end{align}
Our assumptions are: 
\begin{align}\tag{WC}\label{concave.weak}
\begin{split}
& {\rm (i)} ~U|_{Y=0} \, =\, 0\,, \, \,  \lim_{Y\rightarrow \infty} U \, =\, U^E|_{y=0}\,, \quad U^E, U\in BC^2(\R_+) ~{\rm and }~ U {\rm ~satisfies ~\eqref{sec2.bound.U}},\\
& {\rm (ii)} ~{\rm For ~each ~}\sigma \in (0,1] {\rm ~there ~exists ~} M_\sigma >0 {\rm ~ such ~that}~\,  -M_\sigma \pa_Y^2U \geq (\pa_Y U)^2 \, ~ {\rm for} ~Y\geq \sigma\,.
\end{split}
\end{align}
The item {\rm (ii)} in \eqref{concave.weak} implies $-\pa_Y^2 U (Y)\geq 0$ for all $Y\geq 0$, and thus, 
 it ensures the concave shape of $U$. Moreover,  it is necessary for \eqref{concave.weak} that $\pa_Y U$ is nonnegative, and then, the boundary condition $U(0)=0$, $U(\infty)=U^E(0)$ implies that $U^E(0)$ must be nonnegative. That is, the condition \eqref{concave.weak} automatically leads to
\begin{align*}
\inf_{Y\geq 0} \pa_Y U(Y)\geq 0\,, \qquad U^E(0)\geq 0\,.
\end{align*}
Moreover, the following statements hold:
\begin{align}\label{sec2.nontrivial.1}
\begin{split}
&\cdot ~{\rm if}~~\pa_Y U(Y_0)=0~~{\rm for~some}~~Y_0\geq 0\,,~{\rm then}~~\pa_Y U(Y)=0~~{\rm for~all}~~Y\geq Y_0\,,\\
&\cdot ~U^E(0)=0~~{\rm if~and~only~if}~~U(Y)=0~~{\rm for~all}~~Y\geq 0\,.\end{split}
\end{align}
Since we are interested in the stability of the boundary layer, the case $U^E(0)=0$ is excluded in this paper
in order to have the nontrivial boundary layer profile $U$. Therefore,  as stated in \eqref{sec2.condition}, we always assume the positivity of $U^E(0)$, that is,
\begin{align}\label{sec2.nontrivial.2}
U^E(0)>0\,.
\end{align}
Then, from the statements \eqref{sec2.nontrivial.1} and the boundary condition on $U$, the condition $U'(0)>0$ is always satisfied. 

\medskip
\noindent
\begin{rem}\label{sec2.rem.nontrivial}{\rm   Let
\begin{align*}
V (Y) \, = \, U^E (\sqrt{\nu} Y) - U^E (0) + U (Y) \,, 
\end{align*}
which is the full boundary layer expansion expressed in variable $Y$. 
By \eqref{sec2.nontrivial.1} the derivative of the boundary layer, $\pa_Y U$,  must be strictly positive in the compact set $D=\{Y\geq 0~|~ U(Y) \leq \frac34 U^E(0)\}$. Thus we have
\begin{align*}
\pa_Y V(Y)\geq \frac{1}{2} \min_{Y\in D} \pa_Y U(Y) >0 \quad {\rm for}~Y\in D \qquad {\rm  if} ~\nu~ {\rm is~ sufficiently~ small}\,.
\end{align*}
Then it is not difficult to see that \eqref{concave.weak} implies the following integral condition:
\begin{align}\tag{${\rm IC}$}\label{integral.condition}
\begin{split}
& {\rm There ~exists ~} M'>0~{\rm such~that}~\, \|\frac{Y^\frac12 (\pa_Y U)^2}{(V-\lambda)^2}\|_{L^2}\leq M' (\Im \lambda)^{-\frac32}\, ~{\rm holds}\\
&{\rm for ~all}~\, \lambda\in \{ \mu\in \C~|~ \Re \mu \leq \frac{U^E(0)}{2}, \, \Im \mu>0\} ~{\rm and~for~all~}\nu\in (0,\nu_0]~{\rm with}~0<\nu_0\ll 1\,. 
\end{split}
\end{align} 
The proof of \eqref{integral.condition} is straightforward and omitted here.
}
\end{rem}

\bigskip
\noindent
{\bf Steady flow,  strongly concave case}.
Here, we consider again the steady case, but assume
\begin{align}\tag{SC}\label{concave.strong}
\begin{split}
&  {\rm (i)} ~U|_{Y=0} \, =\, 0\,, \, \,  \lim_{Y\rightarrow \infty} U\, =\, U^E|_{y=0}\,, \quad U^E, U\in BC^2(\R_+) ~{\rm and }~ U {\rm ~satisfies ~\eqref{sec2.bound.U}},\\
& {\rm (ii)} ~{\rm There ~exists ~} M>0 {\rm ~ such ~that}~ \, -M \partial_Y^2 U   \geq (\partial_Y U) ^2 ~\,  {\rm ~for} ~Y\geq 0\,.
\end{split}
\end{align} 
The condition  (ii) in \eqref{concave.strong}  implies strict concavity of the profile $U$ up to the boundary $Y = 0$: $\pa_Y^2 U$ can not vanish even at $Y=0$. This is the main difference with assumptions \eqref{concave.weak}, which will allow us to push the analysis further. 

\bigskip
\noindent
{\bf Time dependent flow, weakly concave case}. In this last setting, we fix $T > 0$ and consider $U^E=U^E(t,y)$ and  $U^P = U^P(t,Y)$, $t \in [0,T]$. Our assumptions are 
\begin{align}\tag{WC-t}\label{concave.weak.t}
\begin{split}
& {\rm (0)}~U^P|_{Y=0} =0\,, \quad \lim_{Y\rightarrow \infty} U^P =U^E|_{y=0}>0 \qquad {\rm for} ~t\in [0,T]\,,\\ 
& {\rm (i)} ~U^E, ~U^P\in BC([0,T]\times \overline{\R_+})\cap L^\infty (0,T; BC^2 (\R_+))\,, {\rm ~and}\\
& \quad \sup_{0<t<T} \bigg ( \| \pa_t U^E (t) \|_{L^\infty (\R_+)} + \| \pa_t U^P (t) \|_{L^\infty(\R_+)} \\
& \qquad \qquad \qquad \qquad \qquad + \| Y \pa_t \pa_Y U^P (t) \|_{L^\infty (\R_+)} + \| U^P (t) \| \bigg ) <\infty \,, \\
& \quad {\rm Here }~\| U^P (t) \|~ {\rm is ~ defined ~as ~in ~\eqref{sec2.bound.U}.} \\
& {\rm (ii)} ~{\rm For ~each ~}\sigma \in (0,1] {\rm ~there ~exists ~} M_\sigma >0 {\rm ~independent ~of~}t \in [0,T] {\rm ~ such ~that}~ \\
& \quad - M_\sigma \partial_Y^2U^P  \geq (\partial_Y U^P)^2\,  {\rm ~for} ~Y\geq \sigma ~{\rm and }~t\in [0,T]\,.
\end{split}
\end{align}
We emphasize that the condition \eqref{concave.weak.t} is always satisfied when $U^P$ is the solution to 
\begin{equation}\label{eq.heat}
\left\{
\begin{aligned}
\partial_t U^P - \partial_Y^2 U^P & \, = \, 0\,, &  t>0\,, ~ Y>0\,,\\
U^P|_{Y=0} & \, = \, 0 \,, \quad \lim_{Y\rightarrow \infty} U^P \, = \, U^E|_{y=0} \,, &  t\geq 0\,,\\
U^P |_{t=0} & \, = \, U\,. & 
\end{aligned}\right.
\end{equation}
with $\displaystyle \inf_{0\leq t\leq T} U^E|_{y=0}>0$, if the initial data $U \in BC^3 (\R_+)$ satisfies \eqref{sec2.bound.U} and $Y\pa_Y^3U\in L^\infty (\R_+)$,  compatibility conditions on $Y=0,\infty$,  and (ii) of \eqref{concave.strong}. See the appendix for details. 
 
 \bigskip
 \noindent 
 We are now ready to state our main result: 
 \begin{thm}\label{sec2.thm.nonlinear} Assume that \eqref{concave.weak.t} holds for some $T>0$. For any $\gamma \in [\frac79,1)$, $d>\frac92-\frac72\gamma$, and $K>0$, there exist $C, T,' K'>0$ such that the following statement holds for any sufficiently small $\nu>0$.
If $\| a\|_{X_{d,\gamma,K}} \leq \nu^{\frac12+\beta}$ with $\beta = \frac{2(1-\gamma)}{\gamma}$ then the system \eqref{eq.perturb.intro} admits a unique solution $u\in C([0,T']; L^2_\sigma (\Omega))\cap L^2 (0,T'; W^{1,2}_0 (\Omega)^2)$ satisfying the estimate
\begin{align}\label{sec2.est.thm.nonlinear.1}
\sup_{0<t\leq T'} \big ( \| u (t) \|_{X_{d,\gamma,K'}} + (\nu t)^\frac14 \| u (t) \|_{L^\infty (\Omega)} +  (\nu t)^\frac12 \| \nabla u (t) \|_{L^2 (\Omega)} \big )  \leq C \| a\|_{X_{d,\gamma,K}}\,.
\end{align}

\medskip
\noindent
If $U^E(t,y)=U^E(y)$ and $U^P(t,Y) = U(Y)$ are steady and satisfy \eqref{concave.weak} instead of \eqref{concave.weak.t}, the result holds for any $\gamma \in [\frac57,1]$. 

\medskip
\noindent
If  $U^E(t,y)=U^E(y)$ and $U^P(t,Y) = U(Y)$ are steady and satisfy \eqref{concave.strong} instead of \eqref{concave.weak.t}, the result holds for any $\gamma \in [\frac23,1]$. 
\end{thm}

\begin{rem}{\rm (i) Since $a\in L^2_\sigma (\Omega)$ and the problem is a two-dimensional one, the unique existence of global solutions to \eqref{eq.perturb.intro} in $C([0,\infty); L^2_\sigma (\Omega))\cap L^2_{loc} (0,\infty; W^{1,2}_0 (\Omega)^2)$ is classical for any $\nu>0$. The nontrivial part of Theorem \ref{sec2.thm.nonlinear} is the estimate \eqref{sec2.est.thm.nonlinear.1}, which is uniform with respect to sufficiently small $\nu>0$.

\noindent 
(ii) Gevrey stability as in Theorem \ref{sec2.thm.nonlinear} can be obtained under the slightly weaker condition $-M_\sigma \pa_Y^2 U\geq (\pa_Y U)^4$, rather than $-M_\sigma \pa_Y^2 U \geq (\pa_Y U)^2$ in \eqref{concave.weak}. However, under this weaker condition the range of the exponent $\gamma$ is confined to $[\frac45,1]$ in the case of time-dependent shear flows, which is narrower than $[\frac79,1]$, and one also needs to take larger $\beta$ for the size of initial perturbations. Hence, in this paper we restrict ourselves to the case $-M_\sigma \pa_Y^2 U \geq (\pa_Y U)^2$, which is still admissible in applications.

\noindent 
(iii) When the shear flow is steady as mentioned in the latter part of Theorem \ref{sec2.thm.nonlinear}, one can in fact obtain slightly stronger result than stated in Theorem \ref{sec2.thm.nonlinear}. In particular, when the shear flow is steady and \eqref{concave.strong} holds, the exponent $\theta_{\gamma,n}$ in the definition of $X_{d,\gamma,K}$ (see \eqref{sec2.def.theta}) is simply taken as $|n|^\gamma$. Thus, in this case, by arguing as in the proof of Theorem \ref{thm.nonlinear}, we have the stability estimate like \eqref{sec2.est.thm.nonlinear.1} for the initial data $a$ satisfying 
$$\sup_{n\in \mathbb{Z}}\, (1+|n|^d) e^{K|n|^\frac23} \| \mathcal{P}_n a\|_{L^2}\leq \nu^{\frac12 + \beta}$$
for large enough $d,\beta>0$. Indeed, the logarithmic correction of $\theta_{\gamma,n}$ in \eqref{sec2.def.theta} is needed for the estimates of the evolution operator, stated in Theorem \ref{thm.evolution.middle}, rather than the estimates of the semigroup; see Theorem \ref{thm.semigroup}.
}
\end{rem}

Our result can be seen as improving the celebrated result of Sammartino and Caflisch, dedicated to the stability of Prandtl expansions in analytic regularity, see \cite{SaCa1,SaCa2}. Note nevertheless that article \cite{SaCa1} treats general  $x$-dependent boundary layer expansions, while ours restricts to the case of shear flows. Extension of our result to arbitrary (meaning $x$-dependent) expansions is a very interesting open problem. 

\medskip
\noindent
To go from an analytic setting to a Gevrey setting requires new ideas, and of course the use of the concavity of the boundary layer  profile in $Y$. The heart of the proof is  the resolvent analysis in Section \ref{sec.middle.notime}  for the linearized Navier-Stokes operator at a  steady shear flow.  This analysis leads to a temporal growth estimate of the associated semigroup,  that is compatible with the stability in the Gevrey class.  Concretely, for each Fourier mode in the $x$ variable,  the resolvent problem is reduced to the analysis of the classical Orr-Sommerfeld equation. This equations is  solved and estimated through  an iterative process, based on the alternate resolution of the so-called  Rayleigh and Airy equations. This Rayleigh-Airy iteration finds its origin in  the work of \cite{GGN2014}, about the linear instability of monotonic shear flows in the Navier-Stokes equation.   However, we stress that the approach in \cite{GGN2014} is dedicated to a specific regime of spectral parameter and Fourier frequency, for which the  fundamental solutions of the Rayleigh and Airy equations have quite explicit expressions.  The specific regime of  \cite{GGN2014} is enough to construct an unstable eigenfunction, but is far from sufficient here. In order to estimate the evolution of the semigroup,  we need to estimate the resolvent when the spectral parameter is in  a whole sector of the complex plane, and for arbitrary frequencies. This yields a much wider regime than the one where the approach of \cite{GGN2014} can be applied.  Hence, our strategy to handle the Orr-Sommerfeld equation is very different, and based on energy methods. It notably relies on the famous Rayleigh's trick \cite[page 131]{DrRe}.  Remarkably enough, at least for the class of shear flows satisfying \eqref{concave.strong}, our approach provides a fairly  optimal result on the spectral bound of the generator as well as the growth bound of the associated semigroup. 

\medskip
\noindent
Once the estimate for the semigroup is obtained, the proof of Theorem \ref{thm.nonlinear} in the case of steady shear flows  follows from Duhamel's formula. In the time dependent case, one must substitute to the growth bound on the semigroup  a growth bound on the operator solution, over some fixed time  interval $[0,T]$. Roughly, such estimate is derived by partitioning $[0,T] $  into small subintervals $[t_l, t_{l+1}]$, and freezing the time in the linearized operator over these subintervals, that is replacing $\mathbb{A}_\nu (t)$ by $\mathbb{A}_\nu (t_l)$. This allows to connect to the analysis of the steady case. We do not manage to maintain the same  Gevrey exponent in this process, which explains the higher value of $\gamma$ in the theorem in the time dependent case. 

\medskip
\noindent 
The outline of the paper is as follows. Next section presents the mode-by-mode (in Fourier in $x$) reformulation of the problem. Section \ref{sec.low.high} collects  standard estimates on the operator solutions, at very low and very high frequencies. The main part of the paper is in Section  \ref{sec.middle.notime}: it is devoted to analysis of the linearized operator at steady shear flow and intermediate frequencies. It contains an analysis of the resolvent operator through the Orr-Sommerfeld formulation, and provides a growth bound on the associated semigroup. Section \ref{sec.middle.time} is devoted to the growth bound in the time dependent case, through the time stepping alluded to above. Finally, the proof of Theorem \ref{thm.nonlinear} is achieved in \ref{sec.nonlinear}.

\section{Formulation of linearized problem in Fourier series}\label{sec.formulation.rescale}

Since the shear flow in \eqref{def.A_t} is assumed to be $x$-independent, it is natural to study $\mathbb{A}_\nu (t)$ on each Fourier mode with respect to the $x$ variable. To this end, we use the projections $\mathcal{P}_n$, $n \in \mathbb{Z}$ given in \eqref{def.projection.P}. We define 
\begin{align}
\mathcal{Q}_{m, {\rm low}} f \, = \, \sum_{|n|\leq m-1} \mathcal{P}_n f \,, \qquad \qquad 
\mathcal{Q}_{m, {\rm high}} f \, = \, \sum_{|n|\geq m+1} \mathcal{P}_n f\,.
\end{align} 
These are projections from $L^2(\Omega)^2$ to $L^2(\Omega)^2$,
and the divergence free condition is preserved under their actions. 
More precisely, we have 
\begin{align*}
\mathcal{P}_n \mathbb{P} \, = \, \mathbb{P} \mathcal{P}_n \,, \qquad 
\mathcal{Q}_{m, {\rm low}} \mathbb{P} \, = \, \mathbb{P} \mathcal{Q}_{m, {\rm low}}\,, \qquad 
\mathcal{Q}_{m, {\rm high}} \mathbb{P} \, = \, \mathbb{P} \mathcal{Q}_{m, {\rm high}} \,.
\end{align*}
Hence, for each $m_1,m_2\in \N$ with $m_1\leq m_2$ the space $L^2_\sigma (\Omega)$ is decomposed as 
\begin{align}\label{decompose.L^2}
L^2_\sigma (\Omega) \, = \, \mathcal{Q}_{m_1, {\rm low}}  L^2_{\sigma} (\Omega) 
\, \oplus \, \mathcal{Q}_{m_2, {\rm high}}  L^2_{\sigma} (\Omega) 
\, \oplus \,   \big ( \oplus_{m_1\leq |n|\leq m_2} \mathcal{P}_n L^2_\sigma (\Omega) \big )\,.
\end{align}
Moreover, since the shear flow is independent of $x$ it is straightforward to see that for each $t\geq 0$,
\begin{align*}
\mathcal{P}_n \mathbb{A}_\nu (t) \subset \mathbb{A}_\nu (t) \mathcal{P}_n \,, \qquad 
\mathcal{Q}_{m,{\rm low}}  \mathbb{A}_\nu (t) \subset \mathbb{A}_\nu (t) \mathcal{Q}_{m,{\rm low}} \,, \qquad 
\mathcal{Q}_{m,{\rm high}} \mathbb{A}_\nu (t) \subset \mathbb{A}_\nu (t) \mathcal{Q}_{m,{\rm high}}\,,
\end{align*}
which leads to the diagonalization of $\mathbb{A}_\nu (t)$ such as 
\begin{align}\label{decompose.A}
\mathbb{A}_\nu (t) \, = \, \mathbb{A}_{\nu, m_1,{\rm low}} (t) \, \oplus \, \mathbb{A}_{\nu, m_2,{\rm high}} (t) \, \oplus \, \big ( \oplus_{m_1\leq |n|\leq m_2} \mathbb{A}_{\nu,n} (t) \big ) \,,
\end{align}
where each operator in the right-hand side of \eqref{decompose.A} is naturally defined as the restriction of $\mathbb{A}_\nu (t)$ on the invariant subspace described in the right-hand side of \eqref{decompose.L^2}.
Then, the evolution operator $\{\mathbb{T}_\nu (t,s)\}_{t\ge s\geq 0}$ generated by $-\mathbb{A}_\nu (t)$ is also diagonalized as 
\begin{align}\label{decompose.evolution}
\mathbb{T}_\nu (t,s) \, = \, \mathbb{T}_{\nu, m_1,{\rm low}} (t,s) 
\, \oplus \, \mathbb{T}_{\nu, m_2,{\rm high}} (t,s) 
\, \oplus \big ( \oplus_{m_1\leq |n|\leq m_2} \mathbb{T}_{\nu,n} (t,s)  \big )\,.
\end{align}

In Section \ref{sec.low.high} we first study the evolution operators 
$\mathbb{T}_{\nu, m_1,{\rm low}}$ and $\mathbb{T}_{\nu, m_2,{\rm high}}$
in the case $m_1=\mathcal{O}(1)$ and $m_2=\mathcal{O} (\nu^{-\frac34})$,
where their temporal growth are well controlled  by a simple energy method
for general shear flows.
On the contrary, in the middle range of frequencies $m_1 \le |n| \le m_2$, the problem becomes more complicated due to the underlying derivative loss property of the equations, and the behavior of the evolution operator  $\mathbb{T}_{\nu,n} (t,s)$ highly depends on the structure of the shear flows. 
In Section \ref{sec.middle.notime} we discuss this problem under suitable monotonicity conditions and the assumption that the shear flows are time-independent.  The case of the time-dependent shear flows are studied in Section \ref{sec.middle.time}.


\section{Linear evolution operator in low and high frequency}\label{sec.low.high}

In this section we study the evolution operators defined in the previous section for the low frequency part and the high frequency part. 
The low frequency part is $\mathbb{T}_{\nu,m_1,{\rm low}}$ with $m_1=\mathcal{O}(1)$.
In this case the derivative loss is negligible and the evolution of $\mathbb{T}_{\nu,m_1,{\rm low}}$
can be estimated without difficulty for general time-dependent shear flows. 
For convenience we introduce the rescaled variable $Y=\frac{y}{\sqrt{\nu}}$, which will be used for the boundary layer $U^P$.

\begin{prop}\label{prop.general.low}
If $U^E, U^P\in L^\infty (0,t; BC^1 (\R_+))$ and $\| Y \partial_Y U^P \|_{L^\infty (0,t; L_Y^\infty)} <\infty$, then 
\begin{align}
& \| \mathbb{T}_{\nu,m_1,{\rm low}} (t,s) f\|_{L^2(\Omega)}   \leq e^{(t-s) C_1 (m_1,t)} \| f \|_{L^2(\Omega)}\,,\label{est.prop.general.low.1}
\end{align}
and
\begin{align}
\begin{split}
\| \nabla \mathbb{T}_{\nu,m_1,{\rm low}}(t,s) f\|_{L^2(\Omega)}  \leq \frac{C}{\nu^\frac12 (t-s)^\frac12} \, \big ( 1 + (t-s) C_2 (m_1,t) e^{(t-s) C_1 (m_1, t)} \big )   \|  f \|_{L^2(\Omega)}\,.\label{est.prop.general.low.2}
\end{split}
\end{align}
for all $f\in \mathcal{Q}_{m_1, {\rm low}}  L^2_{\sigma} (\Omega)$. Here $C$ is a universal constant and 
\begin{align*}
C_1 (m_1,t) & \, = \, \| \partial_y U^E\|_{L^\infty(0,t; L^\infty_y)} + 2 (m_1-1) \| Y \partial_Y U^P \|_{L^\infty (0,t; L_Y^\infty)}\,,\\
C_2 (m_1,t) & \, = \,  C_1 (m_1,t) + \big ( \| U^E \|_{L^\infty(0,t; L^\infty_y)} + \| U^P\|_{L^\infty (0,t; L^\infty_{Y})} \big ) (m_1-1) \,.
\end{align*}

\end{prop}

\begin{proof} It suffices to consider the case $s=0$. 
Set $v_{{\rm low}} (t) = \mathbb{T}_{\nu,m_1,{\rm low}} (t,0) f$ for $f \in \mathcal{Q}_{m_1, {\rm low}}  L^2_{\sigma} (\Omega)$. Then,  by the boundary condition and the divergence free condition,
the vertical component of $v_{{\rm low}}$ satisfies 
\begin{align}
\| y^{-1}v_{{\rm low},2} (t,x) \|_{L^2_y} \leq 2 \| \partial_y v_{{\rm low,2}} (t,x) \|_{L^2_y}  = 2 \| \partial_x v_{{\rm low},1} (t,x) \|_{L^2_y}\,.\label{proof.prop.general.low.1}
\end{align}
Here we have used the Hardy inequality in the first inequality. Thus we have
\begin{align*}
|\langle v_{{\rm low},2} \partial_y \big ( U^E - U^E(0) + U^P \big ) , \, v_{{\rm low},1} \rangle_{L^2}|  
& \leq \big ( \| \partial_y U^E\|_{L^\infty} + 2 (m_1-1)  \| Y \partial_Y U^P \|_{L_Y^\infty} \big )  \| v_{\rm low} \|_{L^2}^2 \,. 
\end{align*}
Here $\langle , \rangle_{L^2}$ denotes the inner product of $L^2 (\Omega)$ or $L^2(\Omega)^2$.
Hence we have 
\begin{align}
\frac{\dd }{\dd t} \| v_{{\rm low}} \|_{L^2}^2 & \, = \, - 2 \nu \| \nabla v_{{\rm low}} \|_{L^2}^2 - 2 \Re \langle v_{{\rm low},2} \partial_y \big ( U^E - U^E(0) + U^P \big ) , \, v_{{\rm low},1} \rangle_{L^2} \nonumber \\
& \, \leq \, - 2 \nu \| \nabla v_{{\rm low}} \|_{L^2}^2 +  2 \big ( \| \partial_y U^E\|_{L^\infty} + 2 (m_1-1)  \| Y \partial_Y U^P \|_{L_Y^\infty} \big )  \| v_{{\rm low}}\|_{L^2}^2\,,\label{proof.prop.general.low.2}
\end{align}
which gives \eqref{est.prop.general.low.1} by the Gronwall inequality. 
To show the derivative estimate let us denote by $\mathbb{A}$ the Stokes operator with the viscosity coefficient $1$, i.e., $\mathbb{A}=-\mathbb{P}\Delta$,  in $L^2_\sigma (\Omega)$. 
Then by the Duhamel formula we have 
\begin{align}\label{proof.prop.general.low.3}
\begin{split}
v_{{\rm low}} (t) & = e^{-\nu t \mathbb{A}} f \\
&- \int_0^t e^{-\nu (t-s) \mathbb{A}} \mathbb{P} \bigg ( \big ( U^E + U^{BL} (s,\frac{\cdot}{\sqrt{\nu}}) \big ) \partial_x v_{\rm low} (s) + v_{{\rm low},2} \partial_y \big ( U^E + U^{BL} (s, \frac{\cdot}{\sqrt{\nu}}) \big ) {\bf e}_1 \bigg ) \dd s
\end{split}
\end{align}
 It is well known that 
\begin{align*}
\| \nabla e^{-\nu t\mathbb{A}} f\|_{L^2 (\Omega)}  \leq \frac{C}{\nu^\frac12 t^\frac12} \| f\|_{L^2 (\Omega)}\,, \qquad t>0\,.
\end{align*}
Then we have from \eqref{proof.prop.general.low.1}, 
\begin{align*}
\| \nabla v_{{\rm low}} (t) \|_{L^2(\Omega)} & \leq \frac{C}{\nu^\frac12 t^\frac12} \| f\|_{L^2 (\Omega)}  \\
& + C \int_0^t \frac{(\|U^E(s)\|_{L^\infty} + \| U^{BL} (s) \|_{L^\infty}) (m_1-1) + C_1 (m_1,s)}{\nu^\frac12 (t-s)^\frac12}  \| v_{{\rm low}} (s)\|_{L^2(\Omega)}\dd s\,.
\end{align*}
Then \eqref{est.prop.general.low.2} follows from \eqref{est.prop.general.low.1}.
The proof is complete.
\end{proof}

Next we consider the high frequency part $\mathbb{T}_{\nu, m_2, {\rm high}} = \oplus_{|n|\geq m_2+1} \mathbb{T}_{\nu,n}$, with $m_2 = \mathcal{O}(\nu^{-\frac34})$. In this case the dissipation due to the viscosity works enough and the derivative loss does not appear. 

\begin{prop}\label{prop.general.high} Let $U^E, U^P\in L^\infty (0,t; BC^1 (\R_+))$ and $\| Y \partial_Y U^P \|_{L^\infty (0,t; L_Y^\infty)} <\infty$. Then there exists $\delta_0 \in (0,1)$ such that if $|n|\geq \delta_0^{-1}\nu^{-\frac34}$ then 
\begin{align}
\| \mathbb{T}_{\nu, n} (t,s) f\|_{L^2(\Omega)}  & \leq e^{-\frac14 \nu n^2 (t-s)} \| f \|_{L^2(\Omega)}\,,\label{est.prop.general.high.1}
\end{align}
and
\begin{align}
\| \nabla \mathbb{T}_{\nu, n} (t,s)  f\|_{L^2(\Omega)} & \leq \frac{C e^{-\frac14 \nu n^2 (t-s)}}{\nu^\frac12(t-s)^\frac12} \big ( 1 + |n| (t-s) C_2 (t) \big ) \| f\|_{L^2(\Omega)}\,,\label{est.prop.general.high.2}
\end{align}
for all $f\in \mathcal{P}_{n}  L^2_{\sigma} (\Omega)$ and $t>s\geq 0$.
Here $C$ is a universal constant and $C_2(t)$ is given by 
\begin{align*}
C_2 (t) \, = \, \| U^E \|_{L^\infty(0,t; BC^1_y)} + \| U^P \|_{L^\infty (0,t; L^\infty_Y)} + \| Y \partial_Y U^P \|_{L^\infty (0,t; L^\infty_Y)}\,.
\end{align*}
\end{prop}

\noindent 
\begin{proof} Again it suffices to consider the case $s=0$, and set $v_{{\rm high}} (t) =\mathbb{T}_{\nu, n} (t,0) f$ for $f \in \mathcal{P}_{n}  L^2_{\sigma} (\Omega)$, 
where $|n|\geq \delta_0^{-1}\nu^{-\frac34}$ and $\delta_0$ will be determined later.
The standard energy method yields
\begin{align}
 \frac{\dd }{\dd \tau} \| v_{{\rm high}} \|_{L^2}^2 & \, = \, - 2 \nu \| \nabla v_{{\rm high}} \|_{L^2}^2 - 2 \Re \langle v_{{\rm high},2} \partial_y\big (  U^E-U^E (0) + U^P\big ), \, v_{{\rm high},1} \rangle_{L^2} \nonumber \\
& \, \leq \, - \nu \| \nabla v_{{\rm high}} \|_{L^2}^2 - \nu n^2 \| v_{{\rm high}} \|_{L^2}^2 \nonumber \\
& \quad + 2 \big ( \| \partial_y U^E \|_{L^\infty}  + \nu^{-\frac12} \| \partial_Y U^P \|_{L^\infty_{t,Y}} \big ) \| v_{{\rm high}}\|_{L^2}^2\,.\label{proof.prop.general.high.1}
\end{align}
Thus, if we set 
\begin{align}
\delta_0 \, = \, \frac{1}{2 (1+\| \partial_y U^E \|_{L^\infty}  + \| \partial_Y U^P \|_{L^\infty_{t,Y}})^\frac12}\,, \label{proof.prop.general.high.2}
\end{align}
then \eqref{proof.prop.general.high.1} and the condition $|n|\geq \delta_0^{-1} \nu^{-\frac34}$ lead to
\begin{align}
\frac{\dd }{\dd \tau} \| v_{{\rm high}} \|_{L^2}^2 
& \, \leq \, - \nu \| \nabla v_{{\rm high}} \|_{L^2}^2 - \frac{\nu n^2}{2} \| v_{{\rm high}} \|_{L^2}^2\,.\label{proof.prop.general.high.3}
\end{align}
By the Gronwall inequality Est.\eqref{est.prop.general.high.1} follows.
The derivative estimate is proved as in the proof of Proposition \ref{prop.general.low}.
Indeed, as in \eqref{proof.prop.general.low.3}, we use the formula for $f\in \mathcal{P}_n L^2_\sigma (\Omega)$,
\begin{align}\label{proof.prop.general.high.4}
\begin{split}
& v_{{\rm high}} (t)  = e^{-\nu t \mathbb{A}} \mathcal{P}_n f \\
&- \int_0^t e^{-\nu (t-s) \mathbb{A}} \mathcal{P}_n \mathbb{P} \bigg ( i n \big ( U^E + U^{BL} (s,\frac{\cdot}{\sqrt{\nu}}) \big )  v_{\rm high} (s) +  v_{{\rm high},2} \partial_y \big ( U^E + U^{BL} (s, \frac{\cdot}{\sqrt{\nu}}) \big ) {\bf e}_1 \bigg ) \dd s
\end{split}
\end{align}
It is not difficult to show 
\begin{align*}
\| \nabla e^{-\nu t \mathbb{A}} \mathcal{P}_n f\|_{L^2(\Omega)} \leq \frac{C}{\nu^\frac12 t^\frac12} e^{-\frac12 \nu  n^2 t} \|  f\|_{L^2 (\Omega)}\,,
\end{align*}
which implies, in virtue of the computation as in \eqref{proof.prop.general.low.1}, 
\begin{align}\label{proof.prop.general.high.5}
\begin{split}
\| \nabla  v_{{\rm high}} (t) \|_{L^2 (\Omega)} & \leq  \frac{C}{\nu^\frac12 t^\frac12} e^{-\frac12 \nu n^2 t} \|  f\|_{L^2 (\Omega)} \\
& \quad + CC_2 (t)  |n|  \int_0^t \frac{1}{\nu^\frac12 (t-s)^\frac12} e^{-\frac12 \nu  n^2 (t-s)} \| v_{{\rm high}} (s) \|_{L^2 (\Omega)} \dd s\,.
\end{split}
\end{align}
Since $\| v_{{\rm high}} (s) \|_{L^2 (\Omega)} \leq e^{-\frac14 \nu n^2 s}\|  f \|_{L^2 (\Omega)}$ holds, we have from \eqref{proof.prop.general.high.5},
\begin{align}\label{proof.prop.general.high.6}
\| \nabla  v_{{\rm high}} (t) \|_{L^2 (\Omega)} & \leq  \frac{C}{\nu^\frac12 t^\frac12} e^{-\frac12 \nu n^2 t} \|  f\|_{L^2 (\Omega)}  + \frac{CC_2 (t)  |n|}{\nu^\frac12} e^{-\frac12 \nu n^2 t}   \int_0^t \frac{e^{\frac14 \nu n^2 s}}{(t-s)^\frac12} \dd s \|  f \|_{L^2 (\Omega)} \nonumber \\
& \leq \frac{C}{\nu^\frac12 t^\frac12} e^{-\frac12 \nu n^2 t} \| f\|_{L^2 (\Omega)}  + \frac{C|n|t^\frac12 C_2 (t)}{\nu^\frac12} e^{-\frac14 \nu n^2 t} \| f \|_{L^2 (\Omega)}\,,
\end{align}
which gives \eqref{est.prop.general.high.2}. The proof is complete.
\end{proof}

\section{Linear evolution operator in middle frequency: the case of time-independent shear flow}\label{sec.middle.notime}

In the middle range of frequency $\mathcal{O}(1) \leq |n| \leq \mathcal{O} (\nu^{-\frac34})$ 
the simple energy method as in the previous section does not provide useful estimates anymore,
and we need more sophisticated analysis by taking into account the structure of the shear flow.
Since the time-dependence of the shear flow could make the problem complicated,
we first focus on the case when the shear flow is independent of the time variable.
That is, instead of $\mathbb{A}_\nu (t)$ in \eqref{def.A_t}, we consider the simplified operator 
\begin{align}\label{def.A}
\mathbb{A}_\nu  u \, = \, -\nu \mathbb{P} \Delta u + \mathbb{P} \bigg ( V (\frac{y}{\sqrt{\nu}}) \partial_x u + u_2 \partial_y (V (\frac{y}{\sqrt{\nu}}) ) {\bf e}_1 \bigg )\,,
\end{align}
with a given shear flow profile $V=V(Y)$, which is assumed to have the following form by taking into account \eqref{sec2.condition}, \eqref{sec2.U^{BL}}, and \eqref{def.A_t}:
\begin{align}
V (Y) \, = \, U^E (\sqrt{\nu} Y) - U^E (0) + U (Y)\,, \qquad U(0) \, =\, 0\,, \quad  \lim_{Y\rightarrow \infty} U(Y) \, =\, U^E(0)\,. \label{shear.appl}
\end{align}
Here $U$ is a given function satisfying the boundary condition as in \eqref{shear.appl},
and further conditions will be mentioned later.

As in \eqref{decompose.A}, the operator $\mathbb{A}_\nu$ is diagonalized as the sum of the restrictions on each Fourier mode $n$ with respect to the $x$ variable,  denoted by $\mathbb{A}_{\nu,n}$.
Our aim is to estimate the associated semigroup $e^{-t \mathbb{A}_{\nu,n}}$ with $\delta_0^{-1}\leq |n|\leq \delta_0^{-1} \nu^{-\frac34}$, since the cases $|n|\leq \delta_0^{-1}$ and $|n|\geq \delta_0^{-1} \nu^{-\frac34}$ are already analyzed in Propositions \ref{prop.general.low} and \ref{prop.general.high}.
For this purpose it is convenient to introduce the rescaled velocity 
\begin{align}\label{def.scaling}
u(t,x,y) \, = \, v (\tau, X, Y) \,, \qquad (\tau,X,Y) \, =\, (\frac{t}{\sqrt{\nu}}, \frac{x}{\sqrt{\nu}}, \frac{y}{\sqrt{\nu}})\,.
\end{align}
If $u(t) = e^{-t\mathbb{A}_\nu} a$ then $v$ is the solution to 
\begin{equation}\label{eq.linear.scale}
\left\{
\begin{aligned}
\partial_\tau v  + V \partial_X v + v_2 \partial_Y V {\bf e}_1- \sqrt{\nu} \Delta_{X,Y} v + \nabla_{X,Y} q & \, = \, 0\,, \qquad \tau>0\,, X\in \R\,, Y>0\,,\\
{\rm div}_{X,Y}\, v & \, = \, 0\,,  \qquad \tau\geq 0\,, X\in \R\,, Y>0\,,\\
v|_{Y=0} \, = \, 0 \,, \quad v|_{\tau=0} & \, = \, a^{(\nu)} \,.
\end{aligned}\right.
\end{equation}
Note that the function $v$ is $\frac{2\pi}{\sqrt{\nu}}$-periodic in $X$, and $a^{(\nu)}(X,Y) = a (\nu^\frac12 X, \nu^\frac12 Y)$.
For simplicity of notations we will often omit the symbols $X,Y$ in the differential operators
and write, for example, $\Delta$ instead of $\Delta_{X,Y}$.
To obtain the estimate of solutions to \eqref{eq.linear.scale} in large time 
we study the associated resolvent problem for the operator 
\begin{align}
\begin{split}
\mathbb{L}_\nu v & \, = \, - \sqrt{\nu} \mathbb{P}_\nu \Delta v + \mathbb{P}_\nu \big ( V \partial_X v + v_2 \partial_Y V {\bf e}_1 \big )\,, \qquad v\in D(\mathbb{L}_\nu)\,,\\
D(\mathbb{L}_\nu) & \, = \,  W^{2,2} (\Omega_\nu)^2 \cap W^{1,2}_0 (\Omega_\nu)^2 \cap L^2_\sigma (\Omega_\nu)\,.
\end{split}
\end{align}
Here $\Omega_\nu =(\nu^{-\frac12}\mathbb{T} ) \times \R_+$, $\Delta=\partial_X^2+\partial_Y^2$, 
and $\mathbb{P}_\nu: L^2 (\Omega_\nu)^2 \rightarrow L^2_\sigma (\Omega_\nu)$ is the Helmholtz-Leray projection (we will write in the abbreviated style $\mathbb{P}$ if there is no risk of confusion).
For each fixed $\nu>0$ the Stokes operator $-\sqrt{\nu}\mathbb{P}\Delta$ defines a nonnegative self-adjoint operator in $L^2_\sigma(\Omega_\nu)$ and generates a bounded $C_0$-analytic semigroup acting on $L^2_\sigma (\Omega_\nu)$. Since the function space considered here is periodic in $X$ and $V$ is assumed to be bounded, the general perturbation theory of sectorial operators implies that $-\mathbb{L}_\nu$ also generates 
a $C_0$-analytic semigroup acting on $L^2_\sigma (\Omega_\nu)$. 
When $\displaystyle \lim_{Y\rightarrow \infty} V =0$ the perturbation term $\mathbb{P} \big ( V \partial_X v + v_2 \partial_Y V {\bf e}_1 \big )$ of $\mathbb{L}_\nu$ becomes a relatively compact perturbation of the Stokes operator $-\sqrt{\nu} \mathbb{P}\Delta$, and 
then it is not difficult to show that the difference between the spectrum of $\mathbb{L}_\nu$ and $-\sqrt{\nu}\mathbb{P}\Delta$ consists only of discrete eigenvalues of $\mathbb{L}_\nu$ with finite algebraic multiplicities.
In particular, since the spectrum of $\sqrt{\nu} \mathbb{P}\Delta$ is always included in the nonpositive real axis, we have 
\begin{align*}
\sigma (-\mathbb{L}_\nu) \subset \overline{\R_-} \cup \sigma_{disc} (-\mathbb{L}_\nu)\,,
\end{align*}
where $\sigma (-\mathbb{L}_\nu)$ denotes the spectrum of $-\mathbb{L}_\nu$, $\overline{\R_-} =\{\mu \leq 0\}$, and $\sigma_{disc}(-\mathbb{L}_\nu)$ denotes the set of discrete eigenvlaues of $-\mathbb{L}_\nu$ with finite algebraic multiplicities. 
As an important consequence, the study of the nontrivial spectrum of $-\mathbb{L}_\nu$ is reduced to the search of the eigenvalues in this case.
Our argument in this paper, however, works without any spatial decay of $V$ itself.
Indeed, the a priori estimates of solutions to the resolvent problem in this section do not depend on the spatial decay of $V$.

By the spectral mapping theorem for sectorial operators, the growth bound for the semigroup $e^{-\tau \mathbb{L}_\nu}$ coincides with the spectral bound of $-\mathbb{L}_\nu$, however, the information on the distribution of the spectrum is not enough for our purpose: in order to solve the nonlinear problem we need to  establish the estimates of the operator norm of  $e^{-\tau \mathbb{L}_\nu}$ in such a way as the dependence on $\nu$ is explicit.

To fix the idea, as a natural counter part of \eqref{def.projection.P}, we introduce the projections
\begin{align}\label{def.projection.P.nu}
\begin{split}
\big ( \mathcal{P}_n^{(\nu)}f \big ) (Y)  & \, = \, f_n (Y) e^{in \sqrt{\nu} X} \,, \qquad f_n  (Y)  \, = \, \frac{\sqrt{\nu}}{2\pi} \int_0^{\frac{2\pi}{\sqrt{\nu}}} f (X, Y) e^{-i n \sqrt{\nu} X} \dd X \,, \qquad n \in \mathbb{Z}\,,
\end{split}
\end{align} 
and for $m\in \N$, 
\begin{align*}
\mathcal{Q}_{m, {\rm low}}^{(\nu)} f \, = \, \sum_{|n|\leq m-1} \mathcal{P}_n^{(\nu)} f \,, \qquad \qquad 
\mathcal{Q}_{m, {\rm high}}^{(\nu)} f \, = \, \sum_{|n|\geq m+1} \mathcal{P}_n^{(\nu)} f\,.
\end{align*} 
By the same observation as in the previous section the operator $\mathbb{L}_\nu$ is diagonalized as follows.
For each $m_1,m_2\in \N$ with $m_1\leq m_2$, we have
\begin{align}\label{decompose.L^2.nu}
L^2_\sigma (\Omega_\nu) \, = \, \mathcal{Q}_{m_1, {\rm low}}  L^2_{\sigma} (\Omega_\nu) 
\, \oplus \, \mathcal{Q}_{m_2, {\rm high}}  L^2_{\sigma} (\Omega_\nu) 
\, \oplus \,   \big ( \oplus_{m_1\leq |n|\leq m_2} \mathcal{P}_n L^2_\sigma (\Omega_\nu) \big )\,.
\end{align}
and
\begin{align}\label{decompose.L}
\begin{split}
\mathbb{L}_\nu & \, = \, \mathbb{L}_{\nu, m_1,{\rm low}} \, \oplus \, \mathbb{L}_{\nu, m_2,{\rm high}} \, \oplus \, \big ( \oplus_{m_1\leq |n|\leq m_2} \mathbb{L}_{\nu,n} \big ) \,,\\
e^{- \tau \mathbb{L}_\nu} & \, = \, e^{- \tau \mathbb{L}_{\nu, m_1,{\rm low}} } 
\, \oplus \, e^{- \tau \mathbb{L}_{\nu, m_2,{\rm high}}} 
\, \oplus \big ( \oplus_{m_1\leq |n|\leq m_2} e^{- \tau \mathbb{L}_{\nu,n} } \big )\,.
\end{split}
\end{align}

To obtain the estimate of the original semigroup $e^{-t\mathbb{A}_{\nu,n}}$ for $\delta^{-1}\leq |n|\leq \delta_0^{-1} \nu^{-\frac34}$ it suffices to study the semigroup $e^{-t \mathbb{L}_{\nu,n}}$ for the same regime of $n$. Since the semigroup generated by a sectorial operator can  be expressed in terms of the resolvent using the Dunford integral we study the resolvent problem for the operator $\mathbb{L}_{\nu,n}$ in the next section. 
Due to the rescaling \eqref{def.scaling} the difficulty appears in the analysis in large time, or equivalently, the analysis for the resolvent parameter near the imaginary axis. On the other hand, there is no difficulty in obtaining short-time estimates of $e^{-\tau \mathbb{L}_{\nu,n}}$. Indeed, by taking into account the rescaling, Proposition \ref{prop.general.low} implies the following estimate.

\begin{prop}\label{prop.general.sg.short} It follows that
\begin{align}
\| e^{-\tau \mathbb{L}_{\nu,n}} f \|_{L^2(\Omega_\nu)} & \leq e^{C_1 (n)\nu^\frac12 \tau } \| f\|_{L^2(\Omega_\nu)}\,, \qquad \tau>0\,, \qquad f\in \mathcal{P}_n L^2_\sigma (\Omega_\nu)\,,
\end{align}
where $C_1 (n) = \| \partial_y U^E \|_{L^\infty(\R_+)} + 2 |n| \| Y \partial_Y U \|_{L^\infty(\R_+)}$.
\end{prop}

The proof of Proposition \ref{prop.general.sg.short} is straightforward from Proposition \ref{prop.general.low}, so we omit the details.

\begin{rem}{\rm The estimate in Proposition \ref{prop.general.sg.short} is not useful in large time, and one needs to obtain a better growth exponent than $C|n| \nu^\frac12 \tau$ in order to achieve the inviscid limit in a Gevrey class $s$ with $s>1$. The estimate in Proposition \ref{prop.general.sg.short} will be used only in a short time period, e.g., $0<\tau\leq \nu^{-\frac12} |n|^{-1}$.
}
\end{rem}

\subsection{Resolvent problem and Orr-Sommerfeld equations}\label{subsec.os}

In this subsection we consider the resolvent problem associated to the rescaled equations \eqref{eq.linear.scale}, i.e., 
\begin{equation}\label{eq.resolvent}
\left\{
\begin{aligned}
\mu v + V \partial_X v + v_2 \partial_Y V {\bf e}_1 - \sqrt{\nu} \Delta v + \nabla q & \, = \, f\,, \quad {\rm div}\, v \, = \, 0\,, \qquad   Y>0\,,\\
v & \, = \, 0\,, \qquad Y=0\,.
\end{aligned}\right.
\end{equation}
Here $\mu\in \C$ is a resolvent parameter, and $v$, $\nabla q$, and $f$ are assumed to be $\frac{2\pi}{\sqrt{\nu}}$-periodic in $X$. Note that $V$ is independent of $X$.
The vorticity field of $v$ is defined as $\omega=\partial_X v_2-\partial_Y v_1$, and we consider the stream function $\psi$, which is the solution to the Poisson equations $-\Delta\psi =\omega$. The direct computation leads to the equations for $\psi$ as follows.
\begin{align*}
\mu \Delta \psi + V \partial_X \Delta \psi - \partial_X \psi \partial_Y^2 V - \sqrt{\nu} \Delta^2 \psi \, = \, - \partial_X f_2 + \partial_Y f_1\,.
\end{align*}  
Let $\phi=\phi_n$, $n\in \mathbb{Z}\setminus\{0\}$, be the $n\sqrt{\nu}$ Fourier mode of $\psi$, i.e., 
$\phi = e^{-in\sqrt{\nu} X} \mathcal{P}_n^{(\nu)} \psi$. 
The equation for $\phi$ is then written as 
\begin{equation*}
\left\{
\begin{aligned}
\frac{i}{n} (\partial_Y^2 - n^2 \nu )^2 \phi + (V + \frac{\mu}{i n \sqrt{\nu}}) (\partial_Y^2- n^2 \nu ) \phi - (\partial_Y^2 V ) \phi & \, = \, - f_{2,n} + \frac{1}{i n \sqrt{\nu}} \partial_Y f_{1,n}\,, \quad Y>0\,,\\
\phi \, = \, \partial_Y \phi & \, = \, 0\,, \quad  Y=0\,.
\end{aligned}\right.
\end{equation*}
Note that the $n\sqrt{\nu}$ Fourier modes of $v$ and $\omega$, $\mathcal{P}_n^{(\nu)} v$ and $\mathcal{P}_n^{(\nu)}\omega$, are recovered from $\phi$ through the formula 
\begin{align}\label{formula.stream}
\mathcal{P}_n^{(\nu)} v  \, = \, (\partial_Y \phi \, e^{in\sqrt{\nu} X} , \, -in\sqrt{\nu} \phi \, e^{in\sqrt{\nu} X} )\,, \qquad \mathcal{P}_n^{(\nu)} \omega \, = \, (n^2\nu-\partial_Y^2) \phi \, e^{in\sqrt{\nu} X} \,.
\end{align}
Thus, the estimates for $\mathcal{P}_n^{(\nu)} v$ and $\mathcal{P}_n^{(\nu)}\omega$ are obtained from the analysis of the ordinary differential equations for $\phi$ as above. 
In virtue of Propositions \ref{prop.general.low} and \ref{prop.general.high}
it suffices to consider the case 
\begin{align*}
\frac{1}{\delta_0} \leq |n|\leq \frac{1}{\delta_0 \nu^{\frac34}}\,,
\end{align*}
where $\delta_0\in (0,1)$ is the number in Proposition \ref{prop.general.high}.
More precisely, since $V$ is now time-independent it is taken as
\begin{align}
\delta_0 = \frac{1}{2 \big (1+\| \partial_y U^E\|_{L^\infty} + \| \partial_Y U\|_{L^\infty} \big )^{\frac12}}\,.\label{def.delta_0}
\end{align}
Note that we may assume without loss of generality that $n$ is positive
 (when $n$ is negative it suffices to consider the complex conjugate of the first equation).
Then we set for $n\in \N$,
\begin{align}\label{def.coefficient}
\alpha \, = \, n \sqrt{\nu}\,, \qquad \epsilon \, = \,  - \frac{i}{n}\,, \qquad c =  \frac{i \mu}{\alpha}\,, \qquad h \, = \, - f_{2,n} + \frac{1}{i\alpha} \partial_Y f_{1,n}\,.
\end{align}
With these notation we obtain the Orr-Sommerfeld equations
\begin{equation}\label{eq.os}
\left\{
\begin{aligned}
-\epsilon  (\partial_Y^2 - \alpha^2 )^2 \phi + (V -c) (\partial_Y^2-\alpha^2) \phi - (\partial_Y^2 V ) \phi  & \, = \, h \,, \qquad Y>0\,,\\
\phi \, = \, \partial_Y \phi  & \, = \, 0 \,, \qquad  Y=0\,.
\end{aligned}\right.
\end{equation}
Since $\delta_0^{-1}\leq n\leq \delta_0^{-1}\nu^{-\frac34}$, the regime of parameters we need to study is 
\begin{align*}
\frac{\nu^\frac12}{\delta_0} \leq \alpha\leq \frac{1}{\delta_0  \nu^\frac14} \,, \qquad \qquad \delta_0\nu^\frac34 \leq |\epsilon| = \frac{1}{n} \leq \delta_0 \,.
\end{align*}
Moreover, we consider the case  
\begin{align}\label{parameter.stability}
\Re \mu = \alpha \Im c \geq  \frac{\nu^\frac12n^\gamma}{\delta} 
\end{align}
for some $\gamma \in [0,1]$ and for sufficiently small but fixed $\delta \in (0,\delta_1]$,
where the number $\delta_1\in (0,\delta_0]$ will be determined later in Proposition \ref{prop.general'}.

\begin{rem}\label{rem.parameter} {\rm (i)  The value $\gamma\in [0,1]$ in \eqref{parameter.stability} has a special importance in our analysis. Indeed, it represents the order of derivative loss through the action of the semigroup $e^{-t\mathbb{A}_{\nu,n}}$. In order to achieve the nonlinear stability in the Gevrey class $s$ with  $s>1$ we need to show the resolvent estimate for the value $\gamma$ strictly less than $1$. 
In \cite{GGN2014} the unstable eigenvalue is found in the parameter regime
\begin{align}
\alpha\sim \mathcal{O}(\nu^\frac18)\,, \quad \Im c \sim \mathcal{O}(\nu^\frac18)\,, \quad n=|\epsilon|^{-1} \sim \mathcal{O} (\nu^{-\frac38})\,.
\end{align}
This implies that we can actually expect at most $\gamma=\frac23$ for the spatial frequency $n\sim \mathcal{O}(\nu^{-\frac38})$.

\noindent (ii) Since $\alpha = n \sqrt{\nu}$ the condition \eqref{parameter.stability} gives the lower bound
\begin{align}
\Im c \geq \frac{n^{\gamma-1}}{\delta}\,.\label{bound.nc}
\end{align}
In particular, we have 
\begin{align}
\Im c \geq \frac{1}{\delta}\gg 1\,, \qquad ~~ {\rm if}  ~~ \gamma=1~{\rm in}~\eqref{parameter.stability}\,.
\end{align}
}
\end{rem}

The Orr-Sommerfeld equations \eqref{eq.os} have to be analyzed in a wide regime of parameters as mentioned above. In particular, the value $\alpha$ can be both small and large depending on the location of the frequency $n$. To reduce the possible dependence on $\alpha$ of the equations it is convenient to introduce the number 
\begin{align}\label{def.c.ep}
c_\epsilon = c - \epsilon \alpha^2\,.
\end{align} 
Note that 
\begin{align}
\Im  c_\epsilon \, = \, \Im c + \frac{\alpha^2}{n}\,,\label{Im.c_ep}
\end{align} 
which is positive if $\Im c$ is positive. One can easily check that \eqref{eq.os} is equivalent with the modified form of the Orr-Sommerfeld equations as follows.
\begin{equation}\label{eq.os'}
\left\{
\begin{aligned}
-\epsilon  (\partial_Y^2 - \alpha^2 ) \partial_Y^2 \phi + (V -  c_\epsilon ) (\partial_Y^2-\alpha^2) \phi - (\partial_Y^2 V ) \phi  & \, = \, h \,, \qquad Y>0\,,\\
\phi \, = \, \partial_Y \phi  & \, = \, 0 \,, \qquad  Y=0\,.
\end{aligned}\right.
\end{equation}

\begin{rem}\label{rem.c_epsilon}{\rm  Since $\Im c\geq \frac{n^{\gamma-1}}{\delta}$ by \eqref{bound.nc} and $\frac{\alpha^2}{n} = n \nu$ by the definition of $\alpha$, we have 
\begin{align}
\frac{\alpha^2}{n}\leq \Im c \qquad  {\rm if}  \quad n^{2-\gamma}\leq \frac{1}{\delta\nu}\,.
\end{align}
When $\gamma\in [\frac23,1]$ this condition is satisfied if $n\leq \delta_0^{-1}\nu^{-\frac34}$.
}
\end{rem}

\subsection{Resolvent estimate for general shear profile}\label{subsec.general}

We first consider the case of general shear profiles for some regime of parameters.
In this subsection the shear profile $V$ is always assumed to take the form \eqref{shear.appl} with $U^E\in BC^2 (\R_+)$, and we also assume that 
\begin{align}\label{bound.U}
\begin{split}
& \|U\|\, := \, \sum_{k=0,1,2} \sup_{Y \geq 0}  \, (1+Y)^k |\pa_Y^k U (Y) | \, <\, \infty\,.
\end{split}
\end{align}

\vspace{0.3cm}

Let $\Im c>0$, which ensure $\Im c_\epsilon>0$ due to \eqref{Im.c_ep}. 
We consider the modified form \eqref{eq.os'} of the Orr-Sommerfeld equations rather than \eqref{eq.os}. 
The argument below is based on Rayleigh's trick.
Multiplying both sides of the first equation of \eqref{eq.os'} by $(V- c_\epsilon)^{-1} \bar{\phi}$ 
and integrating over $(0,\infty)$, we have
\begin{align}\label{proof.prop.general.1}
\begin{split}
& -\epsilon \int_0^\infty   \partial_Y^2\phi \, \partial_Y^2 \big ( \frac{\bar{\phi}}{V- c_\epsilon} \big ) + \alpha^2 \partial_Y \phi \, \partial_Y \big (\frac{\bar{\phi}}{V- c_\epsilon} \big ) \dd Y \\
& \qquad - \int_0^\infty |\partial_Y \phi|^2 +\alpha^2 |\phi|^2 \dd Y -\int_0^\infty \frac{\partial_Y^2 V}{V-c_\epsilon} |\phi|^2 \dd Y  \, = \, \int_0^\infty \frac{h}{V-c_\epsilon} \bar{\phi} \dd Y\,.
\end{split}
\end{align}
The first integration in the right-hand side of \eqref{proof.prop.general.1} is computed as 
\begin{align*}
-\epsilon \int_0^\infty \frac{1}{V- c_\epsilon}  \big ( |\partial_Y^2 \phi|^2 + \alpha^2  |\partial_Y \phi |^2 \big ) \dd Y  \, - \,\epsilon R\,,
\end{align*}
where 
\begin{align*}
R =  \int_0^\infty \partial_Y^2 \phi \bigg ( 2 \partial_Y \big (\frac{1}{V-c_\epsilon} \big ) \partial_Y \bar{\phi} + \bar{\phi} \partial_Y^2 \big ( \frac{1}{V-c_\epsilon})  \bigg )  + \alpha^2 (\partial_Y \phi ) \, \bar{\phi} \partial_Y\big ( \frac{1}{V-c_\epsilon}\big ) \dd Y \,.
\end{align*}
Recalling $\epsilon = - \frac{i}{n}$ and taking the real part of both sides of \eqref{proof.prop.general.1},
we obtain
\begin{align}\label{proof.prop.general.2}
\begin{split}
& -\frac{\Im c_\epsilon}{n} \int_0^\infty \frac{1}{|V- c_\epsilon|^2}  \big ( |\partial_Y^2 \phi|^2 + \alpha^2  |\partial_Y \phi |^2 \big ) \dd Y  - \int_0^\infty |\partial_Y \phi|^2 +\alpha^2 |\phi|^2 \dd Y   \\
& \qquad \qquad  - \frac{\Im R}{n} - \int_0^\infty \frac{\partial_Y^2 V}{|V-c_\epsilon|^2} (V-\Re c_\epsilon) |\phi|^2 \dd Y  \, = \, \Re \int_0^\infty  \frac{h}{V-c_\epsilon} \bar{\phi} \dd Y\,.
\end{split}
\end{align}
It is straightforward to see
\begin{align*}
|\Im R| & \leq \frac{\Im c_\epsilon}{2} \int_0^\infty \frac{1}{|V- c_\epsilon|^2}  \big ( |\partial_Y^2 \phi|^2 + \alpha^2  |\partial_Y \phi |^2 \big ) \dd Y\\
& \quad + \frac{1}{2 \Im c_\epsilon} \int_0^\infty |V-c_\epsilon|^2 \bigg ( \big | 2 \partial_Y \big (\frac{1}{V-c_\epsilon} \big ) \partial_Y \bar{\phi} + \bar{\phi} \partial_Y^2 \big ( \frac{1}{V-c_\epsilon})  \big | ^2 + \alpha^2 \big | \partial_Y \frac{1}{V-c_\epsilon}\big |^2 |\phi|^2 \bigg ) \dd Y\\
& \leq \frac{\Im c_\epsilon}{2} \int_0^\infty \frac{1}{|V- c_\epsilon|^2}  \big ( |\partial_Y^2 \phi|^2 + \alpha^2  |\partial_Y \phi |^2 \big ) \dd Y\\
& \quad + \frac{1}{2 \Im c_\epsilon} \int_0^\infty \frac{12 |\pa_Y V|^2}{|V-c_\epsilon|^2} |\partial_Y \phi|^2 + \big ( \frac{3|\pa_Y^2 V|^2}{|V-c_\epsilon|^2} + \frac{12 |\pa_Y V|^4}{|V-c_\epsilon|^4}  + \alpha^2 \frac{|\pa_Y V|^2}{|V-c_\epsilon|^2} \big ) |\phi|^2  \dd Y\,.
\end{align*}
Here $C$ is a universal constant. Hence we have from \eqref{proof.prop.general.2},
\begin{align}\label{proof.prop.general.3}
\begin{split}
& \frac{\Im c_\epsilon}{2 n} \int_0^\infty \frac{1}{|V- c_\epsilon|^2}  \big ( |\partial_Y^2 \phi|^2 + \alpha^2  |\partial_Y \phi |^2 \big ) \dd Y  +  \int_0^\infty |\partial_Y \phi|^2 +\alpha^2 |\phi|^2 \dd Y   \\
& \leq  \frac{6}{n \Im c_\epsilon} \int_0^\infty \frac{|\pa_Y V|^2}{|V-c_\epsilon|^2} |\partial_Y \phi|^2 + \big ( \frac{|\pa_Y^2 V|^2}{|V-c_\epsilon|^2} + \frac{|\pa_Y V|^4}{|V-c_\epsilon|^4}  + \alpha^2 \frac{|\pa_Y V|^2}{|V-c_\epsilon|^2} \big ) |\phi|^2  \dd Y\\
& \qquad  \quad  +  \int_0^\infty \frac{|\pa_Y^2 V|}{|V-c_\epsilon|}  |\phi|^2 \dd Y   - \Re \int_0^\infty  \frac{h}{V-c_\epsilon} \bar{\phi} \dd Y\,.
\end{split}
\end{align}

We note that the inequality \eqref{proof.prop.general.3} is valid for any weak solution $\phi\in H^2_0 (\R_+)$ to \eqref{eq.os'} (defined in a natural manner) since $\Im c_\epsilon>0$.
For some regime of parameters the inequality \eqref{proof.prop.general.3} provides enough informations on the estimates of solutions. Indeed, the next proposition shows that the resolvent estimate is always available if $|c|$ is large enough depending on $V$. Let us recall that the inhomogeneous term $h$ is related with $f$ as in \eqref{def.coefficient}, and $\delta_0\in (0,1)$ is chosen as in \eqref{def.delta_0}. The weighted norm $\|U\|$ is defined in \eqref{bound.U}.

\begin{prop}\label{prop.general'} There exists $\delta_1\in (0,\delta_0]$ such that the following statement holds.
Let $|c|\geq \delta_1^{-1}$ and $n\in \N$. Assume that \eqref{parameter.stability} holds for some $\gamma\in [0,1]$ and $\delta\in (0,\delta_1]$.
Then for any $f=(f_{1,n},f_{2,n})\in L^2 (\R_+)^2$ there exists a unique weak solution $\phi\in H_0^2(\R_+)$ to \eqref{eq.os'}, and $\phi$ satisfies 
\begin{align}
\| \partial_Y \phi \|_{L^2} + \alpha \| \phi \|_{L^2} & \leq \frac{C (1+\|U\|)}{|\mu|}  \|f \|_{L^2}\,,\label{est.prop.general'.1} \\
\| (\partial_Y^2 -\alpha^2 ) \phi \|_{L^2} & \leq \frac{C (1+\|U\|) }{|\mu|^\frac12 \nu^\frac14}  \| f\|_{L^2}\,.\label{est.prop.general'.2} 
\end{align}
Here $\delta_1$ depends only on $\|U^E \|_{C^2}$ and $\|U\|$, while $C\geq 1$ is a universal constant.
\end{prop}

\begin{rem}\label{rem.prop.general'}{\rm (i) By the standard elliptic regularity the solution $\phi$ also belongs to $H^3(\R_+)$.

\noindent 
(ii) Proposition \ref{prop.general'} is valid also for the regime $n\geq \delta_0^{-1}\nu^{-\frac34}$.

\noindent (iii) Proposition \ref{prop.general'} shows that the case $\gamma=1$ in \eqref{parameter.stability} can be handled by the simple energy method, and there are no unstable eigenvalues of $-\mathbb{L}_{\nu,n}$ such that  $\Re \mu \geq  \mathcal{O}(n \nu^\frac12)$ in general. 
On the other hand, when there is an unstable eigenvalue for the Rayleigh equations with some $\alpha>0$
it is likely that the Orr-Sommerfeld equations also possess an unstable eigenvalue
in the parameter regime $\alpha = \mathcal{O}(1)$, $\Im c =\mathcal{O}(1)$, and $|\epsilon| \ll 1$.
From this point of view, by considering the case $n=\mathcal{O}(\nu^{-\frac12})$,
the above elementary energy method actually provides a fairly optimal result
for unstable shear flows satisfying \eqref{bound.U}.
}
\end{rem}

\begin{proofx}{Proposition \ref{prop.general'}} (A priori estimates) Assume that $|c|\geq \delta_1^{-1}$, 
where $\delta_1\in (0,\delta_0]$ is determined below.
Since $|\mu| = \alpha |c|$ and $\Im c_\epsilon = \Im c + \frac{\alpha^2}{n} \geq \Im c > 0$,
we have 
\begin{align*}
|c_\epsilon|\geq |c|\geq \delta_1^{-1}\,.
\end{align*} 
If $\delta_1$ is taken as
\begin{align}\label{proof.prop.general'.-1}
\delta_1 \, = \, \frac{1}{32 ( 1 + \| U^E \|_{C^2}  + \|U\| )}\,,
\end{align} 
where  $\|U\|$ is defined in \eqref{bound.U},
then $\delta_1\in (0,\delta_0]$ and  $\frac{|c|}{2} \leq |V-c_\epsilon|\leq 2 |c|$ holds.
Since the lower bound $n\Im c_\epsilon \gg 1$ is valid as stated in Remark \ref{rem.parameter},
it is not difficult to see that the right-hand side of \eqref{proof.prop.general.3} is bounded from above by 
\begin{align*}
\frac12 \int_0^\infty |\partial_Y \phi|^2 +\alpha^2 |\phi|^2 \dd Y - \Re \int_0^\infty  \frac{h}{V-c_\epsilon} \bar{\phi} \dd Y\,.
\end{align*}
For example, the term $\frac{1}{n\Im c_\epsilon} \int_0^\infty \frac{|\pa_Y V|^4}{|V-c_\epsilon|^4} |\phi|^2 \dd Y$ is estimated as 
\begin{align*}
\frac{1}{n\Im c_\epsilon} \int_0^\infty \frac{|\pa_Y V|^4}{|V-c_\epsilon|^4} |\phi|^2 \dd Y 
& \leq \frac{2^5 \delta_1^4 \|\pa_Y V\|_{L^\infty}^2}{n\Im c_\epsilon}\int_0^\infty \big (\nu \| U^E\|_{C^1}^2 + \| U\|^2 (1+Y)^{-2} \big) |\phi|^2 \dd Y\\
& \ll \alpha^2\|\phi \|_{L^2}^2 + \| \partial_Y\phi\|_{L^2}^2\,.
\end{align*}
Here we have used $\alpha=n\nu^\frac12$ and also used  the Hardy inequality 
\begin{align*}
\| \frac{\phi}{Y} \|_{L^2} \leq 2 \| \partial_Y \phi \|_{L^2}\,.
\end{align*}
The other terms are estimated in the similar manner, and the details are omitted here.
By the definition of $h$ in \eqref{def.coefficient} and \eqref{bound.U}
the last term is estimated as 
\begin{align}\label{proof.prop.general'.0}
\begin{split}
& - \Re \int_0^\infty  \frac{h}{V-c_\epsilon} \bar{\phi} \dd Y  \leq \frac14 \int_0^\infty |\partial_Y \phi|^2 +\alpha^2 |\phi|^2 \dd Y  \\
& \qquad + C \bigg ( \| \frac{f}{\alpha (V-c_\epsilon)}\|_{L^2}^2 + \|U\|^2  \| \frac{f}{\alpha (V-c_\epsilon)^2}\|_{L^2}^2 +\nu  \| U^E \|_{C^1} ^2 \| \frac{f_{1,n}}{\alpha^2 (V-c_\epsilon)^2}\|_{L^2}^2 \bigg )\,
\end{split}
\end{align} 
where $C$ is a numerical constant. 
Note that the Hardy inequality is used again in the derivation of \eqref{proof.prop.general'.0}. 
From $\frac{|c|}{2} \leq |V-c_\epsilon|\leq 2 |c|$ and the identity $|\mu|=\alpha |c|$ we have arrived at
\begin{align}\label{proof.prop.general'.1}
\begin{split}
& \frac{\Im c_\epsilon}{2 n} \int_0^\infty \frac{1}{|V- c_\epsilon|^2}  \big ( |\partial_Y^2 \phi|^2 + \alpha^2  |\partial_Y \phi |^2 \big ) \dd Y  +  \frac14 \int_0^\infty |\partial_Y \phi|^2 +\alpha^2 |\phi|^2 \dd Y   \\
& \leq  C \big ( \frac{1+ \|U\|^2}{|\mu|^2}  + \frac{\nu \| U^E \|_{C^1}^2}{|\mu|^4} \big ) \| f \|_{L^2}^2
= \frac{C}{|\mu|^2} \big ( 1+ \|U\|^2 + \frac{\| U^E \|_{C^1}^2}{n^2|c|^2} \big ) \| f\|_{L^2}^2
\end{split}
\end{align}
with a numerical constant $C$. 
By the assumption $|c|\geq \delta_1^{-1}$ and the choice of $\delta_1$ in \eqref{proof.prop.general'.-1}
we have $\frac{\|U^E\|_{C^1}}{n |c|}\leq 1$, which yields \eqref{est.prop.general'.1}.

Next we multiply both sides of \eqref{eq.os} by $\bar{\phi}$ and integrate over $(0,\infty)$, which leads to 
\begin{align}\label{proof.prop.general'.3}
-\epsilon \| (\partial_Y^2 -\alpha^2) \phi \|_{L^2}^2 + \langle (V-c) (\partial_Y^2-\alpha^2) \phi, \phi\rangle _{L^2} - \langle (\pa_Y^2 V) \phi, \phi \rangle _{L^2} \, = \, \langle h, \phi \rangle _{L^2}\,.
\end{align}
Taking the imaginary part of \eqref{proof.prop.general'.3}, we have 
\begin{align}\label{proof.prop.general'.4}
\begin{split}
& \frac1n \| (\partial_Y^2 -\alpha^2) \phi \|_{L^2}^2  + (\Im c ) \big ( \|\partial_Y \phi \|_{L^2}^2 + \alpha^2 \| \phi \|_{L^2}^2 \big ) \\
& \quad \, = \, - \Im \langle (V-\Re c)  ( \partial_Y^2-\alpha^2 )\phi, \phi\rangle _{L^2} + \Im \langle h, \phi \rangle _{L^2}\\
& \quad \, = \, - \Im \langle V \pa_Y^2 \phi, \phi \rangle _{L^2} + \Im \langle h, \phi \rangle _{L^2}\,.
\end{split}
\end{align} 
It is straightforward to see from \eqref{est.prop.general'.1} and \eqref{bound.U},
\begin{align}
 |\Im \langle V  \partial_Y^2 \phi, \phi \rangle _{L^2} |
& \leq  \| \partial_Y\phi \|_{L^2}\| (\pa_Y V) \phi \|_{L^2} \nonumber \\
& \leq C \| \partial_Y\phi \|_{L^2} \big (  \nu^\frac12 \| U^E \|_{C^1} \|\phi\|_{L^2} + \|U\|  \| \partial_Y \phi \|_{L^2} \big ) \nonumber \\
& \leq \frac{C}{|\mu|} \big ( \frac{\nu^\frac12 \| U^E \|_{C^1} }{\alpha |\mu|} + \frac{\|U\|}{|\mu|} \big )  (1+\|U\|)^2 \| f\|_{L^2}^2\,.\label{proof.prop.general'.5}
\end{align}
Here we have also used the Hardy inequality and \eqref{est.prop.general'.1}.
By the definition of $h$ and \eqref{est.prop.general'.1} we also have 
\begin{align}
|\langle h,\phi \rangle_{L^2}|\leq \frac{C}{\alpha} \| f\|_{L^2} \big ( \| \partial_Y \phi \|_{L^2} + \alpha \| \phi \|_{L^2}\big ) \leq \frac{C}{\alpha |\mu|} (1+\|U\|) \|f\|_{L^2}^2\,.\label{proof.prop.general'.6}
\end{align}
From \eqref{proof.prop.general'.4}, \eqref{proof.prop.general'.5}, and \eqref{proof.prop.general'.6}
we obtain by using the relation $\mu = - i \alpha c = -i n \nu^\frac12 c$,
\begin{align*}
\| (\partial_Y^2 -\alpha^2) \phi \|_{L^2}^2  & \leq \frac{C n}{|\mu|} \big ( \frac{\nu^\frac12\| U^E \|_{C^1}}{\alpha |\mu|} + \frac{\|U\|}{|\mu|}  + \frac{1}{\alpha} \big ) \| f\|_{L^2}^2 \\
& \leq \frac{C}{|\mu|}  \big ( \frac{\| U^E \|_{C^1}}{|\mu|} + \frac{\|U\|}{\nu^\frac12 |c|} + \frac{1}{\nu^\frac12} \big ) (1+\|U\|)^2 \| f\|_{L^2}^2\\
& \leq \frac{C}{|\mu|\nu^\frac12} \big ( \frac{\| U^E \|_{C^1}+\|U\|}{|c|} + 1\big ) (1+\|U\|)^2 \| f \|_{L^2}^2\,.
\end{align*}
By the assumption $|c|\geq \delta_1^{-1}$ and the choice of $\delta_1$ in \eqref{proof.prop.general'.-1}, we obtain \eqref{est.prop.general'.2}.

\noindent (Uniqueness) The uniqueness of weak solutions in $H^2_0(\R_+)$ is immediate from the a priori estimates.

\noindent (Existence) We first observe that the operator $T=-\epsilon (\pa_Y^2-\alpha^2)^2 - i \Im c (\pa_Y^2-\alpha^2)$ with the domain $H^4 (\R_+) \cap H^2_0 (\R_+)$ is always invertible in $L^2 (\R_+)$ for $\Im c, \alpha>0$. If $\Im c \gg 1$ the operator $mOS=T + (V-\Re c) (\pa_Y^2-\alpha^2) - (\pa_Y^2 V)$ with the domain $H^4(\R_+) \cap H^2_0 (\R_+)$ is a small perturbation of $T$ and is invertible by considering Neumann series; indeed, we can write $mOS = \big (I +  (V-\Re c) (\pa_Y^2-\alpha^2) T^{-1} \big ) T$ and $W=I + (V-\Re c) (\pa_Y^2-\alpha^2)T^{-1}$ is bounded in $L^2 (\R_+)$ and invertible if $\Im c$ is large enough, for we can show the bound  
$$\|  (V-\Re c) (\pa_Y^2-\alpha^2)T^{-1} h \|_{L^2}\leq \frac{C(\|V\|_{L^\infty} + |\Re c| )}{\alpha (|\epsilon| \Im c)^{\frac12}} \| h \|_{L^2}$$
with a universal constant $C>0$, by a simple energy estimate. 
Hence the inverse of $mOS$ is given by $T^{-1} W^{-1}$ for sufficiently large $\Im c$. 
We may assume that $\Im c\gg \delta_1^{-1}$.
Then, for such $\Im c$, approximating $f\in L^2(\R_+)^2$ by a sequence in $W^{1,2}(\R_+)^2$,
we obtain the unique weak solution $\phi\in H^2_0 (\R_+)$ to \eqref{eq.os'} for any $f\in L^2 (\R_+)^2$. 
In particular, the solution $\phi$ satisfies the a priori estimates \eqref{est.prop.general'.1} and \eqref{est.prop.general'.2}. Then, in virtue of the method of continuity using the a priori estimates \eqref{est.prop.general'.1} and \eqref{est.prop.general'.2}, we obtain the weak solution $\phi$ to \eqref{eq.os'} for any value of $c$ satisfying $|c|\geq \delta_1^{-1}$. The proof is complete.
\end{proofx}

Proposition \ref{prop.general'} enables us to obtain the resolvent in the high temporal frequency,
which will be used in the latter section to obtain the estimate for the analytic semigroup generated by $-\mathbb{L}_{\nu,n}$.

\begin{cor}\label{cor.prop.general'} Let $\delta_1\in (0,1)$ be the number in Proposition \ref{prop.general'}. 
Then there exist $\theta\in (\frac{\pi}{2},\pi)$ such that the set 
\begin{align}\label{est.cor.prop.general'.1}
S_{\nu,n} (\theta) \, = \, \big \{ \mu \in \C ~\big |~ |\Im \mu |\geq (\tan \theta ) \Re \mu +  \delta_1^{-1} (\alpha + |\tan \theta| n^\gamma \nu^\frac12 )\,, \quad |\mu |\geq \delta_1^{-1} \alpha \big \}
\end{align}
is included in the resolvent set of $-\mathbb{L}_{\nu,n}$, and 
\begin{align}
\| (\mu + \mathbb{L}_{\nu,n} )^{-1} f \|_{L^2(\Omega_\nu)} & \leq \frac{C'}{|\mu|} \| f \|_{L^2 (\Omega_\nu)}\,,\label{est.cor.prop.general'.2}\\
\| \nabla (\mu + \mathbb{L}_{\nu,n} )^{-1} f \|_{L^2(\Omega_\nu)} & \leq \frac{C'}{\nu^\frac14 |\mu|^\frac12} \| f \|_{L^2 (\Omega_\nu)}\,,\label{est.cor.prop.general'.3}
\end{align}
for all $\mu \in S_{\nu,n}(\theta)$ and $f\in \mathcal{P}_n^{(\nu)} L^2_\sigma (\Omega_\nu)$.
The numbers $\theta$ and $C'$ depend only on $\|U^E\|_{C^1}$ and $\|U\|$ in \eqref{bound.U}.
\end{cor}

\begin{proof} Let $n\in \N$. We observe that, since $\mu = -i \alpha c$, the condition of Proposition \ref{prop.general'} is satisfied if $|\mu|\geq \delta_1^{-1} \alpha$ and $\Re \mu=\alpha \Im c\geq \delta_1^{-1} n^\gamma \nu^\frac12$. We first show that such $\mu$ belongs to the resolvent set of $-\mathbb{L}_{\nu,n}$ in $\mathcal{P}_n^{(\nu)} L^2_\sigma (\Omega_\nu)$. 
To this end take any $f=e^{i n \sqrt{\nu}X} (f_{1,n},f_{2,n}) \in \mathcal{P}_n^{(\nu)}L^2_\sigma (\Omega_\nu)$ (that is, $(f_{1,n},f_{2,n}) \in L^2 (\R_+)^2$ satisfies $in \sqrt{\nu} f_{1,n} + \pa_Y f_{2,n}=0$ in $Y>0$ and $f_{2,n}=0$ on $Y=0$).  By applying Proposition \ref{prop.general'}, let $\phi\in H^2_0(\R_+)$ be the unique weak solution to \eqref{eq.os'}, and thus, to \eqref{eq.os}, with $h=-f_{2,n} +\frac{1}{i\alpha} \pa_Y f_{1,n}$.
Note that the solution $\phi$ obtained in Proposition \ref{prop.general'} belongs to $H^3(\R_+)$ by the elliptic regularity. Then the velocity $v=(\pa_Y \phi e^{in\sqrt{\nu}X}, -in \sqrt{\nu} \phi e^{in \sqrt{\nu}X} )$ belongs to $D(\mathbb{L}_{\nu,n})=H^2 (\Omega_\nu)^2 \cap H^1_0(\Omega_\nu)^2 \cap \mathcal{P}_n^{(\nu)} L^2_\sigma (\Omega_\nu)$ and satisfies \eqref{eq.resolvent} for a suitable pressure $q$.
In particular, $\mu+\mathbb{L}_{\nu,n}$ is surjective in $\mathcal{P}_n^{(\nu)} L^2_\sigma (\Omega_\nu)$.
Moreover, in virtue of \eqref{est.prop.general'.1} and \eqref{est.prop.general'.2}, the norms $\|v\|_{L^2(\Omega_\nu)}$ and $\|\nabla v\|_{L^2(\Omega_\nu)}$ are estimated as in \eqref{est.cor.prop.general'.2} and \eqref{est.cor.prop.general'.3} by recalling the relations $|\epsilon|=n^{-1}$, $\alpha = n\nu^\frac12$, and $\mu = -i \alpha c$, and $\Im c_\epsilon=\Im c + n\nu$. 
Next we suppose that $v\in D(\mathbb{L}_{\nu,n})$ satisfies $(\mu + \mathbb{L}_{\nu,n}) v=0$ then for its stream function $\phi (Y) e^{in\sqrt{\nu}X}$, we see that $\phi$ solves the Orr-Sommerfeld equations \eqref{eq.os} with the source term $h=0$. By the uniqueness proved in Proposition \ref{prop.general'}, we have $\phi=0$, and thus, $\omega=\pa_X v_2-\pa_Y v_1=0$. Then $v$ is harmonic in $\Omega_\nu$ and vanishes on the boundary, which implies $v=0$ in $\Omega_\nu$. Therefore, $\mu + \mathbb{L}_{\nu,n}$ is also injective, and thus, we conclude that $\mu\in \rho (-\mathbb{L}_{\nu,n})$ in $\mathcal{P}_n^{(\nu)}L^2_\sigma (\Omega_\nu)$
if $|\mu|\geq \delta_1^{-1}\alpha$ and $\Re (\mu) \geq \delta_1^{-1} n^\gamma \nu^\frac12$. 
Recall again that, for such $\mu$,  \eqref{est.prop.general'.1} implies
\begin{align*}
\| (\mu + \mathbb{L}_{\nu,n} )^{-1} f \|_{L^2(\Omega_\nu)} \leq \frac{C(1+\|U\|)}{|\mu|} \| f \|_{L^2 (\Omega_\nu)}\,,
\qquad f\in \mathcal{P}_n^{(\nu)} L^2_\sigma (\Omega_\nu)\,,
\end{align*}
with the same  $C\geq 1$ as in \eqref{est.prop.general'.1}. 
By considering the Neumann series the ball $B_{r_\mu}(\mu)=\{ \lambda\in \C~|~ |\lambda-\mu|\leq r_\mu\}$ with $r_{\mu}=\frac{|\mu|}{2C(1+\|U\|)}$ belongs to the resolvent set of $-\mathbb{L}_{\nu,n}$, and we have 
\begin{align*}
\| (\lambda + \mathbb{L}_{\nu,n})^{-1} f \|_{L^2(\Omega_\nu)} \leq \frac{4C(1+\|U\|)}{|\mu|} \| f\|_{L^2(\Omega_\nu)} 
\leq \frac{8C(1+\|U\|)}{|\lambda|} \| f\|_{L^2(\Omega_\nu)}\,, \qquad \lambda \in B_{r_\mu} (\mu)\,. 
\end{align*}
This estimate corresponds to \eqref{est.cor.prop.general'.2}, and we will prove \eqref{est.cor.prop.general'.3} later. 
So far, we have shown that
\begin{align}\label{proof.cor.prop.general'.0}
\cup_{\mu\in E_{\nu,n}} B_{r_\mu} (\mu) \subset \rho (-\mathbb{L}_{\nu,n})\,, \qquad E_{\nu,n} \, = \, \big \{ \mu\in \C~|~ \Re \mu \geq \delta_1^{-1} n^\gamma \nu^{\frac12}\,, ~ |\mu | \geq \delta_1^{-1} \alpha \big \}
\end{align} 
Then it suffices to show that there exist $\theta\in (\frac{\pi}{2},\pi)$ and $\nu_0\in (0,1)$ satisfying 
\begin{align}\label{proof.cor.prop.general'.1}
S_{\nu,n} (\theta) \subset \cup_{\mu\in E_{\nu,n}} B_{r_\mu} (\mu)\,.
\end{align}
But it is not difficult to see that $\theta = \frac{\pi}{2} + \theta_0$ with $\theta_0\in (0,\frac{\pi}{2})$ satisfying 
$\tan \theta_0 = \frac{1}{2 C(1+\|U\|)}$ meets our purpose.

Finally let us prove \eqref{est.cor.prop.general'.3}. It suffices to estimate the norm of $(\partial_Y^2-\alpha^2) \phi$. To this end we go back to the identity \eqref{proof.prop.general'.4} and also recall the estimates \eqref{proof.prop.general'.5} and \eqref{proof.prop.general'.6}. 
Then we have from $\| \pa_Y\phi \|_{L^2} + \alpha \| \phi \|_{L^2}\leq C|\mu|^{-1} \| f\|_{L^2}$, 
\begin{align}
\| (\pa_Y^2-\alpha^2) \phi \|_{L^2}^2 & \leq C n \bigg ( |\Im c| \big (\| \pa_Y\phi \|_{L^2}^2 + \alpha^2 \| \phi \|_{L^2}^2 \big ) + \frac{1}{|\mu|^2} \| f\|_{L^2}^2 + \frac{1}{\alpha |\mu|} \| f\|_{L^2}^2  \bigg ) \nonumber \\
& \leq \frac{C}{|\mu|} \big ( \frac{n |\Im c|}{|\mu|} + \frac{n}{|\mu|} + \frac{1}{\nu^\frac12} \big ) \| f\|_{L^2}^2\,.\label{proof.cor.prop.general'.2}
\end{align} 
Here $C$ depends only on $\| U^E\|_{C^1}$ and $\|U\|$. 
Since $|\mu| = \alpha |c|=n\nu^\frac12 |c|$ and $|c|\geq \delta_1^{-1}$ by the assumption, 
the inequality \eqref{proof.cor.prop.general'.2} implies that 
\begin{align*}
\| (\pa_Y^2-\alpha^2) \phi \|_{L^2}^2 \leq \frac{C}{|\mu| \nu^\frac12} \| f\|_{L^2}^2\,,
\end{align*}
as desired. Thus \eqref{est.cor.prop.general'.3} holds. The proof is complete.
\end{proof}

\vspace{0.3cm}

The next proposition shows that the solvability of \eqref{eq.os'} in the regime $\alpha\Im c_\epsilon\gg 1$.

\begin{prop}\label{prop.large.alpha} 
There exists $\delta_2\in (0,1)$ such that if 
\begin{align}\label{est.prop.large.alpha.1}
\alpha \Im c_\epsilon \geq \frac{1}{\delta_2}\,,
\end{align}
then for any $f=(f_{1,n}, f_{2,n})\in L^2 (\R_+)^2$ there exists a unique weak solution $\phi \in H^2_0 (\R_+)$ to \eqref{eq.os'}, and $\phi$ satisfies the estimates
\begin{align}
\| \partial_Y \phi \|_{L^2} + \alpha \| \phi \|_{L^2} & \leq \frac{C}{\alpha \Im c_\epsilon} \|f \|_{L^2}\,,\label{est.prop.large.alpha.2} \\
\| (\partial_Y^2 -\alpha^2 ) \phi \|_{L^2} & \leq \frac{C}{\alpha (|\epsilon| \Im c_\epsilon)^\frac12}  \| f\|_{L^2}\,.\label{est.prop.large.alpha.3} 
\end{align} 
Here $\delta_2$ depends only on $\|U^E \|_{C^1}$ and $\|U\|$ in \eqref{bound.U}, while $C$ is a universal constant.

\end{prop}
 
\begin{rem}\label{rem.prop.large.alpha} {\rm (i) By the standard elliptic regularity the solution $\phi$ also belongs to $H^3(\R_+)$.

\noindent (ii) Proposition \ref{prop.large.alpha} is valid also for $n\geq \delta_0^{-1}\nu^{-\frac34}$.
Note that, from $\Re \mu = \alpha \Im c$ and $\alpha = n \nu^\frac12$, the condition \eqref{est.prop.large.alpha.1} is satisfied if
\begin{align*}
\Re \mu \geq \frac{1}{\delta_2}\,.
\end{align*} 
Hence, under the assumption of $\Re \mu \geq \delta^{-1}\nu^\frac12 n^\gamma$ in \eqref{parameter.stability} the condition \eqref{est.prop.large.alpha.1} is always satisfied if 
\begin{align}
n^{\gamma}\geq \frac{\delta}{\delta_2} \nu^{-\frac12}\,.
\end{align}
}
\end{rem}

\begin{proofx}{Proposition \ref{prop.large.alpha}} We first establish the a priori estimates.
Multiplying both sides of \eqref{eq.os'} by $\bar{\phi}$ and integrating over $(0,\infty)$, we obtain 
\begin{align}\label{proof.prop.large.alpha.1} 
\begin{split}
& -\epsilon \big ( \| \pa_Y^2\phi \|_{L^2}^2 + \alpha^2 \|\pa_Y \phi \|_{L^2}^2 \big ) + \langle (V-c_\epsilon) (\pa_Y^2 -\alpha^2 ) \phi, \phi\rangle_{L^2} - \langle (\pa_Y^2 V)\phi, \phi \rangle _{L^2} \\
& \quad \, = \, -\langle f_{2,n}, \phi \rangle _{L^2} +\frac{i}{\alpha} \langle f_{1,n}, \pa_Y \phi\rangle _{L^2} \,.
\end{split}
\end{align}
Note that \eqref{proof.prop.large.alpha.1} can be derived from the standard definition of weak solutions to \eqref{eq.os'}.
Then the imaginary part of \eqref{proof.prop.large.alpha.1} gives 
\begin{align}
& \frac1n \big ( \| \pa_Y^2 \phi \|_{L^2}^2 + \alpha^2 \|\pa_Y \phi\|_{L^2}^2 \big ) + \Im c_\epsilon \big ( \|\pa_Y \phi \|_{L^2}^2 + \alpha^2 \| \phi \|_{L^2}^2 \big ) \nonumber \\
& \quad \, = \, - \Im \langle (V-\Re c) (\pa_Y^2-\alpha^2) \phi, \phi \rangle _{L^2} + \Im \langle h, \phi \rangle _{L^2} \nonumber \\
& \quad \, = \, -\Im \langle V\pa_Y^2 \phi, \phi \rangle_{L^2} + \Im \langle h, \phi \rangle _{L^2}\,. \label{proof.prop.large.alpha.2} 
\end{align}
For the first term of the right hand side of \eqref{proof.prop.large.alpha.2} we see
\begin{align}
|\Im \langle V \pa_Y^2\phi,\phi \rangle _{L^2}| & \leq \| \pa_Y V \|_{L^\infty} \| \pa_Y\phi \|_{L^2} \|\phi \|_{L^2}\nonumber \\
& \leq \frac{\Im c_\epsilon}{2} \| \pa_Y \phi \|_{L^2}^2 + \frac{\| \pa_Y V \|_{L^\infty}^2}{2 \Im c_\epsilon} \| \phi \|_{L^2}^2\,. \label{proof.prop.large.alpha.3} 
\end{align}
Hence, if $\alpha \Im c_\epsilon\geq \| \pa_Y V \|_{L^\infty}$ then 
\begin{align*}
\frac1n \big ( \| \pa_Y^2 \phi \|_{L^2}^2 + \alpha^2 \|\pa_Y \phi\|_{L^2}^2 \big ) + \frac{\Im c_\epsilon}{2} \big ( \|\pa_Y \phi \|_{L^2}^2 + \alpha^2 \| \phi \|_{L^2}^2 \big ) \leq  \Im \langle h, \phi \rangle _{L^2}\,. 
\end{align*}
Since the term $\Im \langle h, \phi \rangle _{L^2}$ is estimated from above by $C\alpha^{-1} \| f \|_{L^2} \big ( \| \pa_Y \phi \|_{L^2} + \alpha \|\phi \|_{L^2}\big )$ we finally obtain 
\begin{align}\label{proof.prop.large.alpha.4} 
\frac1n \big ( \| \pa_Y^2 \phi \|_{L^2}^2 + \alpha^2 \|\pa_Y \phi\|_{L^2}^2 \big ) + (\Im c_\epsilon) \big ( \|\pa_Y \phi \|_{L^2}^2 + \alpha^2 \| \phi \|_{L^2}^2 \big ) \leq  \frac{C}{\alpha^2 \Im c_\epsilon} \| f\|_{L^2}^2\,.
\end{align}
Here $C$ is a universal constant.
Recalling $\Im c_\epsilon=\Im c+ \frac{\alpha^2}{n}$, we thus obtain \eqref{est.prop.large.alpha.2} and \eqref{est.prop.large.alpha.3} for any weak solution $\phi\in H^2_0 (\R_+)$ to \eqref{eq.os'} with $h=-f_{2,n}+\frac{1}{i\alpha} \pa_Y f_{1,n}$. The uniqueness directly follows from this a priori estimates.
As for the existence, we can use the method of continuity based on the a priori estimates  \eqref{est.prop.large.alpha.2} and \eqref{est.prop.large.alpha.3}. 
Since the argument is parallel to the existence part of the proof of Proposition \ref{prop.general'}, we omit the details here.
The proof is complete.
\end{proofx}

By arguing as in Corollary \ref{cor.prop.general'}, Proposition \ref{prop.large.alpha} yields the following result for the resolvent problem. Recall that $\alpha \Im c_\epsilon = \Re \mu + n^2\nu^\frac32$.
\begin{cor}\label{cor.prop.large.alpha} Let $\delta_2\in (0,1)$ be the number in Proposition \ref{prop.large.alpha}.
If $\Re \mu + n^2\nu^\frac32 \geq \delta_2^{-1}$ then $\mu \in \rho (-\mathbb{L}_{\nu,n})$, and 
\begin{align}
\| (\mu + \mathbb{L}_{\nu,n})^{-1} f \|_{L^2 (\Omega_\nu)} 
& \leq \frac{C}{\Re \mu + n^2\nu^\frac32} \| f\|_{L^2 (\Omega_\nu)}\,,\\
\| \nabla (\mu + \mathbb{L}_{\nu,n})^{-1} f\|_{L^2 (\Omega_\nu)} 
& \leq \frac{C}{\nu^\frac14 (\Re \mu + n^2\nu^\frac32)^\frac12} \| f\|_{L^2 (\Omega_\nu)}\,, 
\end{align}
for any $f\in \mathcal{P}_n^{(\nu)}L^2_\sigma (\Omega_\nu)$. Here $C$ is a universal constant.
\end{cor}
Corollary \ref{cor.prop.large.alpha} follows from Proposition \ref{prop.large.alpha}
and the proof is just parallel to that of Corollary \ref{cor.prop.general'}, so we omit the details here.

\subsection{Resolvent estimate for special shear profiles}\label{subsec.special}

We study \eqref{eq.os} for a special class of shear profiles such that 
the Rayleigh equations does not admit unstable eigenvalues. 
Indeed, when there is an unstable eigenvalue for the Rayleigh equations with some $\alpha>0$
it is likely that the unstable eigenvalue for \eqref{eq.resolvent} also exists in the regime $\Re \mu = \mathcal{O}(1)>0$. In other words, for the Fourier modes $n$ located in the middle range 
\begin{align}\label{middle.n}
\frac{1}{\delta_0}  \leq n \leq \frac{1}{\delta_0 \nu^\frac34}
\end{align}
the structure of the shear profile $V$ should play a crucial role.
The basic condition on $V$ in this subsection is either \eqref{concave.weak} or \eqref{concave.strong} introduced in Section \ref{sec.result}, and the positivity of $U^E(0)$ is also assumed as stated in \eqref{sec2.nontrivial.2}.
We also recall the integral condition \eqref{integral.condition} in Section \ref{sec.result}, which is confirmed under the condition \eqref{concave.weak}.

In virtue of Proposition \ref{prop.general'} we may focus only on the case $|c|\leq \delta_1^{-1}$ in this subsection.
Before going into the details, it is worthwhile noting that, as a resolvent problem, 
the Orr-Sommerfeld equations potentially possess three kinds of singularities:

\vspace{0.3cm}

\noindent (I) boundary layer singularity due to the condition $|\epsilon|\ll 1$,

\vspace{0.1cm}

\noindent (II) critical layer singularity due to the possible presence of the critical point for $V$ (roughly speaking, the point $Y_{c}$ such that $V(Y_{c})=\Re c$ holds),

\vspace{0.1cm}

\noindent (III) global singularity originated from the possible presence of the spectrum for the Rayleigh operator.

\vspace{0.3cm}

The appearance of the singularity (I) is directly seen from the equations, and it is responsible for the presence of the fast mode of  the Orr-Sommerfeld equations in Section \ref{subsubsec.fast}.
The singularity (II) appears both in the Rayleigh equations and in the Airy equations and it can cause a difficulty for the local regularity of solutions near the critical point.   
The singularity (III) has a different nature from (I) and (II), and we are forced to deal with this singularity because of our purpose: in order to show the stability in a wider class of functions in Gevrey spaces it is crucial to establish the resolvent estimate for the resolvent parameter $\mu$ as close as the imaginary axis.
In terms of the parameter $c_\epsilon$ this corresponds to the case $0<\Im c_\epsilon \ll 1$.
However, even if the shear profile is stable for the Euler (Rayleigh) equations the imaginary axis involves the spectrum of the Rayleigh operator. Hence, in general, the inverse of the Rayleigh operator naturally possesses a singularity of the form $(\Im c_\epsilon)^{-k}$ for some $k>0$ near $\Im c_\epsilon =0$, reflecting the above fact. This singularity is nothing but (III), which is somewhat independent of the local singularity as in (II), and it is more related to a global character of the Rayleigh operator. 
Finding the optimal power $k$ is an important but highly nontrivial  task. 
The best possible value is considered to be $k=1$, and the analysis of the Rayleigh equations by \cite{GGN2014}  indicates that this is indeed the case at least for some specific regime of parameters. However, there seems to be no reasons why the singularity is controlled by the order $(\Im c_\epsilon)^{-1}$ also in a wider regime of parameters, since the Rayleigh operator is not a normal operator. Our analysis below provides at least the bound $(\Im c_\epsilon)^{-2}$ for the inverse of the Rayleigh operator, but it is not clear whether this value is optimal or not in general.

\subsubsection{Analysis of Rayleigh equations}\label{subsubsec.Ray}

In this subsection we consider the Rayleigh equations
\begin{equation}\label{eq.ray.half}
\left\{
\begin{aligned}
(V -c_\epsilon) (\partial_Y^2-\alpha^2) \varphi - (\partial_Y^2 V ) \varphi  & \, = \, h \,, \qquad Y>0\,,\\
\varphi & \, = \, 0 \,, \qquad Y=0\,.
\end{aligned}\right.
\end{equation}
Here $h$ is assumed to be a given smooth function.
Note that we drop the term $-\epsilon (\partial_Y^2-\alpha^2)\partial_Y^2 \varphi$ in \eqref{eq.os'}, 
but keep the $\epsilon$-dependence of the number $c_\epsilon$
in order to make the estimate as sharp as possible. 

The goal of this subsection is to prove the following proposition about the solvability of \eqref{eq.ray.half}
with the source term $h$ of the form $h=\partial_Y \big (\psi \partial_Y V \big )$ for some $\psi\in H^1_0 (\R_+)$. Let us recall that $\delta_0\in (0,1)$ and $\delta_1\in (0,\delta_0]$ are the numbers in \eqref{def.delta_0}
and \eqref{proof.prop.general'.-1}, respectively, which depend only on $\|U^E\|_{C^2}$ and $\|U\|$ in \eqref{bound.U}. Note that we always assume \eqref{middle.n} in Section \ref{subsec.special}.

\begin{prop}\label{prop.Ray.WC} Let $0<\nu\leq \nu_0$, where $\nu_0\in (0,1]$ is the number in \eqref{integral.condition}, and let $|c|\leq \delta_1^{-1}$. Assume that \eqref{concave.weak} holds. Then there exists $\delta_1'\in (0,\delta_1]$ such that the following statement holds for all $\delta\in (0,\delta_1']$ in \eqref{parameter.stability}:
there exists a unique solution $\varphi\in H^1_0 (\R_+)\cap H^2 (\R_+)$ to \eqref{eq.ray.half} with $h=\partial_Y \big (\psi \partial_Y V \big )$ satisfying the estimates
\begin{align}
\| \partial _Y \varphi \|_{L^2} + \alpha \|\varphi \|_{L^2} & \leq \frac{C_{\rm wc}}{(\Im c_\epsilon)^\frac52} \| \psi \|_{L^2}\,, \label{est.prop.Ray.WC.1}\\
\| (\partial_Y^2 -\alpha^2)\varphi \|_{L^2} & \leq C_{\rm wc} \big ( \frac{1}{(\Im c_\epsilon)^\frac52} \| \psi \|_{L^2} + \frac{1}{\Im c_\epsilon} \| \partial_Y \psi \|_{L^2}\big )\,. \label{est.prop.Ray.WC.2}
\end{align}
If \eqref{concave.strong} holds then the above estimates can be replaced by 
\begin{align}
\| \partial _Y \varphi \|_{L^2} + \alpha \|\varphi \|_{L^2} & \leq \frac{C_{\rm sc}}{(\Im c_\epsilon)^2} \| \psi \|_{L^2}\,,\label{est.prop.Ray.WC.3}\\
\| (\partial_Y^2 -\alpha^2)\varphi \|_{L^2} & \leq C_{\rm sc} \big ( \frac{1}{(\Im c_\epsilon)^2} \| \psi \|_{L^2} + \frac{1}{\Im c_\epsilon} \| \partial_Y \psi \|_{L^2}\big )\,.\label{est.prop.Ray.WC.4}
\end{align}
Here $\delta_1'$ depends only on $\|U^E\|_{C^2}$ and $\|U\|$, $C_{\rm wc}$ depends only on $\|U^E\|_{C^2}$, $\|U\|$, $\frac{1}{U^E(0)}$, $M'$ in \eqref{integral.condition}, and $M_\sigma$ in \eqref{concave.weak} with $\sigma=\min \{1,\frac{U^E(0)}{4\|U\|}, \frac{U^E(0)}{8\|U^E\|_{C^1}}\}$, while $C_{\rm sc}$ depend only on $\|U^E\|_{C^2}$, $\|U\|$, and $M$ in \eqref{concave.strong}.
\end{prop} 

\begin{rem}\label{rem.prop.Ray.WC}{\rm In fact, Proposition \ref{prop.Ray.WC} holds under the weaker condition $-M_\sigma \pa_Y^2 U\geq (\pa_Y U)^4$ for $Y\geq \sigma$, $\sigma\in (0,1]$, rather than $-M_\sigma \pa_Y^2 U\geq (\pa_Y U)^2$ in \eqref{concave.weak}.
}
\end{rem}

In order to prove Proposition \ref{prop.Ray.WC} we start from the following lemma which is obtained from Rayleigh's trick. For the moment we take the general source term $h\in L^2 (\R_+)$.

\begin{lem}\label{lem.Ray.identity}
Let $\Im c_\epsilon >0$ and $V\in BC^2(\R_+)$.
Then any solution $\varphi\in H^1_0(\R_+) \cap H^2 (\R_+)$ to \eqref{eq.ray.half} satisfies 
\begin{align}\label{lem.Ray.identity.1}
\int_0^\infty |\partial_Y \varphi|^2 +\alpha^2 |\varphi|^2 \dd Y + \int_0^\infty \frac{(V-\Re c)\partial_Y^2 V }{|V-c_\epsilon|^2} |\varphi|^2 \dd Y \, = \, - \Re \int_0^\infty \frac{h}{V-c_\epsilon} \bar{\varphi} \dd Y\,,
\end{align}
and 
\begin{align}\label{lem.Ray.identity.2}
(\Im c_\epsilon ) \int_0^\infty \frac{\partial_Y^2 V}{|V-c_\epsilon|^2} |\varphi|^2 \dd Y  
\, = \,  - \Im \int_0^\infty \frac{h}{V-c_\epsilon} \bar{\varphi} \dd Y\,.
\end{align}

\end{lem}

\begin{proof} Multiplying both sides of \eqref{eq.ray.half} by $(V-c_\epsilon)^{-1}\bar{\varphi}$, 
we obtain from the integration by parts,
\begin{align}\label{proof.lem.Ray.identity.1}
\int_0^\infty |\partial_Y \varphi|^2 +\alpha^2 |\varphi|^2 \dd Y + \int_0^\infty \frac{\partial_Y^2 V}{V-c_\epsilon} |\varphi|^2 \dd Y \, = \, - \int_0^\infty \frac{h}{V-c_\epsilon} \bar{\varphi} \dd Y\,.
\end{align}
Then \eqref{lem.Ray.identity.1} and \eqref{lem.Ray.identity.2} follow from the real part and the imaginary part of \eqref{proof.lem.Ray.identity.1}, respectively.
The proof is complete.
\end{proof}

\vspace{0.3cm}
Lemma \ref{lem.Ray.identity} is the key to obtain the global estimate for the Rayleigh solutions when $U$ is concave. The next lemma gives a fundamental existence result and a priori estimates of solutions to \eqref{eq.ray.half} under the condition \eqref{concave.weak}.

\begin{lem}\label{lem.Ray.0} Let $|c|\leq \delta_1^{-1}$, where $\delta_1\in (0,1)$ is the number in 
Proposition \ref{prop.general'}. Assume that \eqref{concave.weak} holds.
Then for any $h\in L^2 (\R_+)$ there exists a unique solution $\varphi\in H^2(\R_+) \cap H^1_0 (\R_+)$ to \eqref{eq.ray.half}, and $\varphi$ satisfies 
\begin{align}
\| \partial_Y \varphi \|_{L^2}^2 + \alpha^2 \|\varphi\|_{L^2}^2 \leq C \frac{ \| U^E \|_{L^\infty} + |c_\epsilon| }{\Im c_\epsilon}  \,  \big | \langle \frac{h}{V-c_\epsilon}, \varphi \rangle _{L^2} \big | \,,\label{est.lem.Ray.0.2}
\end{align}
and 
\begin{align}
\| (\partial_Y^2 -\alpha^2) \varphi \|_{L^2}^2 \leq C \bigg ( \frac{\| \pa_Y^2 V\|_{L^\infty} }{\Im c_\epsilon}  \, \big | \langle \frac{h}{V-c_\epsilon}, \varphi \rangle _{L^2} \big |   + \| \frac{h}{V-c_\epsilon}\|_{L^2}^2 \bigg )\,.\label{est.lem.Ray.0.3}
\end{align}
Here $C$ is a universal constant.

\end{lem}

\begin{proof} (A priori estimates) Let $\varphi\in H^2(\R_+) \cap H^1_0 (\R_+)$ be any solution to \eqref{eq.ray.half}. Recalling \eqref{lem.Ray.identity.1},  we first observe from \eqref{parameter.stability} that 
\begin{align}
-\int_0^\infty \frac{\nu  (V-\Re c)(\partial_y^2 U^E) (\nu^\frac12 Y)}{|V-c_\epsilon|^2} |\varphi |^2 \dd Y 
\leq \frac{\nu \| U^E\|_{C^2}}{\alpha^2\Im c_\epsilon}  \alpha^2 \| \varphi \|_{L^2}^2
& \leq \frac{\delta_1 \| U^E \|_{C^2}}{n^{1+\gamma}} \alpha^2 \| \varphi\|_{L^2}^2 \nonumber \\
& \leq \frac{1}{32} \alpha^2 \| \varphi \|_{L^2}^2\,.\label{proof.lem.Ray.0.0}
\end{align}
Here we have used the choice of $\delta_1$ in \eqref{proof.prop.general'.-1}.
Thus our aim is now to obtain the upper bound of  $-\int_0^\infty \frac{\partial_Y^2 U_s (V-\Re c)}{|V-c_\epsilon|^2} |\varphi|^2 \dd Y$, which is written as 
\begin{align}
- \int_0^\infty \frac{(V-\Re c)\partial_Y^2 U}{|V-c_\epsilon|^2} |\varphi|^2 \dd Y  & \, = \,  -\int_0^\infty \frac{\big  (U^E (\nu^\frac12 Y ) - \Re c \big )\partial_Y^2 U}{|V-c_\epsilon|^2} |\varphi|^2 \dd Y  \nonumber \\
& \qquad  -\int_0^\infty \frac{\big (U (Y) - U^E(0) \big )\partial_Y^2 U}{|V-c_\epsilon|^2}  |\varphi|^2 \dd Y \nonumber \\
& \, \leq \,   -\int_0^\infty \frac{\big (U^E (\nu^\frac12 Y ) - \Re c\big )\partial_Y^2 U}{|V-c_\epsilon|^2} |\varphi|^2 \dd Y \,.\label{proof.lem.Ray.0.1}
\end{align}
Here we have used the monotonicity of $U$ and $-\pa_Y^2 U\geq 0$.
The last term of \eqref{proof.lem.Ray.0.1} is estimated as, again from $-\pa_Y^2 U\geq 0$,
\begin{align}
& -\int_0^\infty \frac{\big (U ^E (\nu^\frac12 Y) - \Re c \big ) \partial_Y^2 U}{|V-c_\epsilon|^2} |\varphi|^2 \dd Y 
\leq  
(\| U^E \|_{L^\infty} + |\Re c|) \int_0^\infty \frac{- \partial_Y^2 U}{|V-c_\epsilon|^2} |\varphi|^2 \dd Y \nonumber \\
& \qquad \qquad \qquad \qquad  =  
(\| U^E \|_{L^\infty} + |\Re c|) \int_0^\infty \frac{-\partial_Y^2 V + \nu (\partial_Y^2 U^E) (\nu^\frac12 Y)}{|V-c_\epsilon|^2} |\varphi|^2 \dd Y \nonumber \\
& \qquad \qquad \qquad \qquad = \frac{(\| U^E \|_{L^\infty} + |\Re c|)}{\Im c_\epsilon}  \Im \int_0^\infty \frac{h}{V-c_\epsilon} \bar{\varphi} \dd Y\nonumber \\
& \quad \qquad \qquad \qquad \qquad + \nu  (\| U^E \|_{L^\infty} + |\Re c|) \int_0^\infty \frac{(\partial_Y^2 U^E) (\nu^\frac12 Y)}{|V-c_\epsilon|^2} |\varphi|^2 \dd Y \nonumber \\
\begin{split}\label{proof.lem.Ray.0.2}
& \qquad \qquad \qquad \qquad \leq \frac{(\| U^E \|_{L^\infty} + |\Re c|)}{\Im c_\epsilon}    \Im \int_0^\infty \frac{h}{V-c_\epsilon} \bar{\varphi} \dd Y \\
& \qquad \qquad \quad \qquad \qquad + \frac{\nu (\| U^E \|_{L^\infty} + |\Re c|)  \, \| U^E\|_{C^2}}{(\alpha \Im c_\epsilon)^2} \alpha^2 \|\varphi \|_{L^2}^2\,. 
\end{split}
\end{align}
From the condition on the parameters we have 
\begin{align}
\frac{\nu (\| U^E \|_{L^\infty} + |\Re c|) \, \| U^E\|_{C^2}}{(\alpha \Im c_\epsilon)^2}  
&\leq \nu (\| U^E \|_{L^\infty} + |\Re c|) \, \| U^E\|_{C^2} (\frac{\delta_1}{n^\gamma \nu^\frac12})^2  \nonumber \\
& =  \frac{\delta_1^2 (\| U^E \|_{L^\infty} + |\Re c|) \, \| U^E\|_{C^2}}{ n^{2\gamma}} \leq \frac{1}{32}\,.\label{proof.lem.large.Re.3}
\end{align}
Here we have used $|c|\leq \delta_1^{-1}$ and the choice of $\delta_1$ in \eqref{proof.prop.general'.-1}.
Hence \eqref{est.lem.Ray.0.2} follows. 
To show \eqref{est.lem.Ray.0.3} we observe
\begin{align*}
\| (\partial_Y^2 -\alpha^2 ) \varphi\|_{L^2}^2 & = \| \frac{\pa_Y^2 V}{V-c_\epsilon} \varphi  + \frac{h}{V-c_\epsilon} \|_{L^2}^2  \leq 2 \big ( \| \frac{\pa_Y^2 V}{V-c_\epsilon} \varphi \|_{L^2}^2 + \| \frac{h}{V-c_\epsilon} \|_{L^2}^2\big )\,,
\end{align*}
and then, by using $|\pa_Y^2 V|\leq -\pa_Y^2 U + \nu \| U^E \|_{C^2}\leq -\pa_Y^2 V + 2\nu \| U^E \|_{C^2}$ and \eqref{lem.Ray.identity.2},
\begin{align}
\| \frac{\pa_Y^2 V}{V-c_\epsilon} \varphi \|_{L^2}^2 & \leq \| \pa_Y^2 V\|_{L^\infty} \int_0^\infty \frac{|\pa_Y^2 V|}{|V-c_\epsilon|^2} |\varphi|^2 \dd Y \nonumber  \\
& \leq \| \pa_Y^2 V\|_{L^\infty} \bigg ( \int_0^\infty \frac{-\partial_Y^2 V}{|V-c_\epsilon|^2} |\varphi|^2 \dd Y  + \frac{2 \nu \| U^E \|_{C^2}}{(\Im c_\epsilon)^2} \| \varphi \|_{L^2}^2\bigg ) \nonumber \\
& = \| \pa_Y^2 V\|_{L^\infty} \bigg ( \frac{1}{\Im c_\epsilon}  \Im \langle \frac{h}{V-c_\epsilon}, \varphi \rangle _{L^2} 
+ \frac{2\nu \| U^E \|_{C^2}}{(\Im c_\epsilon)^2} \| \varphi \|_{L^2}^2\bigg ) \,.
\end{align}
Thus we arrive at, from \eqref{est.lem.Ray.0.2},
\begin{align*}
& \| (\partial_Y^2 -\alpha^2)  \varphi \|_{L^2}^2 \\
& \leq 
2 \bigg (  \frac{\| \pa_Y^2 V\|_{L^\infty} }{\Im c_\epsilon} \Im  \langle \frac{h}{V-c_\epsilon}, \varphi \rangle _{L^2} 
+ \frac{2\nu \| U^E \|_{C^2} \| \pa_Y^2 V\|_{L^\infty} }{(\alpha \Im c_\epsilon)^2} \alpha^2 \| \varphi \|_{L^2}^2 
+ \| \frac{h}{V-c_\epsilon} \|_{L^2}^2\bigg ) \\
& \leq C \bigg ( \frac{\| \pa_Y^2 V\|_{L^\infty} }{\Im c_\epsilon} \big | \langle \frac{h}{V-c_\epsilon}, \varphi \rangle _{L^2} \big | 
\\
& \quad + \frac{\delta_1^2\| U^E \|_{C^2} \| \pa_Y^2 V\|_{L^\infty} }{n^{2\gamma}}
\big ( \frac{ \| U^E \|_{L^\infty} + |\Re c| }{\Im c_\epsilon} + 1 \big )
\big | \langle \frac{h}{V-c_\epsilon}, \varphi \rangle _{L^2}  \big | 
 + \| \frac{h}{V-c_\epsilon} \|_{L^2}^2\bigg )\\
& \leq  C \bigg ( \frac{\| \pa_Y^2 V\|_{L^\infty} }{\Im c_\epsilon}   \, \big | \langle \frac{h}{V-c_\epsilon}, \varphi \rangle _{L^2} \big | 
+ \| \frac{h}{V-c_\epsilon} \|_{L^2}^2\bigg )\,.
\end{align*}
Here we have used $|c_\epsilon|\leq |c| + n\nu\leq \delta_1^{-1} + \delta_0^{-1} \nu^\frac14$ and the choice of $\delta_0$ and $\delta_1$ in \eqref{def.delta_0}, \eqref{proof.prop.general'.-1}. 

\noindent (Uniqueness) The uniqueness of solutions in $H^2(\R_+)\cap H^1_0 (\R_+)$ is immediate from the a priori estimates.

\noindent (Existence) As in Proposition \ref{prop.general'}, the proof is based on the method of continuity using the a priori estimates. Thanks to the choice of $\delta_1$ in \eqref{proof.prop.general'.-1}, exactly by the same argument as in the proof of Proposition \ref{prop.general'}, we can show the unique existence of solutions to \eqref{eq.ray.half} in $H^2(\R_+)\cap H^1_0 (\R_+)$ at least when $|c|\geq (2\delta_1)^{-1}$. 
In particular, when $(2\delta_1)^{-1} \leq |c| \leq \delta_1^{-1}$ and $\Im c>0$ the solution satisfies the estimates \eqref{est.lem.Ray.0.2} and \eqref{est.lem.Ray.0.3}, which provides the a priori estimates of $\|\pa_Y \varphi \|_{L^2}+\alpha \|\varphi \|_{L^2}$ and $\|(\pa_Y^2-\alpha^2)\varphi\|_{L^2}$. 
Then, one can construct the solution $\varphi\in H^2(\R_+)\cap H^1_0 (\R_+)$ to \eqref{eq.ray.half} for all $c$ satisfying $|c|\leq \delta_1^{-1}$ and $\Im c>0$ by the method of continuity using the a priori estimates obtained from \eqref{est.lem.Ray.0.2} and \eqref{est.lem.Ray.0.3}. The proof is complete.
\end{proof}

\vspace{0.3cm}

\begin{prop}\label{prop.Ray.1} Under the assumptions of Lemma \ref{lem.Ray.0} 
if $h=\partial_Y g$ then 
\begin{align}\label{est.prop.Ray.1.1} 
\| \partial_Y \varphi\|_{L^2} + \alpha \|\varphi \|_{L^2} 
&  \leq C  \frac{ \| U^E \|_{L^\infty} + |c_\epsilon|}{\Im c_\epsilon}  \big ( \| \frac{g}{V-c_\epsilon}\|_{L^2} + \| \frac{Y^\frac12 \pa_Y U\,  g}{(V-c_\epsilon)^2} \|_{L^1}   \big )\,,
\end{align}
and 
\begin{align}
\begin{split}
\| (\partial_Y^2 - \alpha^2) \varphi \|_{L^2} &\leq C \bigg (  \frac{ \| U^E \|_{C^2} + \|U\| + |c_\epsilon|}{\Im c_\epsilon}  \, \big ( \| \frac{g}{V-c_\epsilon}\|_{L^2} + \| \frac{Y^\frac12 \pa_Y U\,  g}{(V-c_\epsilon)^2} \|_{L^1} \big )  + \| \frac{h}{V-c_\epsilon}\|_{L^2} \bigg )\,.\label{est.prop.Ray.1.2} 
\end{split}
\end{align} 
Here $C$ is a universal constant.

\end{prop}

\begin{proof} From the integration by parts we have 
\begin{align}
\big | \langle \frac{\partial_Y g}{V-c_\epsilon}, \varphi \rangle _{L^2} \big | & = \big | -\langle \frac{g}{V-c_\epsilon}, \partial_Y \varphi \rangle _{L^2}  + \langle \frac{\pa_Y V\,  g}{(V-c_\epsilon)^2} , \varphi \rangle _{L^2} \big |  \nonumber \\
\begin{split}
& \leq \big ( \| \frac{g}{V-c_\epsilon}\|_{L^2} + \| \frac{Y^\frac12 \pa_Y U\, g}{(V-c_\epsilon)^2} \|_{L^1} \big ) \| \partial_Y \varphi\|_{L^2}\\
& \quad + \frac{\nu^\frac12 \| U^E \|_{C^1}}{\alpha \Im c_\epsilon} \| \frac{g}{V-c_\epsilon} \|_{L^2}  \, \alpha \| \varphi \|_{L^2} \,. \label{proof.prop.Ray.1.1} 
\end{split}
\end{align} 
Note that 
\begin{align*}
\frac{\nu^\frac12 \| U^E \|_{C^1}}{\alpha \Im c_\epsilon}\leq \frac{\delta_1}{n^\gamma} \| U^E \|_{C^1}\leq 1
\end{align*}
by the condition \eqref{parameter.stability} and the choice of $\delta_1$ in \eqref{proof.prop.general'.-1}.
Then \eqref{est.prop.Ray.1.1} follows from \eqref{est.lem.Ray.0.2} and \eqref{proof.prop.Ray.1.1}.
Next, \eqref{est.prop.Ray.1.1} and \eqref{proof.prop.Ray.1.1} yield
\begin{align*}
\big | \langle \frac{\partial_Y g}{V-c_\epsilon}, \varphi \rangle _{L^2}  \big |
& \leq C   \frac{ \| U^E \|_{L^\infty} + |c_\epsilon| }{\Im c_\epsilon}   \big ( \| \frac{g}{V-c_\epsilon}\|_{L^2} + \| \frac{Y^\frac12 \pa_Y U\,  g}{(V-c_\epsilon)^2} \|_{L^1} \big )^2\,.
\end{align*}
Thus \eqref{est.prop.Ray.1.2} holds by applying \eqref{est.lem.Ray.0.3}.
The proof is complete.
\end{proof}

\begin{rem}\label{rem.prop.Ray.1.1}{\rm Instead of \eqref{proof.prop.Ray.1.1}, we can also compute as, from the Hardy inequality and $\alpha \Im c_\epsilon \geq \delta^{-1} n^\gamma$ by \eqref{parameter.stability},
\begin{align*}
\big | \langle \frac{\partial_Y g}{V-c_\epsilon}, \varphi \rangle _{L^2} \big | & \leq 
 \| \frac{g}{V-c_\epsilon}\|_{L^2} \| \partial_Y \varphi\|_{L^2} + | \langle \frac{\pa_Y V\, g}{(V-c_\epsilon)^2} , \varphi \rangle _{L^2} | \\
& \leq \| \frac{g}{V-c_\epsilon}\|_{L^2} \| \partial_Y \varphi\|_{L^2} + \nu^\frac12 \alpha^{-1} \| U^E \|_{C^1} \| \frac{g}{(V-c_\epsilon)^2}\|_{L^2} \, \alpha \| \varphi \|_{L^2} \\
& \quad + 2 \| \frac{Y \pa_Y U}{(V-c_\epsilon)^2} g \|_{L^2} \| \partial_Y \varphi \|_{L^2}\\
& \leq \| \frac{g}{V-c_\epsilon}\|_{L^2} \| \partial_Y \varphi\|_{L^2} + \delta n^{-\gamma} \| U^E \|_{C^1} \| \frac{g}{V-c_\epsilon}\|_{L^2} \, \alpha \| \varphi \|_{L^2} \\
& \quad + 2 \| \frac{Y \pa_Y U}{(V-c_\epsilon)^2} g \|_{L^2} \| \partial_Y \varphi \|_{L^2}\,.
\end{align*}
Since $\delta\leq \delta_1$ and $\delta_1$ is taken as in \eqref{proof.prop.general'.-1}, we obtain 
\begin{align*}
\big | \langle \frac{\partial_Y g}{V-c_\epsilon}, \varphi \rangle _{L^2} \big | & \leq C \big ( \| \frac{g}{V-c_\epsilon}\|_{L^2}  + \| \frac{Y \pa_Y U}{(V-c_\epsilon)^2} g \|_{L^2} \big ) \big ( \| \partial_Y \varphi \|_{L^2} + \alpha \| \varphi \|_{L^2} \big )\,,
\end{align*}
where $C$ is a universal constant. This implies from Lemma \ref{lem.Ray.0}, 
\begin{align}\label{est.rem.prop.Ray.1.1} 
\| \partial_Y \varphi\|_{L^2} + \alpha \|\varphi \|_{L^2} 
&  \leq C  \frac{ \| U^E \|_{L^\infty} + |c_\epsilon|}{\Im c_\epsilon} \big ( \| \frac{g}{V-c_\epsilon}\|_{L^2}  + \| \frac{Y \pa_Y U}{(V-c_\epsilon)^2} g \|_{L^2} \big ) \,,
\end{align}
and 
\begin{align}
\begin{split}
\| (\partial_Y^2-\alpha^2) \varphi \|_{L^2} &\leq C \bigg (  \frac{ \| U^E \|_{C^2} + \|U\| + |c_\epsilon|}{\Im c_\epsilon}  \,  \big ( \| \frac{g}{V-c_\epsilon}\|_{L^2}  + \| \frac{Y \pa_Y U}{(V-c_\epsilon)^2} g \|_{L^2} \big ) + \| \frac{h}{V-c_\epsilon}\|_{L^2} \bigg )\,.\label{est.rem.prop.Ray.1.2} 
\end{split}
\end{align}
Here $C$ is a universal constant.
}
\end{rem}

\vspace{0.3cm}

The next lemma is useful to obtain a sharp estimate when the critical point $Y_c$ for $U$, i.e., $U(Y_c)=\Re c$, is away from the origin $Y=0$.

\begin{lem}\label{lem.nodeg} Let $0<\nu\leq 1$ and $|c|\leq \delta_1^{-1}$. Assume that \eqref{concave.weak} holds. Suppose in addition that $\Re c >\frac{U^E(0)}{2}$. Then there exist $\sigma\in (0,1]$ and $\delta_1'\in (0,\delta_1]$ such that if $\delta\in (0,\delta_1']$ in \eqref{parameter.stability}  then any solution $\varphi\in H^2 (\R_+) \cap H^1_0 (\R_+)$ to \eqref{eq.ray.half} satisfies 
\begin{align}\label{est.lem.nodeg.1}
\begin{split}
& \| \frac{\sqrt{-\pa_Y^2 U}}{V-c_\epsilon} \varphi \|_{L^2}^2 + \frac{1}{\Im c_\epsilon} \big | \langle \frac{h}{V-c_\epsilon}, \varphi \rangle _{L^2} \big | \\
& \leq \frac{C}{(\Im c_\epsilon)^2} \bigg ( \big (\frac{\| Y h \|_{L^2(\{0\leq Y\leq \sigma\})}}{\delta_1 U^E(0)} \big )^2 \,  + \|\frac{h}{\sqrt{-\pa_Y^2 U}}\|_{L^2 (\{Y\geq \sigma \})}^2 \bigg )\,.
\end{split}
\end{align}
Moreover, it follows that
\begin{align}
\| \partial_Y \varphi \|_{L^2} + \alpha \| \varphi \|_{L^2} & \leq \frac{C}{\delta_1 \Im c_\epsilon} \bigg ( \frac{\| Y h \|_{L^2(\{0\leq Y\leq \sigma \})}}{U^E(0)} + \|\frac{h}{\sqrt{-\pa_Y^2 U}}\|_{L^2 (\{Y\geq \sigma\})} \bigg )\,,\label{est.lem.nodeg.2}\\
\| (\partial_Y^2 -\alpha^2 ) \varphi\|_{L^2} & \leq \frac{C}{\delta_1^\frac12 \Im c_\epsilon} \bigg ( \frac{\| Y h \|_{L^2(\{0\leq Y\leq \sigma \})}}{\delta_1 U^E(0)} + \|\frac{h}{\sqrt{-\pa_Y^2 U}}\|_{L^2 (\{Y\geq \sigma\})} \bigg ) + C \| \frac{h}{V-c_\epsilon}\|_{L^2} \,.\label{est.lem.nodeg.3}
\end{align}
Here $C$ is a universal constant,  $\delta_1'$ depends only on $\|U^E\|_{C^2}$ and $\|U\|$, while $\sigma$ is chosen as in \eqref{def.sigma} which depends only on $\|U^E\|_{C^1}$, $U^E(0)$, and $\|U\|$.

\end{lem}

\begin{rem}{\rm (i) We will use Lemma \ref{lem.nodeg} with $h=\frac{(U')^2}{V-c_\epsilon} \psi$ in the Rayleigh-Airy iteration.

\noindent 
(ii) Even in the case the critical point is close to the origin we can show an estimate as in Lemma \ref{lem.nodeg}: see Lemma \ref{lem.nodeg'} below.
}
\end{rem}

\begin{proofx}{Lemma \ref{lem.nodeg}} Recall that 
\begin{align}
|-\pa_Y^2 U| = -\pa_Y^2 U\leq  - \pa_Y^2 V + \nu \| U^E\|_{C^2}\,.\label{proof.lem.nodeg.1}
\end{align} 
Combining  \eqref{proof.lem.nodeg.1} with \eqref{lem.Ray.identity.2}, we have 
\begin{align}
\| \frac{\sqrt{-\pa_Y^2 U}}{V-c_\epsilon} \varphi\|_{L^2}^2 & =  \int_0^\infty \frac{|\pa_Y^2 U|}{|V-c_\epsilon|^2} |\varphi |^2 \dd Y\nonumber \\
& \leq \frac{1}{\Im c_\epsilon}  (\Im c_\epsilon) \int_0^\infty \frac{-\pa_Y^2 V + \nu \| U^E \|_{C^2}}{|V-c_\epsilon|^2} |\varphi |^2 \dd Y\nonumber \\
&  = \frac{1}{\Im c_\epsilon}  \Im \langle  \frac{h}{V-c_\epsilon},  \varphi\rangle _{L^2}  + \nu \| U^E\|_{C^2}  \| \frac{\varphi}{V-c_\epsilon}\|_{L^2}^2\nonumber \\
& \leq \frac{1}{\Im c_\epsilon}  \Im \langle  \frac{h}{V-c_\epsilon},  \varphi\rangle _{L^2}  + \frac{\nu \| U^E\|_{C^2}}{(\Im c_\epsilon)^2} \| \varphi \|_{L^2}^2\,.\label{proof.lem.nodeg.2}
\end{align}
Let us estimate the first term in the right-hand side of \eqref{proof.lem.nodeg.2}.
The condition \eqref{bound.U} and $U(0)=0$ give the bound $|U(Y)|\leq \|U\|\log (1+Y)$ for $Y\geq 0$,
and then, the condition $\Re c > \frac{U^E(0)}{2}>0$ implies that
\begin{align}
|V (Y) -c_\epsilon|\geq |U^E(\nu^\frac12 Y) - U^E(0) + U(Y) - \Re c| & \geq \Re c- \|U\| \log (1+Y) - \nu^\frac12 Y \| U^E \|_{C^1} \nonumber \\
& \geq \frac{U^E(0)}{2} - \|U\| \log (1+Y) - Y \| U^E \|_{C^1}\,.\nonumber 
\end{align}
Set
\begin{align}
\sigma \, = \, \min \{ 1, \, \frac{U^E(0)}{4 \|U\|}, \, \frac{U^E(0)}{8 \| U^E \|_{C^1}} \}\,. \label{def.sigma}
\end{align}
Then we have 
\begin{align*}
|V (Y) -c_\epsilon|  \geq \frac{U^E(0)}{8}\,, \qquad 0\leq Y \leq \sigma\,,
\end{align*}
which gives from the Hardy inequality
\begin{align}
\big |\int_{0\leq Y\leq \sigma} \frac{h}{V-c_\epsilon} \bar{\varphi} \dd Y \big |\leq \frac{16}{U^E(0)} \| Y h \|_{L^2(\{0\leq Y\leq \sigma\})} \|\partial_Y \varphi \|_{L^2}\,.\label{proof.lem.nodeg.3}
\end{align}
On the other hand, for the integral in the range $Y\geq \sigma$ we have
\begin{align}
\big |\int_{Y\geq \sigma} \frac{h }{V-c_\epsilon} \bar{\varphi} \dd Y \big |
& =\big |  \int_{Y\geq \sigma} \frac{h}{\sqrt{-\pa_Y^2 U}} \frac{\sqrt{-\pa_Y^2 U}}{V-c_\epsilon} \bar{\varphi} \dd Y \big | \nonumber \\
& \leq  \| \frac{h}{\sqrt{-\pa_Y^2 U}}\|_{L^2(\{Y\geq \sigma\})} \| \frac{\sqrt{-\pa_Y^2 U}}{V-c_\epsilon} \varphi \|_{L^2}\,.\label{proof.lem.nodeg.4}
\end{align}
Collecting \eqref{proof.lem.nodeg.2}, \eqref{proof.lem.nodeg.3}, and \eqref{proof.lem.nodeg.4}, we arrive at 
\begin{align*}
\| \frac{\sqrt{-\pa_Y^2 U}}{V-c_\epsilon} \varphi\|_{L^2}^2 \leq I + II\,,
\end{align*}
where 
\begin{align}\label{proof.lem.nodeg.5}
\begin{split}
I & = \frac{1}{\Im c_\epsilon} \| \frac{h}{\sqrt{-\pa_Y^2 U}}\|_{L^2(\{Y\geq \sigma\})} \| \frac{\sqrt{-\pa_Y^2 U}}{V-c_\epsilon} \varphi \|_{L^2}\,,\\
II & = \frac{16}{U^E(0) \Im c_\epsilon} \| Y h \|_{L^2(\{0\leq Y\leq \sigma\})} \|\partial_Y \varphi \|_{L^2} + \frac{\nu \| U^E\|_{C^2}}{(\Im c_\epsilon)^2} \| \varphi \|_{L^2}^2\,. 
\end{split}
\end{align}
First we consider the case $I\leq II$. In this case we have 
\begin{align}
\| \frac{\sqrt{-\pa_Y^2 U}}{V-c_\epsilon} \varphi\|_{L^2}^2  \leq 2 II\,,\label{proof.lem.nodeg.6}
\end{align}
and we also have from \eqref{proof.lem.nodeg.4} and \eqref{proof.lem.nodeg.5},
\begin{align}
\frac{1}{\Im c_\epsilon} \big | \langle \frac{h}{V-c_\epsilon}, \varphi \rangle _{L^2} \big | \leq 2  II\,.\label{proof.lem.nodeg.6'}
\end{align}
Then \eqref{est.lem.Ray.0.2} and \eqref{proof.lem.nodeg.6'} imply
\begin{align*}
& \| \partial_Y \varphi \|_{L^2}^2 + \alpha^2 \| \phi \|_{L^2}^2 \\
&\leq C \frac{\| U^E \|_{L^\infty} + |c_\epsilon| }{\Im c_\epsilon}  \big ( \frac{ \| Y h \|_{L^2(\{0\leq Y\leq \sigma\})}}{U^E(0)} \|\partial_Y \varphi \|_{L^2} + \frac{\nu \| U^E\|_{C^2}}{\Im c_\epsilon} \| \varphi \|_{L^2}^2 \big )\\
& \leq C  \frac{\| U^E \|_{L^\infty} + |c_\epsilon| }{U^E(0)\Im c_\epsilon}   \| Y h \|_{L^2(\{0\leq Y\leq \sigma\})} \|\partial_Y \varphi \|_{L^2}  + C (\| U^E \|_{L^\infty} + |c_\epsilon| ) \| U^E \|_{C^2} \, \delta^2 \alpha^2 \| \phi \|_{L^2}^2
\end{align*}
Here $C$ is a universal constant. Note that in the last line we have used the inequality 
$\alpha \Im c_\epsilon \geq \delta^{-1}\nu^\frac12 n^{\gamma}$ in \eqref{parameter.stability}.
Since $|c_\epsilon|\leq |c| + n\nu \leq \delta_1^{-1} + \delta_0^{-1}\nu^\frac14$,  if we take $\delta\in (0,\delta_1]$ small enough so that 
\begin{align}\label{condition.delta}
C (\| U^E \|_{L^\infty} + \delta_1^{-1} + \delta_0^{-1} ) \| U^E \|_{C^2} \, \delta^2\leq \frac12\,,
\end{align}
we have 
\begin{align}
\| \partial_Y \varphi \|_{L^2} + \alpha \| \phi \|_{L^2} \leq C\frac{\| U^E \|_{L^\infty} + |c_\epsilon|}{U^E(0)\Im c_\epsilon}  \, \| Y h \|_{L^2(\{0\leq Y\leq \sigma\})} \label{proof.lem.nodeg.7}
\end{align}
in the case $I\leq II$, where $C$ is a universal constant.
Then \eqref{proof.lem.nodeg.6}, \eqref{proof.lem.nodeg.6'}, and \eqref{proof.lem.nodeg.7} lead to 
\begin{align}
\begin{split}
\| \frac{\sqrt{-\pa_Y^2 U}}{V-c_\epsilon} \varphi\|_{L^2}^2  + 
\frac{1}{\Im c_\epsilon} \big | \langle \frac{h}{V-c_\epsilon}, \varphi \rangle _{L^2} \big | & \leq \big ( \frac{C}{\delta_1 U^E(0) \Im c_\epsilon} \big )^2 \| Y h \|_{L^2(\{0\leq Y\leq \sigma\})}^2\,.\label{proof.lem.nodeg.8}
\end{split}
\end{align}
Here we have used $\| U^E\|_{L^\infty} + |c_\epsilon| +1 \leq \delta_1^{-1}$ by the choice of $\delta_1$ in \eqref{proof.prop.general'.-1}. Next we consider the case $II\leq I$. In this case we have 
\begin{align*}
\| \frac{\sqrt{-\pa_Y^2 U}}{V-c_\epsilon} \varphi\|_{L^2}^2  \leq 2 I\,, \qquad \frac{1}{\Im c_\epsilon} \big | \langle \frac{h}{V-c_\epsilon}, \varphi \rangle _{L^2} \big | \leq 2  I\,,
\end{align*}
which gives 
\begin{align}
\| \frac{\sqrt{-\pa_Y^2 U}}{V-c_\epsilon} \varphi\|_{L^2}^2 + \frac{1}{\Im c_\epsilon} \big | \langle \frac{h}{V-c_\epsilon}, \varphi \rangle _{L^2} \big |  \leq \frac{C}{(\Im c_\epsilon)^2} \| \frac{h}{\sqrt{-\pa_Y^2 U}} \|_{L^2 (\{Y\geq \sigma\})}^2\,.\label{proof.lem.nodeg.9}
\end{align}
Here $C$ is a universal constant.
In particular, combining \eqref{proof.lem.nodeg.9} with \eqref{proof.lem.nodeg.8}, we obtain \eqref{est.lem.nodeg.1}.
The estimate \eqref{est.lem.nodeg.2} follows from \eqref{proof.lem.nodeg.7} and \eqref{proof.lem.nodeg.9} combined with \eqref{est.lem.Ray.0.2}. The estimate \eqref{est.lem.nodeg.3} follows from \eqref{est.lem.Ray.0.3} and \eqref{est.lem.nodeg.1}. The proof is complete.
\end{proofx}

\vspace{0.3cm}

Exactly in the same way as in the proof of Lemma \ref{lem.nodeg} we have the following estimate in the case the critical point $Y_c$ is close to the origin $Y=0$.

\begin{lem}\label{lem.nodeg'} Let $0<\nu\leq 1$ and $|c|\leq \delta_1^{-1}$. Assume that \eqref{concave.weak} holds.  Then there exists $\delta_1'\in (0,\delta_1]$ such that if $\delta\in (0,\delta_1']$ in \eqref{parameter.stability} then any solution $\varphi\in H^2(\R_+) \cap H^1_0 (\R_+)$ to \eqref{eq.ray.half} satisfies 
\begin{align}
\| \frac{\sqrt{-\pa_Y^2 U}}{V-c_\epsilon} \varphi \|_{L^2}^2 + \frac{1}{\Im c_\epsilon} \big | \langle \frac{h}{V-c_\epsilon}, \varphi \rangle _{L^2} \big |  \leq \frac{C}{(\Im c_\epsilon)^2}  \|\frac{h}{\sqrt{-\pa_Y^2 U}}\|_{L^2}^2\,,\label{est.lem.nodeg'.1}
\end{align}
as long as $\|\frac{h}{\sqrt{-\pa_Y^2 U}}\|_{L^2}$ is bounded. Moreover, it follows that
\begin{align}
\| \partial_Y \varphi \|_{L^2} + \alpha \| \varphi \|_{L^2}  & \leq \frac{C}{\delta_1^\frac12\Im c_\epsilon}  \|\frac{h}{\sqrt{-\pa_Y^2 U}}\|_{L^2}\,,\label{est.lem.nodeg'.2}\\
\| (\partial_Y^2 -\alpha^2 ) \varphi \|_{L^2} & \leq \frac{C}{\delta_1^\frac12\Im c_\epsilon}  \|\frac{h}{\sqrt{-\pa_Y^2 U}}\|_{L^2} + C \| \frac{h}{V-c_\epsilon}\|_{L^2} \,.\label{est.lem.nodeg'.3}
\end{align}
Here $\delta_1'$ depends only on $\|U^E\|_{C^2}$ and $\|U\|$, while $C$ is a universal constant.

\end{lem}

\begin{proof} The proof is similar to the one for Lemma \ref{lem.nodeg}. The only difference is that, for the first term in \eqref{proof.lem.nodeg.2}, we simply compute as 
\begin{align}
\frac{1}{\Im c_\epsilon} \big | \langle \frac{h}{V-c_\epsilon}, \varphi \rangle _{L^2} \big | \leq \frac{1}{\Im c_\epsilon} \| \frac{h}{\sqrt{-\pa_Y^2 U} }\|_{L^2} \| \frac{\sqrt{-\pa_Y^2 U}}{V-c_\epsilon} \varphi \|_{L^2} =: I'\,.
\end{align}
Then it suffices to consider two cases as in the proof of Lemma \ref{lem.nodeg}: (i) $I'\leq \frac{\nu \| U^E\|_{C^2}}{(\Im c_\epsilon)^2} \| \varphi \|_{L^2}^2$, (ii) $\frac{\nu \| U^E\|_{C^2}}{(\Im c_\epsilon)^2} \| \varphi \|_{L^2}^2\leq I'$. In fact, the case (i) leads to $\varphi=0$ under the smallness condition on $\delta$ as in \eqref{condition.delta}, and thus, it can not happen unless $h=0$. 
The case (ii) gives the estimates \eqref{est.lem.nodeg'.1} - \eqref{est.lem.nodeg'.3} by the same argument as in the proof of Lemma \ref{lem.nodeg}. The details are omitted here. The proof is complete.
\end{proof}

\vspace{0.3cm}

\noindent 
\begin{proofx}{Proposition \ref{prop.Ray.WC}} We decompose $h$ as $h= \partial_Y ( \psi \partial_Y V) = \partial_Y (\psi \partial_Y U^E (\nu^\frac12 \cdot)) + \partial_Y (\psi \partial_Y U)=:h_1 + h_2$.
Let $\varphi_j$, $j=1,2$, be the solution to \eqref{eq.ray.half} with $h=h_j$, respectively.
For $\varphi_1$ we have from \eqref{est.rem.prop.Ray.1.1},
\begin{align*}
\| \partial_Y \varphi _1 \|_{L^2} + \alpha \| \varphi_1 \|_{L^2} \leq \frac{C\nu^\frac12}{(\Im c_\epsilon)^2} ( 1+\frac{1}{\alpha \Im c_\epsilon} ) \| \psi \|_{L^2}\,,
\end{align*}
where $C$ depends only on $\| U^E\|_{C^1}$ and $\|U\|_{C^1}$. Then, from \eqref{parameter.stability} for the lower bound of $\alpha \Im c$, we have 
\begin{align*}
\| \partial_Y \varphi _1 \|_{L^2} + \alpha \| \varphi_1 \|_{L^2} \leq \frac{C}{(\Im c_\epsilon)^2} \| \psi \|_{L^2}\,.
\end{align*}
Similarly, we see from \eqref{est.rem.prop.Ray.1.2},
\begin{align*}
\| (\partial_Y^2 -\alpha^2 ) \varphi _1 \|_{L^2} \leq C\big ( \frac{1}{(\Im c_\epsilon)^2} \| \psi \|_{L^2}  + \frac{1}{\Im c_\epsilon} \| \partial_Y \psi \|_{L^2}\big )\,,
\end{align*}
with $C$ depends only on $\| U^E \|_{C^2}$ and $\|U\|_{C^1}$. The details are omitted here.
Next we estimate $\varphi_2$. First we consider the case $\Re c\leq U^E(0)/2$. 
In this case we apply Proposition \ref{prop.Ray.1} and obtain 
\begin{align*}
\| \partial_Y \varphi _2 \|_{L^2} + \alpha \| \varphi _2 \|_{L^2} \leq \frac{C}{\Im c_\epsilon} \big ( \| \frac{\pa_Y U\, \psi}{V-c_\epsilon} \|_{L^2} + \| \frac{Y^\frac12 (\pa_Y U)^2 \psi}{(V-c_\epsilon)^2} \|_{L^1} \big )\,.
\end{align*}
Then, using the condition \eqref{integral.condition}, we arrive at 
\begin{align*}
\| \partial_Y \varphi _2 \|_{L^2} + \alpha \| \varphi_2  \|_{L^2} \leq \frac{C}{\Im c_\epsilon} \big ( \frac{1}{\Im c_\epsilon} \| \psi \|_{L^2} + \frac{1}{(\Im c_\epsilon)^\frac32} \|  \psi \|_{L^2} \big ) \leq \frac{C}{ (\Im c_\epsilon)^\frac52} \| \psi \|_{L^2} \,.
\end{align*}
The estimate of $\| (\partial_Y^2 -\alpha^2 ) \varphi_2\|_{L^2}$ for the case $\Re c \leq U^E(0)/2$ is  proved in the similar manner by using \eqref{est.prop.Ray.1.2}. The details are omitted here.
In the case $\Re c\geq U^E(0)/2>0$ we can apply Lemma \ref{lem.nodeg} with $\sigma=\sigma_0>0$, 
where $\sigma_0$ depends only on $U^E$ and $U$ for any sufficiently small $\nu$.
To this end we write the term $h_2=\partial_Y (\psi \partial_Y U)$ as 
\begin{align*}
h_2 & \, = \, (V-c_\epsilon) \partial_Y \big ( \frac{\psi \partial_Y U}{V-c_\epsilon}\big ) + \frac{\psi (\partial_Y U)^2}{V-c_\epsilon}  + \nu^\frac12 (\partial_y U^E )(\nu^\frac12 Y)\frac{\psi \partial_Y U}{V-c_\epsilon}\\
& \, =: \, h_{2,1} + h_{2,2} + h_{2,3}\,. 
\end{align*}
The corresponding solutions to \eqref{eq.ray.half} are respectively denoted by $\phi_{2,j}$, $j=1,2,3$, and thus, $\varphi_2 = \sum_{j=1}^3\varphi_{2,j}$. 
By applying Lemma \ref{lem.Ray.0} the solution $\varphi_{2,1}$ is estimated as
\begin{align*}
& \| \partial_Y \varphi _{2,1} \|_{L^2} + \alpha \|\varphi _{2,1} \|_{L^2}  + \| (\partial_Y^2 - \alpha^2 )\varphi_{2,1} \|_{L^2} \leq C\big ( \frac{1}{(\Im c_\epsilon)^2}\| \psi\|_{L^2} + \frac{1}{\Im c_\epsilon} \| \partial_Y \psi \|_{L^2} \big )\,.
\end{align*}
On the other hand, Est. \eqref{est.lem.nodeg.2} together with the condition \eqref{concave.weak} implies 
\begin{align*}
& \| \partial_Y \varphi _{2,2} \|_{L^2} + \alpha \|\varphi _{2,2} \|_{L^2}  + \| (\partial_Y^2 - \alpha^2 )\varphi_{2,2} \|_{L^2} \\
& \leq \frac{C}{\Im c_\epsilon} \big ( \| \frac{Y  \psi  (\partial_Y U )^2}{V-c_\epsilon} \|_{L^2(\{0\leq Y\leq \frac{\sigma_0}{2}\})}  + \| \frac{ \psi (\partial_Y U )^2}{(V-c_\epsilon)\sqrt{-\pa_Y^2 U}} \|_{L^2(\{Y\geq \frac{\sigma_0}{2}\})}\big ) \nonumber \\
& \leq \frac{C}{(\Im c_\epsilon)^2} \| \psi \|_{L^2} \,.
\end{align*}
Here the lower bound $-M_{\sigma_0/2} \pa_Y^2 U\geq  (\pa_Y U)^2$ in \eqref{concave.weak} is essentially used
(in fact, the argument works under the weaker condition $-M_{\sigma_0/2} \pa_Y^2 U\geq  (\pa_Y U)^4$).
The solution $\varphi_{2,3}$ is estimated by using Lemma \ref{lem.Ray.0}, in virtue of the factor $\nu^\frac12$.
The details are omitted here, since the similar calculation has been already done. 
Collecting the estimates of $\varphi_1$ and $\varphi_2$ above, we obtain \eqref{est.prop.Ray.WC.1} and \eqref{est.prop.Ray.WC.2}.
The estimates \eqref{est.prop.Ray.WC.3} and \eqref{est.prop.Ray.WC.4} under the condition \eqref{concave.strong} is proved in the similar manner. 
In this case we can simply use Lemma \ref{lem.nodeg'} for the estimate of $\varphi_2$, instead of Lemma \ref{lem.nodeg}, which especially improves the estimate in the case $\Re c\leq U^E(0)/2$.
The details are omitted here. The proof is complete. 
\end{proofx}





\subsubsection{Analysis of  Airy equations}\label{subsubsec.Airy}

In this subsection we consider the Airy equations
\begin{equation}\label{eq.airy}
\left\{
\begin{aligned}
-\epsilon \partial_Y^2 \psi + (V- c_\epsilon) \psi & \, = \, h\,, \qquad Y>0\,,\\
\psi & \, = \, 0 \,, \qquad Y=0\,.
\end{aligned}\right.
\end{equation}

\begin{prop}\label{prop.airy.1} For any $h\in L^2 (\R_+)$ there is a unique solution $\psi\in H^2(\R_+) \cap H^1_0 (\R_+)$ to \eqref{eq.airy}, and $\psi$ satisfies 
\begin{align}
\| \psi \|_{L^2} & \leq \frac{1}{\Im c_\epsilon}  \| h \|_{L^2}\,,\label{est.prop.airy.1.1} \\
\| \partial_Y \psi \|_{L^2}  & \leq \big (\frac{1}{4 |\epsilon| \Im c_\epsilon}\big )^\frac12 \| h \|_{L^2}\,,\label{est.prop.airy.1.2} \\
\| \partial_Y^2 \psi \|_{L^2}  & \leq 2 \big ( \frac{1}{|\epsilon|}  + \frac{\| \pa_Y V \|_{L^\infty} }{|\epsilon|^\frac12 (\Im c_\epsilon)^\frac32}  \big )  \| h \|_{L^2}\,.\label{est.prop.airy.1.3} 
\end{align}

\end{prop}

\begin{proof} It suffices to prove the a priori estimates. By taking the inner product for the first equation of \eqref{eq.airy} with $\psi$ we have 
\begin{align}\label{proof.prop.airy.1}
\epsilon \int_0^\infty |\partial_Y \psi |^2 \dd Y  + \int_0^\infty (V-c_\epsilon) |\psi|^2 \dd Y \, = \, \int_0^\infty h \bar{\psi} \dd Y\,.
\end{align}
Recalling that $\epsilon=-\frac{i}{n}$, we take the imaginary part of \eqref{proof.prop.airy.1}, which yields
\begin{align}
\frac{1}{n} \| \partial_Y \psi \|_{L^2}^2 + (\Im c_\epsilon ) \|\psi\|_{L^2}^2 \, = \, - \Im \, \langle h,  \psi \rangle_{L^2}\,.\label{proof.prop.airy.2}
\end{align} 
Thus \eqref{est.prop.airy.1.1} holds  by the H${\rm \ddot{o}}$lder inequality and the Young inequality 
\begin{align*}
|\langle h, \psi\rangle_{L^2} |\leq \frac{1}{4\Im c_\epsilon} \| h \|_{L^2}^2 + (\Im c_\epsilon ) \|\psi \|_{L^2}^2
\end{align*}
for the estimate of $\| \partial_Y \psi \|_{L^2}$ and $|\langle h, \psi\rangle_{L^2} |\leq \frac{1}{2 \Im c_\epsilon} \| h \|_{L^2}^2 + \frac{\Im c_\epsilon}{2} \|\psi \|_{L^2}^2$ for the estimate of $\|\psi\|_{L^2}$. 
To show the estimate of $\partial_Y^2 \psi$ we multiply both sides of the first equation of \eqref{eq.airy} by $\partial_Y^2\bar{\psi}$ and then integrate over $(0,\infty)$, which gives
\begin{align*}
-\epsilon \|\partial_Y^2 \psi \|_{L^2}^2 + \langle V \psi , \partial_Y^2\psi \rangle _{L^2} + c_\epsilon \| \partial_Y \psi \|_{L^2}^2 \, = \, \langle h, \partial_Y^2 \psi \rangle _{L^2}\,.
\end{align*} 
Taking the imaginary part of both sides of the above equation, we obtain from the integration by parts,
\begin{align}
\frac{1}{n} \| \partial_Y^2 \psi \|_{L^2}^2 + \Im c_\epsilon  \, \| \partial _Y \psi \|_{L^2}^2 & \, = \, \Im \langle (\pa_Y V) \psi, \partial_Y \psi \rangle _{L^2}  + \Im \langle h, \partial_Y^2 \psi \rangle _{L^2}  \nonumber \\
& \, \leq \, \| \pa_Y V \|_{L^\infty} \| \psi \|_{L^2} \| \partial_Y \psi \|_{L^2} + \Im \langle h, \partial_Y^2 \psi \rangle _{L^2}  \nonumber \\
& \, =: \, I + II \,.\label{proof.prop.airy.3}
\end{align}
If $I\leq II$ then \eqref{proof.prop.airy.3} implies 
\begin{align*}
\frac1n \| \partial_Y^2\psi \|_{L^2}^2 \leq 2 II \leq 2 \| h \|_{L^2} \| \partial_Y^2 \psi \|_{L^2}\,,
\end{align*}
which yields $\| \partial_Y^2 \psi \|_{L^2} \leq 2|\epsilon|^{-1} \| h \|_{L^2}$. 
On the other hand, if $II\leq I$ then \eqref{proof.prop.airy.3} gives 
\begin{align*}
\Im c_\epsilon \, \| \partial_Y\psi \|_{L^2}^2 \leq 2 I = 2 \| \pa_Y V \|_{L^\infty} \| \psi \|_{L^2} \| \partial_Y \psi \|_{L^2}\,,
\end{align*}
and therefore,  $\|\partial_Y \psi \|_{L^2} \leq \frac{2 \| \pa_Y V \|_{L^\infty}}{\Im c_\epsilon} \| \psi \|_{L^2}$.
Then, again from \eqref{proof.prop.airy.3}, we have 
\begin{align*}
\frac{1}{n} \| \partial_Y^2 \psi \|_{L^2}^2 \leq 2 I \leq \frac{4 \| \pa_Y V \|_{L^\infty}^2}{\Im c_\epsilon} \| \psi \|_{L^2} ^2 \leq \frac{4 \| \pa_Y V \|_{L^\infty}^2}{(\Im c_\epsilon)^3} \| h\|_{L^2}^2 \,.
\end{align*}
Collecting these, we obtain the estimate of $\| \partial_Y^2 \psi \|_{L^2}$. The proof is complete.
\end{proof}


 
\subsubsection{Solutions to modified Orr-Sommerfeld equations with Dirichlet boundary condition}\label{subsubsec.iterate}

The modified Orr-Sommerfeld operator is defined as
\begin{align}\label{def.mos.op}
mOS (\phi) = -\epsilon (\partial_Y^2-\alpha^2) \partial_Y^2 \phi +  (V - c_\epsilon) (\partial_Y^2-\alpha^2) \phi - (\partial_Y^2 V ) \phi \,.
\end{align}
Then we consider the modified Orr-Sommerfeld equations in $\R_+$:
\begin{equation}\label{eq.os.R}
\left\{
\begin{aligned}
mOS (\phi) & \, = \, h \,, \qquad Y >0\,,\\
\phi & \, = \, 0\,, \qquad Y=0\,.
\end{aligned}\right.
\end{equation}
Here $h$ is a given smooth function. 
Note that the boundary condition on $\partial_Y \phi$ is not imposed for the moment.
Our aim is to construct a function $\Phi_{mOS}[h]$ solving \eqref{eq.os.R}, 
around the Rayleigh mode $\varphi^{(0)}=\Phi_{Ray}[h]$ constructed in Lemma \ref{lem.Ray.0}. 
To this end set $\phi^{(1)}=\phi - \varphi^{(0)}$, which should solve 
\begin{equation}\label{eq.os.R'}
\left\{
\begin{aligned}
mOS ( \phi ^{(1)}) & \, = \, \epsilon  (\partial_Y^2-\alpha^2) \partial_Y^2 \varphi^{(0)} \,, \qquad Y>0\,,\\
\phi^{(1)} & \, = \, 0\,, \qquad Y=0\,.
\end{aligned}\right.
\end{equation}
Since the first equation of \eqref{eq.os.R'} is regarded as 
\begin{align*}
\begin{split}
& (\partial_Y^2-\alpha^2) \big ( -\epsilon \partial_Y^2  +  (V -c_\epsilon) \big )  \phi^{(1)} -2 (\partial_Y V) \partial_Y \phi^{(1)} - 2(\partial_Y^2 V ) \phi ^{(1)}\, = \epsilon  (\partial_Y^2-\alpha^2) \partial_Y^2 \varphi^{(0)} \,,
\end{split}
\end{align*}
we set $\psi^{(1)}$ as the solution to
\begin{equation*}
\left\{
\begin{aligned}
-\epsilon \partial_Y^2  \psi^{(1)} +  (V -c_\epsilon) \psi^{(1)}  & \, = \, \epsilon \partial_Y^2\varphi^{(0)}\,, \qquad Y > 0 \,,\\
\psi^{(1)} & \, = \, 0 \,, \qquad Y=0\,.
\end{aligned}\right.
\end{equation*}
Then $\phi^{(1,1)}=\phi^{(1)}-\psi^{(1)}$ solves
\begin{equation*}
\left\{
\begin{aligned}
mOS (\phi ^{(1,1)}) & \, = \, h^{(1)} \,, \qquad Y>0\,,\\
\phi^{(1,1)} & \, = \, 0 \,, \qquad Y=0\,.
\end{aligned}\right.
\end{equation*}
where 
\begin{align*}
h^{(1)} \, = \,  2 (\partial_Y V) \partial_Y \psi^{(1)} + 2(\partial_Y^2 V ) \psi ^{(1)} \, = \, 2 \partial_Y \big ( (\partial_Y V) \psi^{(1)} \big )\,.
\end{align*}
Next we set 
\begin{align*}
\phi^{(2)} \, = \, \phi^{(1,1)} - \varphi^{(1)}\,, \qquad \quad \varphi^{(1)}\, = \, \Phi_{Ray} [h^{(1)}]\,.
\end{align*}
and therefore $\phi^{(2)}$ should solve
\begin{equation*}
\left\{
\begin{aligned}
mOS (\phi^{(2)} ) & \, = \, \epsilon  (\partial_Y^2-\alpha^2) \partial_Y^2 \varphi^{(1)} \,, \qquad Y>0\,,\\
\phi^{(2)} & \, = \, 0 \,, \qquad Y=0\,.
\end{aligned}\right.
\end{equation*}

In analogy of the above argument the functions $\psi^{(k)}$, $h^{(k)}$, and $\varphi^{(k)}$ are inductively defined as follows: $\psi^{(k)}$ is the solution to 
\begin{equation}\label{def.psi.k}
\left\{
\begin{aligned}
-\epsilon \partial_Y^2  \psi^{(k)} +  (V -c_\epsilon) \psi^{(k)} & \,  = \, \epsilon \partial_Y^2 \varphi^{(k-1)}\,, \qquad Y > 0\,,\\
\psi^{(k)} & \, = \, 0\,, \qquad  Y=0\,.
\end{aligned}\right.
\end{equation}
while $h^{(k)}$ and $\varphi^{(k)}$ are defined as 
\begin{align}\label{def.h.varphi.k}
h^{(k)} \, = \,  2 \partial_Y \big ( (\partial_Y V) \psi^{(k)} \big )  \,, \qquad \varphi^{(k)} \, = \, \Phi_{Ray} [h^{(k)}]\,.
\end{align}
Our goal is to show the convergence of the sums 
\begin{align}\label{def.series}
\sum_{k=1}^\infty \psi^{(k)}\,, \qquad  \sum_{k=1}^\infty \varphi^{(k)}\,.
\end{align}
Indeed, if they converge then the function
\begin{align}\label{def.phi.1}
\phi^{(1)} \, = \, \sum_{k=1}^\infty \psi^{(k)} +  \sum_{k=1}^\infty \varphi^{(k)}
\end{align}
solves \eqref{eq.os.R'}, as desired. To ensure the convergence of the above series the lower bound of $\gamma$ in \eqref{parameter.stability} is required. As will be seen in the next proposition, under the condition \eqref{concave.weak} we need $\gamma\geq \frac57$, while under the condition \eqref{concave.strong} it can be replaced by $\gamma\geq \frac23$. Our argument is built upon Proposition \ref{prop.Ray.WC} for the Rayleigh equations and Proposition \ref{prop.airy.1} for the Airy equations.
In the next proposition the number $\nu_0$ is as in \eqref{integral.condition}, while the number $\delta_1'\in (0,\delta_1]$ and the constants $C_{{\rm wc}}$, $C_{{\rm sc}}$ are given in Proposition \ref{prop.Ray.WC}.

\begin{prop}[Solvability of \eqref{eq.os.R'}]\label{prop.mos} 
Let $0<\nu \leq \nu_0$ and $|c|\leq \delta_1^{-1}$. Assume that \eqref{concave.weak} holds. 
Then there exists $\delta_* \in (0,\delta_1']$ such that the following statement holds.
If $\gamma \in [\frac57,1]$ and $\delta\in (0,\delta_*]$ in \eqref{parameter.stability},
then the series in \eqref{def.series} converges in $H^2(\R_+)$, and it follows that
\begin{align}
\frac{\Im c_\epsilon}{|\epsilon|} \| \sum_{k=1}^\infty \psi^{(k)} \|_{L^2}  + \big (\frac{\Im c_\epsilon}{|\epsilon|}\big )^\frac12 \| \partial_Y \sum_{k=1}^\infty \psi^{(k)}\|_{L^2}  + \| \partial_Y^2 \sum_{k=1}^\infty \psi^{(k)}\|_{L^2} & \leq C \| \partial_Y^2 \varphi^{(0)}\|_{L^2} \,, \label{est.prop.mos.1} 
\end{align}
and 
\begin{align}\label{est.prop.mos.2} 
\begin{split}
\| \partial_Y \sum_{k=1}^\infty \varphi^{(k)} \|_{L^2} + \alpha \|  \sum_{k=1}^\infty \varphi^{(k)}\|_{L^2} & \leq  \frac{2C_{{\rm wc}} |\epsilon|}{(\Im c_\epsilon)^\frac72} \| \partial_Y^2 \varphi^{(0)}\|_{L^2} \,,\\
\| (\partial_Y^2-\alpha^2) \sum_{k=1}^\infty \varphi^{(k)} \|_{L^2}& \leq 2C_{{\rm wc}} \bigg ( \frac{|\epsilon|}{(\Im c_\epsilon)^\frac72} +  \frac{|\epsilon|^\frac12}{(\Im c_\epsilon)^\frac32} \bigg )  \| \partial_Y^2 \varphi^{(0)}\|_{L^2} \,.
\end{split}
\end{align}

\noindent 
If \eqref{concave.strong} holds in addition, then  the condition $\gamma\in [\frac57, 1]$ is relaxed to $\gamma\in [\frac23,1]$, and the estimates for $\sum_{k=1}^\infty \varphi^{(k)}$ stated above are replaced by 
\begin{align}\label{est.prop.mos.3} 
\begin{split}
\| \partial_Y \sum_{k=1}^\infty \varphi^{(k)} \|_{L^2} + \alpha \|  \sum_{k=1}^\infty \varphi^{(k)}\|_{L^2} & \leq  \frac{2C_{{\rm sc}} |\epsilon|}{(\Im c_\epsilon)^3} \| \partial_Y^2 \varphi^{(0)}\|_{L^2} \,,\\
\| (\partial_Y^2-\alpha^2) \sum_{k=1}^\infty \varphi^{(k)}\|_{L^2}& \leq 2C_{{\rm sc}}  \frac{|\epsilon|^\frac12}{(\Im c_\epsilon)^\frac32} \| \partial_Y^2 \varphi^{(0)}\|_{L^2} \,.
\end{split}
\end{align}
Here $C$ is a universal constant, and $\delta_*$ depends only on $C_{{\rm wc}}$ under the assumption \eqref{concave.weak}, while on $C_{{\rm sc}}$ under the assumption \eqref{concave.strong}.
\end{prop}

\begin{rem}\label{rem.prop.mos}{\rm (i) As is seen in the proof below, the constant $\frac{|\epsilon|}{(\Im c_\epsilon)^\frac72}$ is small when $\gamma \in [\frac57,1]$ in \eqref{parameter.stability}, while $\frac{|\epsilon|}{(\Im c_\epsilon)^3}$ is small when $\gamma \in [\frac23,1]$ in \eqref{parameter.stability}.

\noindent (ii) Proposition \ref{prop.mos} gives the estimates for $\phi^{(1)} = \sum_{k=1}^\infty \psi^{(k)} +  \sum_{k=1}^\infty \varphi^{(k)}$ which solves \eqref{eq.os.R'}. If $\alpha \Im c_\epsilon\leq \delta_2^{-1}$ in addition, where $\delta_2$ is the number in Proposition \ref{prop.large.alpha},
then \eqref{est.prop.mos.1} implies $\alpha^2 \| \sum_{k=1}^\infty \psi^{(k)} \|_{L^2}\leq \frac{C|\epsilon|}{(\Im c_\epsilon)^3}\| \pa_Y^2 \varphi^{(0)}\|_{L^2}$. Thus we have   
\begin{align}\label{est.prop.mos.4} 
\| \pa_Y \phi^{(1)} \|_{L^2} + \alpha \| \phi^{(1)} \|_{L^2} + \| (\pa_Y^2-\alpha^2) \phi^{(1)} \|_{L^2} \leq C \| \pa_Y^2 \varphi^{(0)} \|_{L^2}\,,
\end{align}
when $\alpha \Im c_\epsilon\leq \delta_2^{-1}$ under the assumptions of Proposition \ref{prop.mos}. 
}
\end{rem}

\begin{proofx}{Proposition \ref{prop.mos}} 
We will show that 
\begin{align}\label{proof.prop.mos.2}
\begin{split}
\| \partial_Y^2 \varphi^{(k)} \|_{L^2} \leq \| (\pa_Y^2-\alpha^2) \varphi^{(k)} \|_{L^2} \leq B_1^k \| \partial_Y^2 \varphi^{(0)}\|_{L^2}\,, \qquad k\in \N \cup \{0\}\,,
\end{split} 
\end{align}
where $B_1$ is given as 
\begin{align}\label{proof.prop.mos.3}
B_1 & \, = \, C_{{\rm wc}}  \bigg ( \frac{|\epsilon|}{(\Im c_\epsilon)^\frac72} +  \frac{|\epsilon|^\frac12}{(\Im c_\epsilon)^\frac32} \bigg ) \,,
\end{align} 
where $C_{{\rm wc}}$ is the constant in Proposition \ref{prop.Ray.WC}.
Note that the first inequality of \eqref{proof.prop.mos.2} is always valid.
The estimate \eqref{proof.prop.mos.2} clearly holds for $k=0$. 
Next suppose that \eqref{proof.prop.mos.2} holds for $k$. Then Proposition \ref{prop.airy.1} for  the Airy equations gives
\begin{align}\label{proof.prop.mos.4}
\begin{split}
\| \psi^{(k+1)} \|_{L^2} & \leq \frac{|\epsilon|}{\Im c_\epsilon} \| \partial_Y^2 \varphi^{(k)} \|_{L^2}\,,\\
\| \partial_Y \psi^{(k+1)} \|_{L^2} & \leq \big (\frac{|\epsilon|}{4\Im c_\epsilon} \big)^\frac12 \| \partial_Y^2 \varphi^{(k)} \|_{L^2}\,,\\
\| \partial_Y^2 \psi^{(k+1)} \|_{L^2}  & \leq 2 \big ( 1  + \frac{\| \pa_Y V \|_{L^\infty} |\epsilon|^\frac12}{(\Im c_\epsilon)^\frac32}  \big )  \| \partial_Y^2\varphi^{(k)} \|_{L^2}\,.
\end{split}
\end{align}
On the other hand, Proposition \ref{prop.Ray.WC} yields
\begin{align}\label{proof.prop.mos.6}
\| \partial_Y \varphi^{(k+1)} \|_{L^2} + \alpha \| \varphi ^{(k+1)} \|_{L^2} \leq \frac{C_{{\rm wc}}}{(\Im c_\epsilon)^\frac52} \|\psi^{(k+1)}\|_{L^2}\,,
\end{align}
and 
\begin{align}\label{proof.prop.mos.7}
\| (\partial_Y^2-\alpha^2) \varphi^{(k+1)}\|_{L^2} \leq C_{{\rm wc}} \bigg ( \frac{1}{(\Im c_\epsilon)^\frac52}\|\psi^{(k+1)}\|_{L^2}  + \frac{1}{\Im c_\epsilon}\| \partial_Y \psi^{(k+1)}\|_{L^2} \bigg )\,.
\end{align}
Hence \eqref{proof.prop.mos.4}, \eqref{proof.prop.mos.6}, and \eqref{proof.prop.mos.7} imply
\begin{align}\label{proof.prop.mos.8}
\| \partial_Y \varphi^{(k+1)} \|_{L^2} + \alpha \| \varphi ^{(k+1)} \|_{L^2} \leq \frac{C_{{\rm wc}} |\epsilon|}{(\Im c_\epsilon)^\frac72} \| \partial_Y^2 \varphi^{(k)} \|_{L^2}\,,
\end{align}
and 
\begin{align}\label{proof.prop.mos.9}
\| (\partial_Y^2-\alpha^2) \varphi^{(k+1)}\|_{L^2} \leq C_{{\rm wc}}  \bigg ( \frac{|\epsilon|}{(\Im c_\epsilon)^\frac72} + \frac{|\epsilon|^\frac12}{(\Im c_\epsilon)^\frac32} \bigg ) \| \partial_Y^2 \varphi^{(k)}\|_{L^2} = B_1 \| \partial_Y^2 \varphi^{(k)}\|_{L^2} 
\end{align}
by the definition of $B_1$.
Thus \eqref{proof.prop.mos.2} holds for all $k$.
To achieve the convergence we need the smallness of $B_1$, and in view of \eqref{proof.prop.mos.3} 
this requires the smallness of 
\begin{align*}
\frac{|\epsilon|^\frac12}{(\Im c_\epsilon)^\frac32} \qquad {\rm and} \qquad \frac{|\epsilon|}{(\Im c_\epsilon)^\frac72}\,.
\end{align*}
Recall that $|\epsilon|=n^{-1}$ and $\Im c_\epsilon \geq \Im c \geq \delta^{-1}n^{\gamma-1}$ by the condition \eqref{parameter.stability}. Therefore, we need the smallness of  
\begin{align*}
& \frac{|\epsilon|^\frac12}{(\Im c_\epsilon)^\frac32}  \, \leq \, \big ( n^{-1} (\delta n^{1-\gamma})^3 \big )^\frac12 \, = \, \big ( \delta^3 n^{2-3\gamma} \big )^\frac12\,, \\
& \frac{|\epsilon|}{(\Im c_\epsilon)^\frac72}  \, \leq \, n^{-1} (\delta n^{1-\gamma})^\frac72 \, =\, \delta^\frac72 n^{\frac52-\frac72\gamma}\,,
\end{align*}
from which we need the condition $\gamma \geq \frac57$ and the smallness of $\delta$.
Finally  the estimates of $\psi^{(k+1)}$ are obtained from \eqref{proof.prop.mos.4} and \eqref{proof.prop.mos.2}
when $B_1\leq \frac12$, while the $H^1$ norm of $\varphi^{(k+1)}$, $k\in \N \cup \{0\}$, is estimated from \eqref{proof.prop.mos.8} and \eqref{proof.prop.mos.2}. Note that the smallness of $\delta_*$ leads to the bound $\frac{\| \pa_Y V \|_{L^\infty} |\epsilon|^\frac12}{(\Im c_\epsilon)^\frac32} \leq 1$ in \eqref{proof.prop.mos.4}. 
When \eqref{concave.strong} holds then the above argument works for $\gamma\geq \frac23$ with the factor $\frac{|\epsilon|}{(\Im c_\epsilon)^\frac72}$ replaced by $\frac{|\epsilon|}{(\Im c_\epsilon)^3}$, in virtue of Proposition \ref{prop.Ray.WC}.  The proof  is complete.
\end{proofx}

A direct consequence of Proposition \ref{prop.mos} and Remark \ref{rem.prop.mos}, we have the following result on the solvability of the Orr-Sommerfeld equations \eqref{eq.os.R}.

\begin{cor}\label{cor.prop.mos} Let $h\in L^2 (\R_+)$. Under the assumption of Proposition \ref{prop.mos} there exists a weak solution $\phi=\Phi_{mOS} [h] \in H^2(\R_+) \cap H^1_0 (\R_+)$ to \eqref{eq.os.R} satisfying the estimate 
\begin{align}
\| \partial_Y (\phi - \Phi_{Ray}[h] ) \|_{L^2} + \alpha \| \phi - \Phi_{Ray}[h]  \|_{L^2} + \| \partial_Y^2 (\phi - \Phi_{Ray}[h] ) \|_{L^2} & \leq C \| \partial_Y^2 \Phi_{Ray}[h] \|_{L^2} \,, \label{est.cor.prop.mos.1}
\end{align}
where $C$ is a universal constant, while the leading term $\Phi_{Ray}[h]$ satisfies the estimates in Lemma \ref{lem.Ray.0}. Moreover, if $\alpha \Im c_\epsilon\leq \delta_2^{-1}$ in addition, where $\delta_2$ is the number in Proposition \ref{prop.large.alpha}, then we also have 
\begin{align}
\| (\pa_Y^2 - \alpha^2 ) ( \phi - \Phi_{Ray}[h] ) \|_{L^2} \leq C \| \pa_Y^2 \Phi_{Ray}[h] \|_{L^2}\,.\label{est.cor.prop.mos.2}
\end{align}

\end{cor}


Proposition \ref{prop.mos} leads to the solvability of \eqref{eq.os.R} also for the case $h=- f_{2,n} + \frac{1}{i\alpha} \partial_Y f_{1,n}$ as follows.

\begin{prop}\label{prop.mos.h}
Let $\alpha \Im c_\epsilon\leq \delta_2^{-1}$. When $h=- f_{2,n} + \frac{1}{i\alpha} \partial_Y f_{1,n}$ for $f=(f_{1,n},f_{2,n})\in L^2 (\R_+)^2$ the solution $\phi$ in Corollary \ref{cor.prop.mos} satisfies the estimates
\begin{align}
\| \partial_Y \phi \|_{L^2} + \alpha \| \phi \|_{L^2} & \leq \frac{C}{\alpha (\Im c_\epsilon)^\frac52} \| f\|_{L^2}\,,\label{est.prop.mos.h.1}\\
\| (\partial_Y^2 -\alpha^2) \phi \|_{L^2} & \leq \frac{C}{\alpha} \big ( \frac{1}{(\Im c_\epsilon)^\frac52} + \frac{1}{(|\epsilon| \Im c_\epsilon)^\frac12}  \big ) \| f\|_{L^2}\,.\label{est.prop.mos.h.2}
\end{align}
Here $C$ depends only on $U^E$ and $U$. When \eqref{concave.strong} holds the factor $(\Im c_\epsilon)^{-\frac52}$ in \eqref{est.prop.mos.h.1} and \eqref{est.prop.mos.h.2} is replaced by $(\Im c_\epsilon)^{-2}$.
\end{prop}

\begin{rem}\label{rem.prop.mos.h}{\rm By allowing  a larger negative power on $\Im c_\epsilon$ in \eqref{est.prop.mos.h.1} and \eqref{est.prop.mos.h.2},  Proposition \ref{prop.mos.h} is valid under the slightly weaker condition than \eqref{concave.weak}: $-M_\sigma \pa_Y^2 U\geq (\pa_Y U)^4$, rather than $-M_\sigma \pa_Y^2 U\geq (\pa_Y U)^2$ in \eqref{concave.weak}. Under this condition the factor $(\Im c_\epsilon)^{-\frac52}$ in \eqref{est.prop.mos.h.1} and \eqref{est.prop.mos.h.2} is replaced by $(-\Im c_\epsilon)^{-3}$.
}
\end{rem}

\begin{proofx}{Proposition \ref{prop.mos.h}} The Rayleigh-Airy iteration in the proof of Proposition \ref{prop.mos} requires the bound of the second  derivative for the solution to the Rayleigh equations. Thus, we need to be careful about the choice of the ``first approximation'', for the solution $\Phi_{mOS}[h]$ has to be estimated in terms of $\| f\|_{L^2}$, rather than $\|h\|_{L^2}$. Let $\Phi_0 = \Phi_0[f]$ be the solution to 
\begin{equation}\label{eq.os.R.0}
\left\{
\begin{aligned}
& -\epsilon (\partial_Y^2 -\alpha^2) \partial_Y^2  \Phi_0 + \partial_Y \big ( (V-c_\epsilon) \partial_Y\Phi_0 \big ) - \nu^\frac12 (\partial_y U^E )(\nu^\frac12 \cdot) \partial_Y \Phi_0 \\
& \qquad \qquad   - (V-c_\epsilon) \alpha^2 \Phi_0  - (\partial_Y^2 V) \Phi_0   \, = \, h \,, \qquad Y >0\,,\\
& \Phi_0 (0)  \, = \, \partial_Y \Phi_0 (0) \, = \, 0\,, \qquad Y=0\,.
\end{aligned}\right.
\end{equation}
This elliptic problem is uniquely solvable. Indeed, it suffices to show the a priori estimate.
By taking the inner product with $\Phi_0$ in the first equation of \eqref{eq.os.R.0}, we have
\begin{align*}
& -\epsilon \big ( \| \partial_Y^2 \Phi_0 \|_{L^2}^2 + \alpha^2 \| \partial_Y \Phi_0 \|_{L^2}^2 \big ) 
- \langle (V-c_\epsilon) \partial_Y \Phi_0, \partial_Y \Phi_0\rangle _{L^2} -  \alpha^2 \langle (V-c_\epsilon)\Phi_0, \Phi_0\rangle_{L^2} \\
& \qquad -\nu^\frac12 \langle (\partial_y U^E )(\nu^\frac12 \cdot) \partial_Y \Phi_0 , \Phi_0\rangle _{L^2}  - \langle (\partial_Y^2 V) \Phi_0, \Phi_0 \rangle_{L^2} \, = \, \langle h, \Phi_0 \rangle _{L^2}\,.
\end{align*}
Then the imaginary part of the above identity gives 
\begin{align}\label{proof.prop.mos.h.1}
\begin{split}
& \frac1n \big ( \| \partial_Y^2 \Phi_0 \|_{L^2}^2 + \alpha^2 \| \partial_Y \Phi_0 \|_{L^2} ^2 \big ) +  \Im c_\epsilon \big ( \| \partial_Y \Phi_0 \|_{L^2}^2 + \alpha^2 \| \Phi_0 \|_{L^2} ^2 \big )  \\
& \qquad \, = \,  \nu^\frac12 \Im  \langle (\partial_y U^E )(\nu^\frac12 \cdot) \partial_Y \Phi_0 , \Phi_0\rangle _{L^2}  + \Im \langle h, \Phi_0 \rangle _{L^2}\,.
\end{split}
\end{align}
We see 
\begin{align*}
\nu^\frac12 \Im  \langle (\partial_y U^E )(\nu^\frac12 \cdot) \partial_Y \Phi_0 , \Phi_0\rangle _{L^2}
& \leq \frac{\nu^\frac12 \| U^E\|_{C^1}}{\alpha} \|\partial_Y \Phi_0\|_{L^2} \, \alpha \|\Phi_0\|_{L^2}\\
& \leq \frac{\nu^\frac12 \| U^E \|_{C^1}}{2\alpha \Im c_\epsilon} \, \Im c_\epsilon \big ( \|\partial_Y \Phi_0\|_{L^2}^2 +  \alpha^2 \|\Phi_0\|_{L^2}^2\big )\\
& \leq \frac{\Im c_\epsilon}{2}\big ( \|\partial_Y \Phi_0\|_{L^2}^2 +  \alpha^2 \|\Phi_0\|_{L^2}^2\big )
\end{align*}
in virtue of \eqref{parameter.stability} and the smallness of $\delta$. 
which implies the unique solvability of \eqref{eq.os.R.0}. Note that 
\begin{align}\label{proof.prop.mos.h.2}
\big | \langle h, \Phi_0\rangle _{L^2}\big | \leq \frac{1}{\alpha} \| f \|_{L^2} \big ( \| \partial_Y \Phi_0 \|_{L^2} + \alpha \|\Phi_0 \|_{L^2} \big )\,.
\end{align}
Hence, \eqref{proof.prop.mos.h.1} and \eqref{proof.prop.mos.h.2} yield
\begin{align}\label{proof.prop.mos.h.3}
\| \partial_Y \Phi_0 \|_{L^2} + \alpha \| \Phi_0 \|_{L^2} \leq \frac{C}{\alpha \Im c_\epsilon} \| f \|_{L^2}\,,
\end{align}
where $C$ is a universal constant. 
Then \eqref{proof.prop.mos.h.2} and \eqref{proof.prop.mos.h.3} imply 
$\big | \langle h, \Phi_0\rangle _{L^2}\big | \leq \frac{C}{\alpha^2 \Im c_\epsilon} \| f \|_{L^2}^2$, 
which leads to, from \eqref{proof.prop.mos.h.1},
\begin{align}\label{proof.prop.mos.h.4}
\| \partial_Y^2 \Phi_0 \|_{L^2} \leq C \big ( \frac{n}{\alpha^2 \Im c_\epsilon} \big ) ^\frac12 \| f \|_{L^2}\,,
\end{align}
and therefore,
\begin{align}\label{proof.prop.mos.h.5}
\|( \partial_Y^2 -\alpha^2) \Phi_0 \|_{L^2}\leq C \bigg ( \big ( \frac{n}{\alpha^2 \Im c_\epsilon} \big ) ^\frac12 + \frac{1}{\Im c_\epsilon} \bigg ) \| f \|_{L^2} \leq C \big ( \frac{n}{\alpha^2 \Im c_\epsilon} \big ) ^\frac12 \| f\|_{L^2}
\end{align}
when $\gamma \geq \frac23$ in \eqref{parameter.stability} and $n \leq \delta_0^{-1} \nu^{-\frac34}$, in virtue of $\alpha = n \nu^\frac12$. 

We will now construct the solution $\Phi_{mOS}[h]$ to \eqref{eq.os} of the form $\Phi_{mOS}[h] = \Phi_0[f] + \Phi_1[f]$, where $\Phi_1[f]$ should solve \eqref{eq.os} with $h$ replaced by $h_1 = (\pa_Y U) \partial_Y \Phi_0$. By Corollary \ref{cor.prop.mos} the function $\Phi_1[f]$ is of the form $\Phi_1[f] = \Phi_{Ray} [h_1] + \tilde \Phi_1[f]$ with the estimate
\begin{align}\label{proof.prop.mos.h.6}
\| (\partial_Y^2 -\alpha^2) \tilde \Phi_1 [f] \|_{L^2} + \| \partial_Y \tilde \Phi_1 [f] \|_{L^2} + \alpha \| \tilde \Phi_1 [f] \|_{L^2} \leq C \| \partial_Y^2 \Phi_{Ray}[h_1] \|_{L^2}\,,
\end{align}
where $C$ is a universal constant. Thus it suffices to estimate $\varphi = \Phi_{Ray}[h_1]$.
To this end we decompose $\varphi=\varphi_1+\varphi_2$, where $\varphi_1$ is the solution to \eqref{eq.os} with $h$ replaced by $\chi_{\{0\leq Y\leq \sigma\}} h_1=\chi_{\{0\leq Y\leq \sigma\}} \pa_Y U\,  \partial_Y \Phi_0$ with sufficiently small $\sigma \in (0,1]$, while $\varphi_2$ is the solution to \eqref{eq.os} with $h$ replaced by $\chi_{\{Y\geq \sigma\}}h_1$. The number $\sigma\in (0,1]$ is chosen so that $U(Y)$ is well approximated by the linear function $\pa_Y U|_{Y=0} Y$ for $Y\in [0,\sigma]$. Note that $\pa_Y U|_{Y=0} >0$.
Then $\varphi_1$ satisfies the estimates in Lemma \ref{lem.Ray.0}, and we have  
\begin{align*}
|\langle \frac{\chi_{\{0\leq Y\leq \sigma\}} \pa_Y U\, \partial_Y\Phi_0}{V-c_\epsilon}, \varphi_1 \rangle _{L^2}|
& \leq \|\frac{Y^\frac12 \chi_{\{ 0 \leq Y\leq \sigma\}} \pa_Y U\, \partial_Y \Phi_0}{V-c_\epsilon}\|_{L^1} \| \partial_Y \varphi_1\|_{L^2} \\
& \leq C \| \frac{\chi_{\{0\leq Y\leq \sigma\}}}{V-c_\epsilon} \|_{L^2} \| \pa_Y \Phi_0 \|_{L^2} \| \partial_Y \varphi_1 \|_{L^2}\\
& \leq \frac{C}{(\Im c_\epsilon)^\frac12} \| \pa _Y \Phi_0\|_{L^2} \| \pa_Y \varphi_1 \|_{L^2}\,.
\end{align*}
Here we have used the fact that $V$ is well approximated by the linear function $\pa_Y V|_{Y=0} Y\approx \pa_Y U|_{Y=0} Y$ for $Y\in [0,\sigma]$ and for sufficiently small $\nu$.
Then we have from \eqref{est.lem.Ray.0.2},
\begin{align}\label{proof.prop.mos.h.7}
\| \partial_Y\varphi_1 \|_{L^2} + \alpha \| \varphi_1 \|_{L^2} \leq \frac{C}{(\Im c_\epsilon)^\frac32} \| \partial_Y \Phi_0 \|_{L^2}\leq \frac{C}{\alpha (\Im c_\epsilon)^\frac52} \| f\|_{L^2}\,, 
\end{align}
where $C$ depends only on $U^E$ and $U$, while \eqref{est.lem.Ray.0.3} implies
\begin{align}\label{proof.prop.mos.h.8}
\| (\partial_Y^2 -\alpha^2)\varphi_1\|_{L^2}\leq C\big ( \frac{1}{(\Im c_\epsilon)^\frac32} \| \partial_Y \Phi_0 \|_{L^2} + \frac{1}{\Im c_\epsilon} \| \partial_Y \Phi_0 \|_{L^2} \big ) \leq \frac{C}{\alpha (\Im c_\epsilon)^\frac52} \| f\|_{L^2}\,.
\end{align}
The estimate of $\varphi_2$ follows from Lemma \ref{lem.nodeg'}, and we have 
\begin{align}
& \| (\partial_Y^2 - \alpha^2)\varphi_2 \|_{L^2} + \|\partial_Y \varphi_2 \|_{L^2} + \alpha \| \varphi_2\|_{L^2} \nonumber \\
& \qquad \leq C \big ( \frac{1}{\Im c_\epsilon} \|\frac{\chi_{\{Y\geq \sigma\}} \pa_Y U\, \pa_Y  \Phi_0}{\sqrt{-\pa_Y^2 U}} \|_{L^2}  + \| \frac{\chi_{\{Y\geq \sigma\}} \pa_Y U\, \pa_Y \Phi_0}{V-c_\epsilon}\|_{L^2}\big )   \nonumber \\
& \qquad \leq \frac{C}{\Im c_\epsilon} \| \pa_Y \Phi_0 \|_{L^2} \leq \frac{C}{\alpha (\Im c_\epsilon)^2} \| f\|_{L^2}\,.\label{proof.prop.mos.h.9}
\end{align} 
Here we have used the concave condition in \eqref{concave.weak}.
Collecting \eqref{proof.prop.mos.h.3} - \eqref{proof.prop.mos.h.9}, we obtain \eqref{est.prop.mos.h.1} and \eqref{est.prop.mos.h.2}. The proof is complete.
\end{proofx}

\subsubsection{Construction of slow mode for modified Orr-Sommerfeld equations}\label{subsubsec.slow}

The goal of this subsection is to construct a slow mode for the modified Orr-Sommerfeld equations, 
which is a solution to the boundary value problem
\begin{equation}\label{eq.slow}
\left\{
\begin{aligned}
mOS (\phi) \, & = \, 0 \,, \qquad Y>0\,,\\
\phi \, & = \, 1\,, \qquad Y=0\,.
\end{aligned}\right.
\end{equation}
The slow mode is constructed around the solution to the Rayleigh equations.
To this end we first consider the boundary value problem 
\begin{equation}\label{eq.ray.b}
\left\{
\begin{aligned}
(V -c_\epsilon) (\partial_Y^2-\alpha^2) \varphi - (\partial_Y^2 V ) \varphi \, & = \, 0 \,, \qquad Y>0\,,\\
\varphi \, & = \, 1\,, \qquad Y=0\,.
\end{aligned}\right.
\end{equation}

\begin{prop}\label{prop.Ray.b} Let $0<\nu\leq 1$ and $|c|\leq \delta_1^{-1}$. Assume that \eqref{concave.weak} holds. Let $\delta\in (0,\delta_*]$ in \eqref{parameter.stability}. 
Then there exists a unique solution $\varphi_{Ray}\in H^2 (\R_+)$ to \eqref{eq.ray.b} of the form $\varphi_{Ray} = e^{-\alpha Y} + \tilde \varphi_{Ray}$ satisfying
\begin{align}
\| \partial_Y \tilde \varphi_{Ray} \|_{L^2} + \alpha \| \tilde \varphi_{Ray} \|_{L^2} + \| (\partial_Y^2 -\alpha^2)\tilde \varphi_{Ray}\|_{L^2} \leq \frac{C}{\Im c_\epsilon}\,.\label{est.prop.Ray.b.1}
\end{align}  
Here $C$ depends only on $\|U^E\|_{C^2}$ and $\|U\|$.

\end{prop}

\begin{proof} The function $\tilde \varphi_{Ray}$ has the form $\tilde \varphi_{Ray} =\tilde \varphi_{Ray,1}+\tilde \varphi_{Ray,2}$, where $\tilde \varphi_{Ray,1}$ is the solution to \eqref{eq.ray.half} with $h (Y)=h_1 (Y) = \nu (\partial_y^2 U^E) (\nu^\frac12 Y) e^{-\alpha Y}$ and  $\tilde \varphi_{Ray,2}$ is the solution to \eqref{eq.ray.half} with $h (Y)=h_2 (Y) = e^{-\alpha Y} \pa_Y^2 U (Y)$. Then, in virtue of Lemma \ref{lem.Ray.0} we have 
\begin{align*}
& \| \partial_Y \tilde \varphi_{Ray,1}\|_{L^2}^2 + \alpha^2 \| \tilde \varphi_{Ray,1}\|_{L^2}^2  + \| (\partial_Y^2 -\alpha^2) \tilde \varphi_{Ray,1}\|_{L^2}^2 \\
& \leq C\big ( \frac{1}{\Im c_\epsilon}  \big | \langle \frac{h_1}{V-c_\epsilon}, \tilde \varphi_{Ray,1} \rangle _{L^2} \big | + \| \frac{h_1}{V-c_\epsilon} \|_{L^2}^2\big )\,.
\end{align*}
Here $C$ depends only on $\|U^E\|_{C^2}$ and $\|U\|$. We observe that 
\begin{align*}
 \frac{1}{\Im c_\epsilon}  \big | \langle \frac{h_1}{V-c_\epsilon}, \tilde \varphi_{Ray,1} \rangle _{L^2} \big | \leq \frac{\nu\| U^E\|_{C^2}}{(\Im c_\epsilon)^2} \| e^{-\alpha Y} \|_{L^2} \| \tilde \varphi_{Ray,1}\|_{L^2} \leq \frac{C \nu\| U^E\|_{C^2}}{\alpha^\frac12 (\Im c_\epsilon)^2} \| \tilde \varphi_{Ray,1}\|_{L^2} \,,
\end{align*}
and 
\begin{align*}
\| \frac{h_1}{V-c_\epsilon} \|_{L^2}^2 \leq \frac{\nu^2 \|U^E\|_{C^2}^2}{(\Im c_\epsilon)^2} \| e^{-\alpha Y}\|_{L^2}^2 \leq \frac{C\nu^2 \|U^E\|_{C^2}^2}{\alpha (\Im c_\epsilon)^2}\,,
\end{align*}
which imply
\begin{align*}
\| \partial_Y \tilde \varphi_{Ray,1}\|_{L^2}^2 + \alpha^2 \| \tilde \varphi_{Ray,1}\|_{L^2}^2 + \| (\partial_Y^2 -\alpha^2) \tilde \varphi_{Ray,1}\|_{L^2}^2  \leq C\big ( \frac{\nu^2 \| U^E \|_{C^2}^2}{\alpha^3 (\Im c_\epsilon)^4} +\frac{\nu^2 \|U^E\|_{C^2}^2}{\alpha (\Im c_\epsilon)^2} \big )\,.
\end{align*}
Recalling the lower bound $\alpha \Im c_\epsilon \geq \delta_1^{-1} n^{\gamma}\nu^\frac12$ assumed in \eqref{parameter.stability}, we thus obtain 
\begin{align*}
\| \partial_Y \tilde \varphi_{Ray,1}\|_{L^2}^2 + \alpha^2 \| \tilde \varphi_{Ray,1}\|_{L^2}^2 + \| (\partial_Y^2 -\alpha^2) \tilde \varphi_{Ray,1}\|_{L^2}^2  \leq \frac{C}{\Im c_\epsilon}\,.
\end{align*}
Next we see from Lemma \ref{lem.nodeg'},
\begin{align*}
\| \partial_Y \tilde \varphi_{Ray,2}\|_{L^2} +\alpha \| \tilde \varphi_{Ray,2} \|_{L^2} + \| (\partial_Y^2 -\alpha^2)\tilde \varphi_{Ray,2}\|_{L^2} & \leq \frac{C}{\Im c_\epsilon} \| \sqrt{-\pa_Y^2 U} e^{-\alpha Y} \|_{L^2} \\
& \leq \frac{C}{\Im c_\epsilon} \| \sqrt{-\pa_Y^2 U} \|_{L^2} \leq \frac{C}{\Im c_\epsilon}\,,
\end{align*}
where $C$ depends only on $U^E$ and $U$. Note that we have used the spatial decay such as $|\pa_Y^2 U(Y)|\leq \|U\| (1+ Y)^{-2}$. Collecting these estimates, we obtain \eqref{est.prop.Ray.b.1}, since $|c|\leq \delta_1^{-1}$.
The proof is complete.
\end{proof}

\vspace{0.3cm}

The existence and the estimates for the slow mode $\phi_s$ are stated as follows.
Recall that $\delta_2>0$ is the number in Proposition \ref{prop.large.alpha},
while $\delta_*>0$ is the number in Proposition \ref{prop.mos}.

\begin{prop}[Slow mode for modified Orr-Sommerfeld equations]\label{prop.slow}
Let $0<\nu \leq \nu_0$, $|c|\leq \delta_1^{-1}$, and $\alpha \Im c_\epsilon \leq \delta_2^{-1}$. 
Assume that \eqref{concave.weak} holds. 
Let $\gamma\in [\frac57,1]$ and $\delta\in (0,\delta_*]$ in \eqref{parameter.stability}. 
Then there exists a solution $\phi_s\in H^2 (\R_+)$ to \eqref{eq.slow} of the form 
$\phi_s = \varphi_{Ray} + \tilde \phi_s$ 
satisfying the estimate
\begin{align}
\| \pa_Y \tilde \phi_s \|_{L^2}  + \alpha \| \tilde \phi_s \|_{L^2} + \| (\pa_Y^2-\alpha^2) \phi_s \|_{L^2} & \leq \frac{C}{\Im c_\epsilon} \,.
\end{align}
Here  $C$ depends only on $U^E$ and $U$. Moreover, if \eqref{concave.strong} holds in addition then the above statement is valid for $\gamma\in [\frac23,1]$.
\end{prop}

\begin{proof} Since the function $\tilde \phi_s$ satisfies the equations \eqref{eq.os.R'} with the source term
\begin{align*}
\epsilon (\partial_Y^2-\alpha^2) \partial_Y^2 \varphi_{Ray} \, = \,
\epsilon (\partial_Y^2-\alpha^2) \partial_Y^2 (e^{-\alpha Y} + \tilde \varphi_{Ray} ) \, = \, \epsilon (\partial_Y^2-\alpha^2 ) \partial_Y^2 \tilde \varphi_{Ray}\,,
\end{align*}
the estimates of $\tilde \phi_s$ follows from Proposition \ref{prop.mos} and Remark \ref{rem.prop.mos} by combining  the estimate 
\begin{align*}
\| \partial_Y^2 \tilde \varphi_{Ray} \|_{L^2} \leq  \frac{C}{\Im c_\epsilon} \,,
\end{align*}
which is proved in Proposition \ref{prop.Ray.b}. The proof is complete.
\end{proof}

\begin{rem}\label{rem.prop.slow}{\rm In virtue of Proposition \ref{prop.large.alpha} and Corollary \ref{cor.prop.large.alpha},  the construction of the slow/fast modes  is needed only in the case $0<\alpha \Im c_\epsilon \leq \delta_2^{-1}$. In this regime of $\alpha$, Propositions \ref{prop.Ray.b} and \ref{prop.slow} lead to the estimate for the slow mode $\phi_s$ such as
\begin{align}
\|\partial_Y ( \phi_s - e^{-\alpha Y} )\|_{L^2} +\alpha \| \phi_s - e^{-\alpha Y}  \|_{L^2} + \| (\pa_Y^2-\alpha^2)  ( \phi_s - e^{-\alpha Y} )\|_{L^2}  \leq \frac{C}{\Im c_\epsilon}\,,\label{est.rem.prop.slow}
\end{align}
by recalling the form $\phi_s = \varphi_{Ray} + \tilde \phi_s=e^{-\alpha Y} + \tilde \varphi_{Ray} + \tilde \phi_s$. Note that the constant $C$ in \eqref{est.rem.prop.slow} depends only on $U^E$ and $U$.
In particular, we have 
\begin{align}
|\partial_Y \phi _s (0)|\leq \frac{C}{\Im c_\epsilon} \qquad {\rm if}~~\alpha \Im c_\epsilon \leq \delta_2^{-1}\,.
\end{align}}
\end{rem}

\subsubsection{Construction of fast mode for modified Orr-Sommerfeld equations}\label{subsubsec.fast}

In this subsection we construct another solution to \eqref{eq.slow} possessing a boundary layer structure,
called the fast mode due to the rapid dependence on the parameter $\epsilon$.
The result is stated as follows.
\begin{prop}[Fast mode for modified Orr-Sommerfeld equations]\label{prop.fast}
Let $0<\nu \leq \nu_0$, $|c|\leq \delta_1^{-1}$, and $\alpha \Im c_\epsilon\leq \delta_2^{-1}$.
Assume that \eqref{concave.weak} holds.  Let $\delta\in (0,\delta_*]$ in \eqref{parameter.stability}.
If  $\gamma\in [\frac57,1]$ in \eqref{parameter.stability}
then there exists a solution $\phi_f\in H^2 (\R_+)$ to \eqref{eq.slow}  satisfying the estimates
\begin{align}
\| \partial_Y^k \phi_f\|_{L^2} \leq C |\frac{c_\eps}{\epsilon}|^{\frac{k}{2}-\frac14}\,, \qquad k=0,1,2\,,\label{est.prop.fast.1}
\end{align}
for all sufficiently small $\nu>0$. Moreover,  
\begin{align}
|\partial_Y \phi_f (0)|\geq \frac{1}{C} |\frac{c_\eps}{\epsilon}|^\frac12\,.\label{est.prop.fast.2}
\end{align}
\noindent 
Here  $C$ depends only on $U^E$ and $U$. 
Moreover, if \eqref{concave.strong} holds in addition then the above statement is valid for $\gamma\in [\frac23,1]$.
\end{prop}

The construction of the fast mode becomes a delicate issue when the critical point $Y_c$ is close to the boundary $Y=0$. By taking into account the monotonicity of $U$ such a situation corresponds to the case $|\Re c| \ll 1$. For convenience we shall treat separately  the following cases.
Fix small $\theta \in (0,\frac{1}{10})$.

\vspace{0.3cm}

\noindent 
(1) Case $|c|\leq |\epsilon|^{\frac{1-\theta}{3}}\ll 1$: in this case the critical layer has to be taken into account in the analysis,

\noindent 
(2) Case $|c| \geq |\epsilon|^{\frac{1-\theta}{3}}$: in this case the effect of the critical layer is negligible.

\begin{proofx}{Proposition \ref{prop.fast} for case (1) $|c|\leq |\epsilon|^{\frac{1-\theta}{3}}\ll 1$} 
We recall that $V(Y)=U^E(\sqrt{\nu} Y)-U^E(0) + U(Y)$ and $\pa_Y U>0$ near  $Y = 0$.
Let us introduce the odd extension of $V$ to $\R$: $V(Y)=-V(-Y)$ for $Y<0$. 
Then $V$ belongs to $W^{2,\infty}(\R)$ and, for sufficiently small $\nu$,
we have $\pa_Y V\geq \kappa$ for $Y\in (-1,1)$. 
Here the number $\kappa$ is positive and can be taken uniformly in small $\nu$.
We may assume that there exists a unique $Y_c \in (-1, 1)$ such that $V(Y_c) = \Re c$ holds.
Moreover, $Y_c$ satisfies $Y_c \simeq \Re c$ and $Y_c=0$ if and  only if $\Re c=0$. 
In particular, we can take a constant $C^*\geq 1$ such that 
\begin{align}
| Y_c| \leq C^* |\Re c| \leq C^* |\epsilon|^{\frac{1-\theta}{3}}\,.\label{bound.Y_c}
\end{align}
Then we rewrite the term $V- c_\epsilon$ as
\begin{align*}
V (Y) - c_\epsilon  & \, = \, \pa_Y V|_{Y=Y_c} ( Y-Y_c ) + R (Y) - i \Im c_\epsilon \nonumber \\
& \, =: \, \pa_Y V|_{Y=Y_c} \big ( Y - Z_c \big ) + R (Y) \,.
\end{align*}
Here we have set 
\begin{align} \label{def_Zc}
Z_c \, = \, Y_c + i \frac{ \Im c_\epsilon}{\pa_Y V|_{Y=Y_c}} \,,
\end{align}
and the remainder term $R$ satisfies the estimate 
\begin{align}
|R(Y)|\leq C \| V\|_{C^2(\R_+)} |Y-Y_c|^2\,, \qquad Y\geq 0\,.\label{bound.R}
\end{align}
The key idea is then to build the fast mode $\phi_f$ around a solution $\psi_{Ai}$ to
\begin{equation} \label{eq_psiAi}
-\epsilon \pa^2_Y (\pa^2_Y - \alpha^2) \psi_{Ai} + \pa_Y V|_{Y=Y_c} (Y - Z_c) (\pa^2_Y - \alpha^2) \psi_{Ai} \, = \, 0\,,
\end{equation}
satisfying $\psi_{Ai}(0) =1$, $\psi_{Ai}(Y) \rightarrow 0, \, Y \rightarrow +\infty$. Therefore, we introduce the classical Airy function $Ai = Ai(z)$ defined by the contour integral
$$Ai(z) \, = \, \int_L \exp\big(zt - \frac{t^3}{3}\big) \dd t $$ 
where $\arg(z) \in (-\pi,\pi)$, and $L : \{r(s) e^{i \varphi(s)}, s \in \R\}$ is any contour satisfying 
\begin{itemize}
\item $\lim_{s \rightarrow +\infty} r(s)  = +\infty$, and in the neigborhood of $s=+\infty$, $\frac{2\pi}{3} \le \varphi(s) \le \frac{2\pi}{3} +   \frac{\pi}{6}$. 
\item $\lim_{s \rightarrow - \infty} r(s) = +\infty$, and in the neigborhood of $s=-\infty$, $-\frac{2\pi}{3}-\frac{\pi}{6} \le \varphi(s) \le -\frac{2\pi}{3}$. 
\end{itemize}
It is well-known that $Ai$ satisfies 
$$ \pa^2_z Ai - z Ai \, = \, 0\,, $$
and decays to zero when $|z| \rightarrow +\infty$ with $|\arg(z)| < \frac{\pi}{3}$. See \cite{Die} for details. 
To fix the idea,  we take the contour 
$L$ defined by $L = L_{-} \cup L_{0} \cup L_{+}$
with 
$$L_{-} \, = \, \{ r e^{-2i\pi/3}, r \in (1,+\infty) \}, \quad L_{0} \, = \, \{  e^{i \theta}, \theta \in [-\frac{4\pi}{3},-\frac{2\pi}{3}] \}, \quad L_{+} \, = \, \{e^{2i\pi/3} r, r \in (1,+\infty) \}$$
(oriented from bottom to top). We then set 
\begin{equation} \label{defAialpha}
 Ai_\alpha(z) \, = \, \int_{L} \frac{\exp(zt - \frac{t^3}{3})}{t^2 - (\tilde \eps^{\frac13} \alpha)^2} \dd t\,, \quad \tilde \eps \, = \, \frac{\eps}{\pa_Y V|_{Y=Y_c}}\,. 
 \end{equation}
Note that, since $\alpha \Im c_\epsilon \leq \delta_2^{-1}$, we see 
\begin{align*}
\tilde \eps^{\frac13} \alpha \, = \, e^{-\frac{\pi}{6}i} n^\frac23 \nu^{\frac12} (\pa_Y V|_{Y=Y_c})^{-\frac13} \quad {\rm and } \quad |\tilde \epsilon^\frac13 \alpha |\leq \frac{|\epsilon|^\frac13}{\delta_2 (\pa_Y V|_{Y=Y_c})^\frac13 \Im c_\epsilon} \ll 1
\end{align*}
under the condition $\gamma\in [\frac23,1]$ with sufficiently small $\delta>0$ in \eqref{parameter.stability}, 
and thus, the poles $\pm \tilde \eps^{1/3} \alpha$ are at the right of  $L$. 
By simple differentiation, one has 
$$ (\pa^2_z - (\tilde \eps^{\frac13} \alpha)^2) Ai_\alpha \, = \, Ai \,. $$
Then we set 
\begin{align}\label{def.psi_Ai}
\psi_{Ai} (Y) \, = \, \frac{1}{Ai_\alpha (-Z_c/\tilde \epsilon^\frac13)} Ai_\alpha (\frac{Y-Z_c}{\tilde \epsilon^\frac13} )\,, 
\end{align}
which clearly solves the equation \eqref{eq_psiAi}. 
We collect in the following lemma a few estimates on $\psi_{Ai}$. 
\begin{lem}\label{lem.psi_Ai} Let $\gamma \in [\frac{2}{3}, 1]$ in \eqref{parameter.stability}. 
Then the function $\psi_{Ai}$ given in  \eqref{def.psi_Ai} is well-defined and satisfies the estimates
\begin{align}\label{est.lem.psi_Ai.1} 
|\partial_Y^k \psi_{Ai} (Y)| \leq C |\tilde \epsilon|^{-\frac{k}{3}} \,  | \frac{Z_c}{\tilde \epsilon^\frac13} |^\frac54 \, |\frac{Y-Z_c}{\tilde \epsilon^\frac13}|^{-\frac{5-2k}{4}}\,  | \exp \bigg (-\frac23 \big ( \frac{Y-Z_c}{\tilde \epsilon^\frac13} \big )^\frac32 + \frac23 \big (-\frac{Z_c}{\tilde \epsilon^\frac13}\big )^\frac32 \bigg ) |\,, 
\end{align}
for all $Y\geq 0$ and $k=0,1,2$, and 
\begin{align}
|\partial_Y \psi_{Ai} (0)|\geq \frac{1}{C} \, |\frac{c_\eps}{\epsilon}|^\frac12\,.\label{est.lem.psi_Ai.2} 
\end{align}
In particular, we have 
\begin{align}\label{est.lem.psi_Ai.3} 
\| \partial_Y^k \psi_{Ai} \|_{L^2}\leq C \, |\frac{c_\eps}{\epsilon} |^{\frac{k}{2}-\frac14}\,, \qquad k=0,1,2  \,. 
\end{align}
\end{lem}
The proof of Lemma \ref{lem.psi_Ai} is postponed to the appendix. 
To achieve the construction of the fast mode, we consider the remainder $\tilde \phi_f =\phi_f-\psi_{Ai}$:  
\begin{equation}\label{eq.fast'}
\left\{
\begin{aligned}
mOS (\tilde \phi_f) \, & = \, h_f, \quad h_f(Y) = - R(Y) \partial_Y^2 \psi_{Ai} + \partial_Y^2 V  \psi_{Ai} \,, \qquad Y>0\,,\\
\phi \, & = \, 0\,, \qquad Y=0\,,
\end{aligned}\right.
\end{equation}
In view of Corollary \ref{cor.prop.mos} and Lemma \ref{lem.Ray.0}, we have 
\begin{equation}
\| \partial_Y \tilde \phi_f \|_{L^2} + \alpha \| \tilde \phi_f   \|_{L^2} + \| (\partial_Y^2-\alpha^2)  \tilde \phi_f  \|_{L^2}  \leq C  \bigg( \, \frac{1}{\Im c_\epsilon} \big | \langle \frac{h_f}{V-c_\epsilon}, \tilde \phi_f \rangle _{L^2} \big |   + \| \frac{h_f}{V-c_\epsilon}\|_{L^2}^2 \bigg )^{1/2}\,. 
\end{equation}
By the Hardy inequality, 
\begin{align*} 
\big | \langle \frac{R \partial_Y^2  \psi_{Ai}}{V-c_\epsilon}, \tilde \phi_f \rangle _{L^2} \big |  
& \le  \frac{C}{\Im c_\epsilon} \| Y R \partial_Y^2 \psi_{Ai} \|_{L^2} \| \partial_Y \tilde \phi_f \|_{L^2}\,, \\
\big | \langle \frac{\partial_Y^2 V \,  \psi_{Ai} }{V-c_\epsilon}, \tilde \phi_f \rangle _{L^2} \big |
&  \le  \frac{C}{\Im c_\epsilon} \big ( \| Y \partial_Y^2 U\,   \psi_{Ai} \|_{L^2} \| \partial_Y \tilde \phi_f \|_{L^2} + \frac{\nu \| U^E \|_{C^2} \|\psi_{Ai} \|_{L^2}}{\alpha} \alpha \| \tilde \phi_f \|_{L^2}\big )\,,
\end{align*}
and 
\begin{align*}
\| \frac{h_f}{V-c_\epsilon} \|_{L^2} \leq \frac{C}{\Im c_\epsilon} \big ( \| R\pa_Y \psi_{Ai} \|_{L^2} + \| \pa_Y^2 V\|_{L^\infty} \| \psi_{Ai}\|_{L^2}\big )\,.
\end{align*}
Thus we end up with 
\begin{align}\label{estim_tildephif}
\begin{split}
& \| \partial_Y \tilde \phi_f \|_{L^2} + \alpha \| \tilde \phi_f   \|_{L^2} + \| (\partial_Y^2-\alpha^2)  \tilde \phi_f  \|_{L^2} \\
&\quad  \leq  
\frac{C}{\Im c_\epsilon} \bigg( \frac{1}{\Im c_\epsilon}  \| Y R \partial_Y^2  \psi_{Ai} \|_{L^2}  + \frac{1}{\Im c_\epsilon} \| Y \psi_{Ai} \|_{L^2} + \| R \partial_Y^2  \psi_{Ai} \|_{L^2}  +  \| \psi_{Ai} \|_{L^2}  \bigg)  
\end{split}
\end{align}
We focus on the control of the terms $\| Y R \partial_Y^2  \psi_{Ai} \|_{L^2}$ and $\| R \partial_Y^2  \psi_{Ai} \|_{L^2}$ at the right-hand side, the other two being similar. We see
\begin{align*}
\| R \partial_Y^2 \psi_{Ai} \|_{L^2} & \leq C \big ( \int_0^\infty \big (Y-Y_c)^4 |\partial_Y^2 \psi_{Ai} (Y) |^2 \dd Y \big )^\frac12 \nonumber \\
& \leq C \big ( \int_0^{2C^* |\epsilon|^\frac{1-\theta}{3}} \big (Y-Y_c)^4 |\partial_Y^2 \psi_{Ai} (Y) |^2 \dd Y \big )^\frac12 \nonumber \\
& \quad + C \big ( \int_{2C^*|\epsilon|^\frac{1-\theta}{3}}^\infty \big (Y-Y_c)^4 |\partial_Y^2 \psi_{Ai} (Y) |^2 \dd Y \big )^\frac12 \nonumber \\
& =: I + II.
\end{align*} 
Since $|Y_c|\leq C |\epsilon|^{\frac{1-\theta}{3}}$ as stated in \eqref{bound.Y_c}, the term $I$ is estimated as 
\begin{align*}
I & \leq C |\epsilon|^{\frac23 (1-\theta)} \| \pa_Y^2 \psi_{Ai} \|_{L^2}\leq C |\epsilon|^{\frac23 (1-\theta)} |\frac{c_\eps}{\epsilon}|^\frac34\,,
\end{align*}
where \eqref{est.lem.psi_Ai.3} is used. 
As for $II$, we use the pointwise estimates
\begin{align}
|\partial_Y^2 \psi_{Ai} (Y) |\leq C |\tilde \epsilon|^{-\frac23} |\frac{Z_c}{\tilde \epsilon^\frac13}|^\frac54 |\frac{Y-Z_c}{\tilde \epsilon^\frac13}|^{-\frac14} e^{-\frac{1}{C} |\frac{Z_c}{\tilde \epsilon}|^\frac12 Y}\,,\label{proof.prop.fast.1}
\end{align}
see \eqref{proof.est.lem.psi.Ai.3.3} in the appendix for the proof. 
Then, since $|\epsilon|\leq C (\Im c_\epsilon)^3$ holds and $|Z_c| \simeq |c_\eps|$, $|\tilde \epsilon| \simeq |\epsilon|$, the term $II$ is small exponentially in $|\epsilon|^{-\frac{\theta}{3}}$. In particular, we have 
\begin{align*}
II\leq C_{\theta,N} |\epsilon|^N\,,
\end{align*}
for large $N\geq 1$. Hence, it follows that 
\begin{equation}
\|R \partial_Y^2 \psi_{Ai} \|_{L^2} \le  C  |\epsilon|^{\frac23 (1-\theta)} |\frac{c_\eps}{\epsilon}|^\frac34\,.\label{proof.prop.fast.2}
\end{equation}
To estimate $\| Y R \partial_Y^2  \psi_{Ai} \|_{L^2}$, we write 
$$ \| Y R \partial_Y^2  \psi_{Ai} \|_{L^2} \, \le \,  \| (Y - Y_c)  R \partial_Y^2  \psi_{Ai} \|_{L^2} + |Y_c| \| R \partial_Y^2  \psi_{Ai} \|_{L^2}, $$
and obtain, by the similar computation as above,
\begin{align*}
\|Y R \partial_Y^2  \psi_{Ai} \|_{L^2} & \le C |\epsilon|^{1-\theta} |\frac{c_\eps}{\epsilon}|^\frac34\,.
\end{align*}
Note that $|c_\epsilon|\leq |c| + |\epsilon| \alpha^2 = |c|+n\nu \leq C |c|$ since $n\nu\leq C \Im c$ for $n\leq \delta_0^{-1}\nu^{-\frac34}$ and $\gamma\in [\frac23,1]$ in \eqref{parameter.stability}.
Then, back to \eqref{estim_tildephif}, still for small enough $\theta$,  we end up with 
\begin{align}
& \| \partial_Y \tilde \phi_f \|_{L^2} + \alpha \| \tilde \phi_f   \|_{L^2} + \| (\partial_Y^2 -\alpha^2) \tilde \phi_f  \|_{L^2} \nonumber \\
& \quad \le  \frac{C}{\Im c_\epsilon} \big (  \frac{1}{\Im c_\epsilon} |\epsilon|^{1-\theta} |\frac{c_\eps}{\epsilon}|^\frac34 + \frac{1}{\Im c_\epsilon} |\epsilon|^{\frac13(1-\theta)} |\frac{\epsilon}{c_\eps}|^\frac14  +  |\epsilon|^{\frac23(1-\theta)} |\frac{c_\eps}{\epsilon}|^\frac34 + |\frac{\epsilon}{c_\eps}|^\frac14 \big ) \nonumber \\
& \quad \leq \frac{C}{\Im c_\epsilon} \bigg ( \frac{|\epsilon|^{\frac12-\frac54 \theta}}{\Im c_\epsilon} + (\frac{|\epsilon|^\frac13}{\Im c_\epsilon})^\frac54 |\epsilon|^{\frac16-\frac{1}{3}\theta}+ |\epsilon|^{\frac16-\frac{11}{12}\theta} + |\epsilon|^{\frac16}\bigg ) \nonumber \\
& \quad \leq \frac{C}{\Im c_\epsilon} |\epsilon|^{\frac16-\frac54\theta} \,,\label{proof.prop.fast.3}
\end{align}
which by the Sobolev embedding implies
\begin{align}
 |\partial_Y \tilde \phi_f(0)| \le \frac{C}{\Im c_\eps} |\epsilon|^{\frac16-\frac54\theta}\,.\label{proof.prop.fast.4}
\end{align}
On the other hand, by \eqref{est.lem.psi_Ai.2} the function $\partial_Y \psi_{Ai}$ has the lower bound at $Y=0$ as follows:
\begin{align}
|\partial_Y \psi_{Ai} (0)| &\geq \frac{1}{C} |\frac{c_\eps}{\epsilon}|^\frac12\,.\label{proof.prop.fast.5}
\end{align}
We recall again the condition $\gamma\in [\frac23,1]$ in \eqref{parameter.stability}, which ensures $|\epsilon|\ll (\Im c_\epsilon)^3$. Then, combining \eqref{proof.prop.fast.4} and \eqref{proof.prop.fast.5},  for small enough $\eps$ (that is large enough $n$), we end up with 
\begin{equation} \label{lowerbound_fastmode}
 |\partial_Y \phi_f(0)| \ge \frac{1}{C} |\frac{c_\eps}{\epsilon}|^\frac12\,.
 \end{equation}
 This lower bound on the derivative of the fast mode at the boundary $Y=0$ will be important when solving the Orr-Sommerfeld equation. From \eqref{est.lem.psi_Ai.3} and \eqref{proof.prop.fast.3} we obtain \eqref{est.prop.fast.1} for $\phi_f = \psi_{Ai} + \tilde \phi_f$. The estimate \eqref{est.prop.fast.2} is just obtained by \eqref{lowerbound_fastmode}. 
The proof of Proposition \ref{prop.fast} for the case (1) is complete.
\end{proofx}

\begin{proofx}{Proposition \ref{prop.fast} for case (2) $|c|\geq |\epsilon|^{\frac{1-\theta}{3}}$}
In this case the construction of the fast mode is more straightforward than the case (1).
Set 
\begin{align}\label{proof.prop.fast.6}
\tau_\epsilon \, = \, \big (\frac{-c_\epsilon}{\epsilon}\big )^\frac12 \, = \, \big ( \frac{\Im c_\epsilon - i \Re c_\epsilon}{|\epsilon|} \big )^\frac12\,, \qquad \omega_\epsilon \, = \, \frac{\tau_\epsilon}{|\tau_\epsilon|}\,.
\end{align}
The root is taken so that $\Re \tau_\epsilon>0$, and since $\Im c_\epsilon>0$ we see
\begin{align}\label{proof.prop.fast.7}
\Re \tau_\epsilon\geq \frac{1}{C} |\frac{c_\epsilon}{\epsilon}|^\frac12 = \frac{|\tau_\epsilon|}{C}
\end{align}
for some universal constant $C>0$. 
Then we look for a solution $\phi_f$ to \eqref{eq.slow} of the form
\begin{align}\label{proof.prop.fast.8}
\phi_f  (Y) \, = \, \sum_{k=0}^N \phi_k (|\tau_\epsilon|Y) + R (Y)\,,
\end{align}
where the leading profile $\phi_0$ is given by 
\begin{align}\label{proof.prop.fast.9}
\phi_0 (z) \, = \, e^{-\omega_\epsilon z}\,,
\end{align}
and $\phi_k$, $k=1,\cdots,N$, are profiles of boundary layer type, and $R$ is a small remainder.
The number $N$ will be taken large enough.
The profile $\phi_k$ is built inductively as the solution to 
\begin{align}\label{proof.prop.fast.10}
\begin{split}
& -\pa_z^4 \phi_k +\omega_\epsilon^2\pa_z^2 \phi_k \\
& \quad \, = \, - (\frac{\alpha}{|\tau_\epsilon|})^2  \pa_z^2\phi_{k-1} - \frac{1}{\epsilon |\tau_\epsilon|^2}  V^{(\epsilon)} \pa_z^2 \phi_{k-1} +\frac{\alpha^2}{\epsilon |\tau_\epsilon|^4}  (V^{(\eps)} -c_\epsilon ) \phi_{k-1} + \frac{1}{\epsilon |\tau_\epsilon|^4}  (\pa_Y^2 V)^{(\eps)} \phi_{k-1}
\end{split}
\end{align}
satisfying $\phi_k=0$ on $z=0$, where $f^{(\eps)} (z) = f(\frac{z}{|\tau_\epsilon|})$ for any $f=f(Y)$.
Set 
$$\pa_z^{-1} f \, = \, -\int_z^\infty f(z') \dd z'\,,\qquad \pa_z^{-2} f \, = \, \int_z^\infty \int_{z'}^\infty f(z'') \dd z'' \dd z'\,.$$
Then the right-hand side of \eqref{proof.prop.fast.10} is written as 
\begin{align}\label{proof.prop.fast.11}
\begin{split}
& \pa_z^{2} \bigg ( - (\frac{\alpha}{|\tau_\epsilon|})^2 \phi_{k-1} - \frac{1}{\epsilon |\tau_\epsilon|^2} V^{(\eps)}  \phi_{k-1}
+ \pa_z^{-1} \big ( \frac{2}{\epsilon |\tau_\epsilon|^3} ({\pa_Y V})^{(\eps)} \phi_{k-1} \big )
\\
& \qquad \qquad \qquad + \pa_z^{-2} \big ( \frac{\alpha^2}{\epsilon |\tau_\epsilon|^4} (V^{(\eps)} -c_\eps) \phi_{k-1} \big ) \bigg ) \, =: \, \pa_z^2 g_k\,.
\end{split}
\end{align}
As a result, it suffices to solve 
\begin{align}\label{proof.prop.fast.12}
-\pa_z^2 \phi_k + \omega_\epsilon^2 \phi_k \, = \, g_k \,, \quad z>0\,, \qquad \phi_k|_{z=0} =0\,.
\end{align}
Thus, $\phi_k$ is expressed as 
\begin{align}\label{proof.prop.fast.13}
\phi_k (z) \, = \, \int_0^z e^{-\omega_\epsilon (z-\xi')} \int_{\xi'}^\infty e^{-\omega_\epsilon (\xi''-\xi')} g_k (\xi'')\dd \xi'' \dd \xi'\,.
\end{align}
By using the estimate
\begin{align*}
|V^{(\eps)}(z)|\leq \frac{z}{|\tau_\epsilon|}\|\pa_Y V \|_{L^\infty_Y}
\end{align*}
each $g_k$ is estimated as, for any $0<\delta<\Re \omega_\epsilon$, 
\begin{align}\label{proof.prop.fast.14}
\sup_{z>0} \, e^{\delta z} |g_k (z) |\leq C D_\epsilon \sup_{z>0} \, (1+z) e^{\delta z} |\phi_{k-1} (z)|\,,
\end{align}
where 
\begin{align}\label{proof.prop.fast.15}
D_\epsilon & \, = \, \frac{\alpha^2}{|\tau_\epsilon|^2} + \frac{1}{|\epsilon| |\tau_\epsilon|^3} \| \pa_Y V \|_{L^\infty_Y} + \frac{\alpha^2}{|\epsilon| |\tau_\epsilon|^5} \| \pa_Y V \|_{L^\infty_Y} + \frac{\alpha^2}{|\tau_\epsilon|^4} |\frac{c_\epsilon}{\eps}| \nonumber \\
& \, = \, 2 \frac{\alpha^2}{|\tau_\epsilon|^2} + \frac{1}{|\epsilon| |\tau_\epsilon|^3} \| \pa_Y V \|_{L^\infty_Y} + \frac{\alpha^2}{|\epsilon| |\tau_\epsilon|^5} \| \pa_Y V \|_{L^\infty_Y}\,.
\end{align}  
Let us recall the conditions  $\alpha \Im c_\epsilon\leq \delta_2^{-1}$ and $|c|\geq |\epsilon|^{\frac{1-\theta}{3}}$. In particular, $|\tau_\epsilon|^{-1} =|\frac{\epsilon}{c_\eps}|^\frac12 \leq |\epsilon|^{\frac13+\frac{\theta}{6}}$ holds.
Then $D_\epsilon$ is estimated as
\begin{align}\label{proof.prop.fast.16}
D_\epsilon\leq C\big ( \frac{|\epsilon|^{\frac23+\frac{\theta}{3}}}{\Im c_\epsilon^2} + |\epsilon|^{\frac{\theta}{2}} \big ) \leq C |\epsilon|^\frac{\theta}{3}\,, 
\end{align}
under the condition $|\epsilon|\leq (\Im c_\epsilon)^3$, which is valid for $\gamma\in [\frac23,1]$ in \eqref{parameter.stability}. From \eqref{proof.prop.fast.13} and \eqref{proof.prop.fast.14} it is straightforward to see
\begin{align*}
|\phi_k (z)| & \leq  \int_0^z e^{-\Re (\omega_\eps) (z-\xi')}\int_{\xi'}^\infty e^{-\Re (\omega_\eps) (\xi''-\xi')} |g_k (\xi'')| \dd \xi'' \dd \xi'\\
& \leq  \int_0^z e^{-\Re (\omega_\eps) (z-\xi')}\int_{\xi'}^\infty e^{-\Re (\omega_\eps) (\xi'' -\xi')-\delta \xi''} \dd \xi'' \dd \xi' \sup_{\xi''>0} |e^{\delta \xi''} g_k (\xi'')|\\
& \leq \frac{1}{\Re \omega_\eps + \delta} \int_0^z e^{-\Re (\omega_\eps) (z-\xi') -\delta \xi'} \dd \xi'\sup_{\xi''>0} |e^{\delta \xi''} g_k (\xi'')|\\
& \leq \frac{C D_\eps e^{-\delta z}}{(\Re \omega_\eps+\delta)(\Re \omega_\eps  - \delta)}   \sup_{z>0} | (1+z) |e^{\delta z} \phi_{k-1} (z)|\,.
\end{align*}
Here we have used the condition $0<\delta<\Re \omega_\eps$.
Similarly, we obtain
\begin{align}\label{proof.prop.fast.17}
\sup_{z>0} \, e^{\delta z} \sum_{j=0}^2 |\pa_z^j \phi_k (z) | & \leq C D_\epsilon \sup_{z>0} \, (1+z) e^{\delta z} |\phi_{k-1} (z)| \nonumber \\
& \leq C_{\delta,k} D_\epsilon^k \,,
\end{align}
for any $0<\delta<\Re \omega_\epsilon$ and $k=1,2,\cdots,N$, where $C_{\delta,k}$ depends only on $\delta$ and $k$.
The remainder $R$ in \eqref{proof.prop.fast.8} is then defined as a solution to 
\begin{align*}
mOS (R) \, = \, h \,, \quad Y>0\,, \qquad R|_{Y=0} \, = \, 0
\end{align*}
with 
\begin{align*}
h \, = \, \big ( -\epsilon \alpha^2 |\tau_\epsilon|^2 \pa_z^2 \phi_N - |\tau_\eps|^2 V^{(\eps)} \pa_z^2\phi_N + (V^{(\eps)} -c_\eps)\alpha^2 \phi_N +(\pa_Y^2 V)^{(\eps)} \phi_N \big ) (|\tau_\eps| Y)\,.
\end{align*}
In virtue of Corollary \ref{cor.prop.mos} we can take $R$ such that 
\begin{align}\label{proof.prop.fast.18}
\begin{split}
& \| \pa_Y R \|_{L^2} + \alpha \| R \|_{L^2} + \| (\pa_Y^2-\alpha^2) R \|_{L^2} \\
& \qquad \qquad \leq C \big ( \| \pa_Y \Phi_{Ray} [h] \|_{L^2} + \alpha \| \Phi_{Ray}[h]\|_{L^2}  + \| (\pa_Y^2 -\alpha^2 )\Phi_{Ray}[h] \|_{L^2} \big )
\end{split}
\end{align}
if $\alpha \Im c_\epsilon \leq \delta_2^{-1}$. Thus it suffices to consider the estimates of $\Phi_{Ray}[h]$.
By using 
$$|\langle \frac{h}{V-c_\epsilon}, \varphi\rangle _{L^2} |\leq C \| \frac{Y h}{V-c_\epsilon} \|_{L^2}\| \pa_Y \varphi \|_{L^2}\,, \qquad \varphi \in H^1_0 (\R_+)\,,$$
Lemma \ref{lem.Ray.0} implies that the right-hand side of \eqref{proof.prop.fast.18} is bounded from above by 
\begin{align}\label{proof.prop.fast.19}
\frac{C}{\Im c_\epsilon} \big ( \| \frac{Y h}{V-c_\epsilon} \|_{L^2} +  \| h\|_{L^2} \big )\,.
\end{align} 
Recalling the definition of $h$, we first observe that 
\begin{align}\label{proof.prop.fast.20}
\|\frac{Y h}{V-c_\epsilon} \|_{L^2} & \leq \frac{C}{\Im c_\epsilon} \| Y h \|_{L^2} \nonumber \\
& \leq \frac{C}{\Im c_\epsilon} \bigg ( |\epsilon| \alpha^2 |\tau_\epsilon|^\frac12 \| z\pa_z^2 \phi_N \|_{L^2_z} + \|V\|_{L^\infty} |\tau_\epsilon|^\frac12 \| z \pa_z^2 \phi_N \|_{L^2_z} \nonumber \\
& \quad +  \alpha^2 \| \pa_Y V\|_{L^\infty} |\tau_\epsilon|^{-\frac52} \| z^2\phi_N \|_{L^2_z}  + \alpha^2 |c_\epsilon| |\tau_\epsilon|^{-\frac32} \| z\phi_N \|_{L^2} + \|\pa_Y^2 V \|_{L^\infty} |\tau_\epsilon|^{-\frac32} \| z\phi_N \|_{L^2_z} \bigg )\nonumber \\
& \leq \frac{C_N}{\Im c_\epsilon} \big ( |\tau_\epsilon|^\frac12 + \alpha^2 |\tau_\epsilon|^{-\frac52} + (\alpha^2 |c_\epsilon| + 1) |\tau_\epsilon|^{-\frac32}  \big ) D_\epsilon^N\,,
\end{align}
since $|\epsilon| \alpha^2=n\nu\leq 1$ by our assumption and \eqref{proof.prop.fast.17}. Similarly, we have 
\begin{align}\label{proof.prop.fast.21}
\| h \|_{L^2} & \leq C_N \big ( |\tau_\epsilon|^{\frac32} + \alpha^2 |\tau_\epsilon|^{-\frac32}  + (\alpha^2 |c_\epsilon| + 1) |\tau_\epsilon|^{-\frac12} \big ) D_\epsilon^N\,.
\end{align}
Since $\Im c_\epsilon\geq |\tau_\epsilon|^{-1}$ and $\alpha^2 \leq |\frac{c}{\eps}|=|\tau_\eps|^2$ under the conditions $\alpha \Im c_\eps \leq \delta_2^{-1}$ and $\gamma\in [\frac23,1]$ in \eqref{parameter.stability} we  
have from \eqref{proof.prop.fast.20} and \eqref{proof.prop.fast.21} applied for \eqref{proof.prop.fast.19},
\begin{align}\label{proof.prop.fast.22}
\| \pa_Y R \|_{L^2} + \alpha \| R \|_{L^2} + \| (\pa_Y^2-\alpha^2) R \|_{L^2} \leq C |\tau_\eps|^\frac52 D_\epsilon^N \leq C |\eps|^{\frac{N\theta}{3}-\frac54}\ll 1\,,
\end{align}
if $N$ is taken so that $N>\frac{15}{4\theta}$.
Collecting \eqref{proof.prop.fast.16}, \eqref{proof.prop.fast.17}, and \eqref{proof.prop.fast.22}, we have constructed the fast mode $\phi_f (Y)$ of the form \eqref{proof.prop.fast.8} and $\phi_f$ satisfies 
\begin{align}\label{proof.prop.fast.23}
\|\pa_Y^k  \big (\phi_f - \phi_0 (|\tau_\epsilon| \cdot ) \big ) \|_{L^2}  & \leq C |\eps|^{\frac{\theta}{3}} |\frac{c_\eps}{\epsilon}|^{\frac{k}{2}-\frac14}\,,
\end{align} 
if $N$ is taken large enough. Comparing the estimate of $\phi_0 (|\tau_\epsilon|Y)$, we obtain \eqref{est.prop.fast.1} and \eqref{est.prop.fast.2}. The proof for the case (2) is complete.
\end{proofx}

\subsubsection{Construction of resolvent}\label{subsubsec.resolvent}

Let us go back to the analysis of the boundary value problem for the modified Orr-Sommerfeld equations \eqref{eq.os'} with the inhomogeneous term $h=- f_{2,n} + \frac{1}{i\alpha} \partial_Y f_{1,n}$.
We look for the solution $\phi$ to \eqref{eq.os'} of the form 
\begin{align}
\phi \, = \, \Phi_{mOS} [h] \, + \, \tilde \phi\,,\label{def.phi.re}
\end{align}
where $\Phi_{mOS}[h]$ is the solution to \eqref{eq.os.R} constructed in Proposition \ref{prop.mos.h}.
Then $\tilde \phi$ should satisfy the homogeneous problem
\begin{equation}\label{eq.os'.h}
\left\{
\begin{aligned}
 -\epsilon  (\partial_Y^2 - \alpha^2 ) \partial_Y^2 \tilde \phi + (V -  c_\epsilon ) (\partial_Y^2-\alpha^2) \tilde \phi - (\partial_Y^2 V ) \tilde \phi    \, = \, 0 \,, \qquad & Y>0\,,\\
 \tilde \phi \, = \, 0 \,, \qquad  \partial_Y \tilde \phi    \, = \, -\partial_Y \Phi_{mOS} [h]  \,, \qquad  & Y=0\,.
\end{aligned}\right.
\end{equation}
In virtue of Propositions \ref{prop.slow} and \ref{prop.fast}, 
the solution $\tilde \phi$ to \eqref{eq.os'.h} is  written as
\begin{align}
\tilde \phi \, = \, A \phi_s + B \phi_f\,,
\end{align}
and the coefficients $A$ and $B$ are determined by the equation
\begin{equation}
\begin{pmatrix}
\phi_s (0) & \phi_f(0)\\
\partial_Y \phi_s (0) & \partial_Y \phi_f (0)
\end{pmatrix}
\begin{pmatrix}
A\\
B
\end{pmatrix}
\, = \, 
\begin{pmatrix}
0 \\
-\partial_Y \Phi_{mOS} [h]  (0)
\end{pmatrix}\,.
\end{equation}
Recall that $\phi_s (0) = \phi_f (0)=1$. Then, $A$ and $B$ are determined under the condition $\partial_Y \phi_f (0)-\partial_Y \phi_s (0) \ne 0$, and we have the formula
\begin{align*}
\begin{pmatrix}
A\\
B
\end{pmatrix}
& \, = \, \frac{1}{\partial_Y \phi_f (0)-\partial_Y \phi_s (0) }
\begin{pmatrix}
 \partial_Y \phi_f (0) & -1\\
-\partial_Y \phi_s (0) & 1
\end{pmatrix}
\begin{pmatrix}
0 \\
-\partial_Y \Phi_{mOS} [h]  (0)
\end{pmatrix}\nonumber \\
& \, = \, \frac{\partial_Y \Phi_{mOS} [h]  (0)}{\partial_Y \phi_f (0)-\partial_Y \phi_s (0) }
\begin{pmatrix}
1 \\
-1
\end{pmatrix}\,.
\end{align*}
Hence, $\tilde \phi$ is given by 
\begin{align}\label{def.tilde.phi.re}
\tilde \phi \, = \, \frac{\partial_Y \Phi_{mOS} [h]  (0)}{\partial_Y \phi_f (0)-\partial_Y \phi_s (0) } \, \big ( \phi_s - \phi_f \big )\,. 
\end{align}
The solvability condition $\partial_Y \phi_f (0)-\partial_Y \phi_s (0) \ne 0$ is ensured by Propositions \ref{prop.slow} and \ref{prop.fast} for $\gamma\geq \frac57$ under \eqref{concave.weak}, while for $\gamma\geq \frac23$ under \eqref{concave.strong}.
Collecting \eqref{def.phi.re} and \eqref{def.tilde.phi.re}, we have the following proposition.
Let us recall that we are interested in the case when $h$ in \eqref{eq.os'} is of the form $h=-f_{2,n} + \frac{1}{i\alpha} \pa_Y f_{1,n}$ for $f=(f_{1,n},f_{2,n})\in L^2 (\R_+)^2$.
 
\begin{prop}\label{prop.resolvent} Assume that \eqref{concave.weak} holds. 

\vspace{0.1cm}

\noindent {\rm (i)} Let $0<\nu\leq \nu_0$, $|c|\leq \delta_1^{-1}$, and $\alpha \Im c_\epsilon \leq \delta_2^{-1}$, where $\delta_2$ is the number in Proposition \ref{prop.large.alpha}. 
Then there exists $\delta_{**} \in (0,\delta_*]$ such that the following statement holds.
If $\gamma \in [\frac57,1]$ and $\delta\in (0,\delta_{**}]$ in \eqref{parameter.stability}, then for any $f=(f_{1,n},f_{2,n})\in L^2 (\R_+)^2$ there exists a weak solution $\phi\in H^2_0 (\R_+)$ to the Orr-Sommerfeld equations \eqref{eq.os'}, and  $\phi$ satisfies the estimates
\begin{align}
\| \partial_Y \phi \|_{L^2} + \alpha \| \phi \|_{L^2} & \leq \frac{C}{\alpha (\Im c_\epsilon)^\frac52}  \| f \|_{L^2} \,, \label{est.prop.resolvent.1}\\
\| (\partial_Y^2-\alpha^2) \phi \|_{L^2} & \leq \frac{C}{\alpha (\Im c_\epsilon)^\frac52} \bigg ( \frac{1}{|\epsilon|^\frac14} + \frac{\Im c_\epsilon}{|\epsilon|^\frac12} \bigg ) \| f\|_{L^2}\,.\label{est.prop.resolvent.2}
\end{align}
Here $C$ depends only on $U^E$ and $U$.
If \eqref{concave.strong} holds in addition, then the above statement is valid for $\gamma\in [\frac23,1]$ with the factors $(\Im c_\epsilon)^{-\frac52}$ and $\frac{\Im c_\epsilon}{|\epsilon|^\frac12}$ replaced by $(\Im c_\epsilon)^{-2}$ and $\frac{(\Im c_\epsilon)^\frac34}{|\epsilon|^\frac12}$, respectively.

\vspace{0.1cm}

\noindent {\rm (ii)} Let $\alpha \Im c_\epsilon \geq \delta_2^{-1}$. Then for any $f=(f_{1,n},f_{2,n})\in L^2 (\R_+)^2$ there exists a unique weak solution $\phi\in H^2_0 (\R_+)$ to the Orr-Sommerfeld equations \eqref{eq.os'}, and  $\phi$ satisfies the estimates
\begin{align}
\| \partial_Y \phi \|_{L^2} + \alpha \| \phi \|_{L^2} & \leq \frac{C}{\alpha \Im c_\epsilon} \| f \|_{L^2} \,, \label{est.prop.resolvent.3}\\
\| (\partial_Y^2-\alpha^2) \phi \|_{L^2} & \leq \frac{C}{\alpha (|\epsilon| \Im c_\epsilon)^\frac12} \| f \|_{L^2}\,.\label{est.prop.resolvent.4}
\end{align}
Here $C$ is a universal constant. 
\end{prop}

\begin{rem}\label{rem.prop.resolvent}{\rm (i) By the standard elliptic regularity the solution $\phi$ in Proposition \ref{prop.resolvent} belongs to $H^3(\R_+)$. In fact, the uniqueness of weak solutions is available also for the case (i) of Proposition \ref{prop.resolvent} by applying the method of continuity. Indeed, following the proof of Proposition \ref{prop.resolvent} below, we can also show the existence of the weak solution $\phi\in H^2_0 (\R_+)$ to \eqref{eq.os'} for arbitrary $h\in L^2(\R_+)$, which satisfies the similar $H^2$ estimate as in \eqref{est.prop.resolvent.1} and \eqref{est.prop.resolvent.2}: in fact, under the conditions on the parameters of Proposition \ref{prop.resolvent} (i), one can show 
\begin{align}
\| \pa_Y \phi \|_{L^2} + \alpha \|\phi\|_{L^2}  \leq \frac{C}{\alpha (\Im c_\epsilon)^2} \| h\|_{L^2}\,, \qquad \| (\pa_Y^2-\alpha^2) \phi \|_{L^2} \leq \frac{C}{\alpha(\Im c_\epsilon)^2} \big ( \frac{1}{|\epsilon|^\frac14} + \frac{1}{|c_\eps|^\frac12} \big ) \| h\|_{L^2} \,.
\end{align}
Moreover, $\phi$ belongs to $H^4(\R_+)$ by the elliptic regularity. 
Since the uniqueness is available at least for the case $\alpha \Im c_\epsilon = \delta_2^{-1}$ in virtue of (ii) of Proposition \ref{prop.resolvent}, one can apply the method of continuity  stated in Proposition \ref{prop.continuity} in the appendix for the operator $mOS=-\epsilon (\pa_Y^2-\alpha^2)\pa_Y^2 + (V-c_\epsilon)(\pa_Y^2-\alpha^2) - (\pa_Y^2V)$  with the domain $H^4(\R_+)\cap H^2_0 (\R_+)$, which shows the uniqueness of weak solutions in the case (i). Although we do not give the details here, we use this argument for the operator $\mathbb{L}_{\nu,n}$, rather than $mOS$, in the next subsection, since it is sufficient for our purpose.

\noindent 
(ii) Recalling Remark \ref{rem.prop.mos.h}, we can show Proposition \ref{prop.resolvent} under the condition $-M_\sigma \pa_Y^2U\geq (\pa_Y U)^4$, rather than $-M_\sigma \pa_Y^2 U\geq (\pa_Y U)^2$ in \eqref{concave.weak}, however, the factor $(\Im c_\epsilon)^{-\frac52}$ in \eqref{est.prop.resolvent.1} and \eqref{est.prop.resolvent.2} is replaced by $(\Im c_\epsilon)^{-3}$ under this weaker condition. 
The similar remark is applied for Corollary \ref{cor.prop.resolvent} below.
}
\end{rem}

\begin{proofx}{Proposition \ref{prop.resolvent}} (i) We may assume that $\delta \in (0,\delta_2]$. 
As we have seen in the beginning of this subsection, if  $\partial_Y \phi_f (0)-\partial_Y \phi_s (0) \ne 0$ then the weak solution $\phi$ to \eqref{eq.os'} exists, which is written as \eqref{def.phi.re} with $\tilde \phi$ as in \eqref{def.tilde.phi.re}. Proposition \ref{prop.mos.h} implies that 
\begin{align*}
\| \partial_Y \Phi_{mOS} [h] \|_{L^2} + \alpha \| \Phi_{mOS} [h] \|_{L^2} & \leq \frac{C}{\alpha (\Im c_\epsilon)^\frac52} \| f\|_{L^2}\,,\\
\| (\partial_Y^2 -\alpha^2) \Phi_{mOS}[h] \|_{L^2} & \leq \frac{C}{\alpha} \big ( \frac{1}{(\Im c_\epsilon)^\frac52} + \frac{1}{(|\epsilon| \Im c_\epsilon)^\frac12} \big ) \| f\|_{L^2}\,. 
\end{align*}
The condition $\partial_Y \phi_f (0)-\partial_Y \phi_s (0) \ne 0$ is satisfied if $\gamma\in [\frac23,1]$ and $\delta$ in \eqref{parameter.stability} is sufficiently small, since Proposition \ref{prop.fast} implies $|\partial_Y \phi_f (0)|\geq \frac{1}{C} |\frac{c_\eps}{\epsilon}|^\frac12 \geq \frac{1}{C} |\epsilon|^{-\frac13}$, while Proposition \ref{prop.slow} and Remark \ref{rem.prop.slow} show $|\partial_Y \phi_s (0)|\leq \frac{1}{C} (\Im c_\epsilon)^{-1}$. In particular, we have $|\partial_Y \phi_f (0)-\partial_Y \phi_s (0)|\geq \frac{1}{C} |\frac{c_\eps}{\epsilon}|^\frac12$
if $\gamma\in [\frac23,1]$ and $\delta$ is sufficiently small. Hence, Propositions \ref{prop.slow}, \ref{prop.fast}, and Remark \ref{rem.prop.slow} yield
\begin{align*}
\|\partial_Y \tilde \phi \|_{L^2} + \alpha \| \tilde \phi \|_{L^2} & \leq C |\frac{\epsilon}{c_\eps}|^\frac12 \big ( \frac{1}{\Im c_\epsilon} +  |\frac{c_\eps}{\epsilon}|^\frac14 + \alpha |\frac{\epsilon}{c_\eps}|^{\frac14}  \big )  | \partial_Y \Phi_{mOS}[h](0)|\\
& \leq C  |\frac{\epsilon}{c_\eps}|^\frac12 \big ( \frac{1}{\Im c_\epsilon} +  |\frac{c_\eps}{\epsilon}|^\frac14 \big )  | \partial_Y \Phi_{mOS}[h](0)|\,,
\end{align*}
by the assumption  $\alpha\leq C (\Im c_\epsilon)^{-1}$ and $|\frac{\epsilon}{c}|\leq C$. 
Similarly, we have  
\begin{align*}
\| (\partial_Y^2 -\alpha^2) \tilde \phi\|_{L^2}  & \leq C  |\frac{\epsilon}{c_\eps}|^\frac12 \big ( \frac{1}{\Im c_\epsilon} + |\frac{c_\eps}{\epsilon}|^{\frac34} + \alpha^2 |\frac{\epsilon}{c_\eps}|^{\frac14} \big )  | \partial_Y \Phi_{mOS}[h](0)| \\
& \leq C \big ( |\frac{c_\eps}{\epsilon}|^\frac14 + \frac{1}{|c_\eps|^\frac12} \big )  | \partial_Y \Phi_{mOS}[h](0)|\,,
\end{align*}
where we have again used $\alpha\leq C (\Im c_\epsilon)^{-1}$ and $|\epsilon|\leq C (\Im c_\epsilon)^3$.
The value  $| \partial_Y \Phi_{mOS}[h](0)|$ is estimated by the interpolation as
\begin{align*}
| \partial_Y \Phi_{mOS}[h](0)| \leq \frac{C}{\alpha (\Im c_\epsilon)^\frac52} \big ( 1+ |\epsilon|^{-\frac12} (\Im c_\epsilon)^2 \big )^\frac12 \| f\|_{L^2}\,.
\end{align*}
Collecting these above, since $\phi = \Phi_{mOS}[h]+\tilde \phi_f$, we have arrived at 
\begin{align*}
\| \partial_Y \phi \|_{L^2} + \alpha \| \phi \|_{L^2} & \leq  \frac{C}{\alpha (\Im c_\epsilon)^\frac52} \bigg ( 1 +|\frac{\epsilon}{c_\eps}|^\frac12 \big ( \frac{1}{\Im c_\epsilon} + |\frac{c_\eps}{\epsilon}|^{\frac14} \big) \big ( 1+ |\epsilon|^{-\frac14} \Im c_\epsilon \big ) \bigg ) \| f \|_{L^2} \\
& \leq  \frac{C}{\alpha (\Im c_\epsilon)^\frac52} \| f\|_{L^2}\,,
\end{align*}
since $|\epsilon|\leq C (\Im c_\epsilon)^3$. Similarly, 
\begin{align*}
\| (\partial_Y^2 -\alpha^2) \phi\|_{L^2}  & \leq \frac{C}{\alpha (\Im c_\epsilon)^\frac52} \bigg ( 1 + \frac{(\Im c_\epsilon)^2}{|\epsilon|^\frac12} +  (1+\frac{\Im c_\epsilon}{|\epsilon|^{\frac14}})  \big ( |\frac{c_\eps}{\epsilon}|^\frac14 + \frac{1}{|c_\eps|^\frac12}\big ) \bigg ) \| f\|_{L^2}\\
& \leq \frac{C}{\alpha (\Im c_\epsilon)^\frac52} \bigg ( \frac{1}{|\epsilon|^\frac14} +  \frac{\Im c_\epsilon}{|\epsilon|^{\frac12}}  \bigg ) \| f\|_{L^2}\,.
\end{align*}
Here we have used $|\frac{c_\eps}{\epsilon}|\leq C |\epsilon|^{-1}$ and $|c_\eps|^{-\frac12}\leq |\epsilon|^{-\frac14}$. Thus, \eqref{est.prop.resolvent.1} and \eqref{est.prop.resolvent.2} follow. 

\noindent 
(ii) In the case $\alpha \Im c_\epsilon \geq \delta_2^{-1}$ we can apply Proposition \ref{prop.large.alpha} and obtain \eqref{est.prop.resolvent.3} and \eqref{est.prop.resolvent.4}. The proof is complete.
\end{proofx}

\begin{cor}\label{cor.prop.resolvent} Assume that \eqref{concave.weak} holds.

\vspace{0.1cm}

\noindent {\rm (i)} Let $\delta_0^{-1}\leq |n|\leq \delta_0^{-1}\nu^{-\frac34}$, $0<\nu\leq \nu_0$, $\gamma\in [\frac57,1]$, and $\delta\in (0,\delta_{**}]$.
Then the set 
\begin{align}
O_{\nu,n} \, = \, \big \{ \mu \in \C~ \big |~|\mu|\leq \frac{n\nu^\frac12}{\delta_1} \,, \quad \Re \mu \geq \frac{n^\gamma \nu^\frac12}{\delta} \big \}\label{est.cor.prop.resolvent.1}
\end{align}
is included in the resolvent set of $-\mathbb{L}_{\nu,n}$.
Moreover, if $\mu \in  O_{\nu,n}$ satisfies $\Re \mu = \frac{n^\gamma \nu^\frac12}{\delta}$ and $\Re \mu +n^2\nu^\frac32 \leq \delta_2^{-1}$, then
\begin{align}
\| (\mu + \mathbb{L}_{\nu,n} )^{-1}  f \|_{L^2(\Omega_\nu)} & \leq  \frac{C n^{\frac32 (1-\gamma)}}{\Re \mu}  \| f\|_{L^2(\Omega_\nu)} \,, \quad f\in \mathcal{P}_n^{(\nu)} L^2_\sigma (\Omega_\nu)\,, \label{est.cor.prop.resolvent.2}\\
\| \nabla (\mu + \mathbb{L}_{\nu,n} )^{-1} f \|_{L^2(\Omega_\nu)}  & \leq  \frac{C n^{\frac32(1-\gamma)}}{\Re \mu} \big (n^\frac14 + n^{\frac12-(1-\gamma)}  \big ) \| f\|_{L^2(\Omega_\nu)} \,, \quad f\in \mathcal{P}_n^{(\nu)} L^2_\sigma (\Omega_\nu) \,.\label{est.cor.prop.resolvent.3}
\end{align}
If \eqref{concave.strong} holds in addition, then the above statement is valid for $\gamma\in [\frac23,1]$ with the factors $n^{\frac32(1-\gamma)}$ and $n^{\frac12 - (1-\gamma)}$ replaced by $n^{1-\gamma}$ and $n^{\frac12-\frac34(1-\gamma)}$, respectively.

\vspace{0.1cm}

\noindent {\rm (ii)} If $\Re \mu +n^2\nu^\frac32 \geq \delta_2^{-1}$ and $\Re \mu >0$ then $\mu$ belongs to the resolvent set of $-\mathbb{L}_{\nu,n}$, and the following estimates hold.
\begin{align}
\| (\mu + \mathbb{L}_{\nu,n} )^{-1}  f \|_{L^2(\Omega_\nu)} & \leq  \frac{C}{\Re \mu} \| f\|_{L^2(\Omega_\nu)} \,, \quad f\in \mathcal{P}_n^{(\nu)} L^2_\sigma (\Omega_\nu)\,, \label{est.cor.prop.resolvent.4}\\
\| \nabla (\mu + \mathbb{L}_{\nu,n} )^{-1} f \|_{L^2(\Omega_\nu)}  & \leq  \frac{C}{\nu^\frac14 (\Re \mu)^\frac12}  \| f\|_{L^2(\Omega_\nu)} \,, \quad f\in \mathcal{P}_n^{(\nu)} L^2_\sigma (\Omega_\nu) \,.\label{est.cor.prop.resolvent.5}
\end{align}
\end{cor}

\begin{proof} The case (ii) is already proved in Corollary \ref{cor.prop.large.alpha}, so we focus on the case (i):
We may assume that $\frac{n^\gamma \nu^\frac12}{\delta}\leq \frac{n\nu^\frac12}{\delta_1}$, otherwise the set $O_{\nu,n}$ is empty and there is nothing to be proved.
It is clear that the set $O_{\nu,n}$ is arcwise connected in $\mathbb{C}$. By Proposition \ref{prop.resolvent} (i),  for any $\mu\in O_{\nu,n}$ and any $f\in \mathcal{P}_n^{(\nu)}L^2_\sigma (\Omega_\nu)$ there exists a weak solution $\phi\in H_0^2(\R_+)$ to \eqref{eq.os'} (and thus, to \eqref{eq.os}), which is $H^3(\R_+)$ by the elliptic regularity and gives the solution $v$ (as in the proof of Corollary \ref{cor.prop.general'}) to \eqref{eq.resolvent} with a suitable pressure $q$. 
From \eqref{est.prop.resolvent.1} and \eqref{est.prop.resolvent.2} the norms $\| v\|_{L^2(\Omega_\nu)}$ and $\|\nabla v \|_{L^2(\Omega_\nu)}$ are estimated so that the constants in the estimates are uniform in $O_{\nu,n}$. Moreover, if $|\mu_0|=\delta_1^{-1} n\nu^\frac12$ and $\Re \mu_0 = \delta^{-1} n^\gamma\nu^\frac12$ then $\mu_0\in O_{\nu,n}$ but also $\mu_0 \in \rho (-\mathbb{L}_{\nu,n})$ from \eqref{proof.cor.prop.general'.0} in the proof of  Corollary \ref{cor.prop.general'}. 
Then we can apply the method of continuity in Proposition \ref{prop.continuity}, which shows $O_{\nu,n}\subset \rho (-\mathbb{L}_{\nu,n})$.
The estimates \eqref{est.cor.prop.resolvent.1} and \eqref{est.cor.prop.resolvent.2} follow from \eqref{est.prop.resolvent.1} and \eqref{est.prop.resolvent.2}; indeed, the equality $\Re \mu = \delta^{-1}n^\gamma \nu^\frac12$ implies that $\Im c = \delta^{-1} n^{\gamma-1}$, and thus, 
$$\frac{n^{\gamma-1}}{\delta}\leq \Im c_\epsilon \, = \, \frac{n^{\gamma-1}}{\delta} + n\nu \leq C n^{\gamma-1}$$
if $\gamma\in [\frac23,1]$ and $n\leq \delta_0^{-1} \nu^{-\frac34}$. 
The proof is complete.
\end{proof}

\begin{rem}{\rm The resolvent estimate is available for all $\mu\in O_{\nu,n}$ also in (i), 
but  for later use we need the estimate only on the borderline $\Re \mu =\delta^{-1} n^\gamma \nu^\frac12$.
On the borderline $\Re \mu=\delta^{-1} n^\gamma \nu^\frac12$ the conditions (i) and (ii) in Corollary \ref{cor.prop.resolvent} are respectively written as 

\vspace{0.3cm}

\noindent 
(i) $n^{\gamma} \nu^\frac12 + \delta n^2\nu^\frac32\leq \delta \delta_2^{-1}$~~~~~~~ (ii) $n^{\gamma} \nu^\frac12 + \delta n^2\nu^\frac32 \geq \delta \delta_2^{-1}$\,.

\vspace{0.3cm}

\noindent 
When $n\leq \delta_0^{-1}\nu^{-\frac34}$ and $\gamma\in [\frac23,1]$ the condition $n^{\gamma}\nu^\frac12 \gg \delta n^2\nu^\frac32$ is satisfied. Therefore, the case (i) essentially correspond to the regime $n^{\gamma}\nu^\frac12 \leq \mathcal{O} (1)$ when $n\leq \delta_0^{-1}\nu^{-\frac34}$ and $\gamma\in [\frac23,1]$. On the other hand, the case (ii) is always satisfied in the regime $n^{\gamma}\nu^\frac12 \geq \mathcal{O}(1)$.
}
\end{rem}

\subsection{Estimate for semigroup}\label{subsec.semigroup}

The resolvent estimates established in Corollaries \ref{cor.prop.general'} and \ref{cor.prop.resolvent} lead to the estimates for the semigroup $e^{-\tau\mathbb{L}_{\nu,n}}$, and hence, by going back to the original variables, we obtain the estimates for the semigroup $e^{-t \mathbb{A}_{\nu,n}}$ as follows. 

\begin{thm}\label{thm.semigroup}  Assume that \eqref{concave.weak} holds. Then the following estimates hold for all $f\in \mathcal{P}_n L^2_\sigma (\Omega)$ and $t>0$. 

\vspace{0.1cm}

\noindent {\rm (i)} Let $\delta_0^{-1}\leq |n|\leq \delta_0^{-1}\nu^{-\frac34}$, $\gamma \in [\frac57,1]$, and $\delta\in (0,\delta_{**}]$. If in addition $|n|^{\gamma} \nu^\frac12 < 1$, then 
\begin{align}
\| e^{- t \mathbb{A}_{\nu,n}}  f\|_{L^2 (\Omega)} & \leq C |n|^{\frac52(1-\gamma)}  e^{\frac{|n|^\gamma}{\delta}t} \| f\|_{L^2 (\Omega)} \,, \label{est.thm.semigroup.1}\\
\| \nabla e^{- t \mathbb{A}_{\nu,n}}  f\|_{L^2 (\Omega)} & \leq \frac{C}{\nu^\frac12} \bigg ( \frac{1}{t^\frac12} +  |n|^{\frac52(1-\gamma)} \big ( |n|^\frac14 +  |n|^{\frac12 - (1-\gamma)}  \big ) e^{\frac{|n|^\gamma}{\delta}t} \bigg ) \| f\|_{L^2 (\Omega)} \,.\label{est.thm.semigroup.2}
\end{align}
Here $C$ depends only on $U^E$ and $U$.  If \eqref{concave.strong} holds in addition, then \eqref{est.thm.semigroup.1} and \eqref{est.thm.semigroup.2} are valid for $\gamma\in [\frac23,1]$ with the factors $|n|^{\frac52(1-\gamma)}$ and $|n|^{\frac12-(1-\gamma)}$ replaced by $|n|^{2(1-\gamma)}$ and $|n|^{\frac12-\frac34(1-\gamma)}$, respectively.

\vspace{0.1cm}

\noindent {\rm (ii)} Let $|n|^{\gamma} \nu^\frac12 \geq 1$. Then 
\begin{align}
\| e^{- t \mathbb{A}_{\nu,n}}  f\|_{L^2 (\Omega)} & \leq C  |n|^{1-\gamma} e^{\frac{|n|^\gamma}{\delta}t}\| f\|_{L^2 (\Omega)}\,,\label{est.thm.semigroup.3}\\
\| \nabla e^{- t \mathbb{A}_{\nu,n}}  f\|_{L^2 (\Omega)} & \leq \frac{C}{\nu^\frac12} \big ( \frac{1}{t^\frac12} + |n|^{1 -\frac{\gamma}{2}} e^{\frac{|n|^\gamma}{\delta}t} \big ) \| f\|_{L^2 (\Omega)} \,.\label{est.thm.semigroup.4}
\end{align}
Here $C$ depends only on $\|U^E\|_{C^2}$ and $\|U\|$.
\end{thm}

\begin{rem}\label{rem.thm.semigroup}{\rm (i) As is mentioned in Remark \ref{rem.parameter}, the best possible value of $\gamma$ in the temporal growth estimates  in (i) of Theorem \ref{thm.semigroup} is $\gamma=\frac23$. Our result achieves this optimal value at least under the strong concave condition \eqref{concave.strong} for the boundary layer profile $U$. In fact, \eqref{concave.strong} can be slightly relaxed as $-M \pa_Y^2 U\geq (\pa_Y U)^4$, but instead, the factor $|n|^{\frac52(1-\gamma)}$ of the derivative loss is replaced by $|n|^{3(1-\gamma)}$ under this condition; see Remark \ref{rem.prop.resolvent} (ii). 

\noindent (ii) The  result for the case (ii) of Theorem \ref{thm.semigroup} is based only on Corollary \ref{cor.prop.general'} and Proposition \ref{prop.large.alpha}. Hence, it holds for $\gamma \in (0,1]$ and for any profile $U$ satisfying \eqref{bound.U} without concave conditions.
}
\end{rem}
\begin{proofx}{Theorem \ref{thm.semigroup}} By the rescaling \eqref{def.scaling} we have 
$$(e^{-t\mathbb{A}_{\nu,n}} f  )(x,y) \, = \, (e^{-\tau \mathbb{L}_{\nu,n}} f^{(\nu)} ) (X,Y)$$
with $\tau=t\nu^{-\frac12}$ and $f^{(\nu)} (X,Y) = f (\nu^\frac12 X, \nu^\frac12 Y)$. Note that we may assume that $n$ is positive without loss of generality.
By the general perturbation theory we have already known that $-\mathbb{L}_{\nu,n}$ generates a $C_0$-analytic semigroup acting on $\mathcal{P}_n^{(\nu)} L^2_\sigma (\Omega_\nu)$, 
and in particular, from Proposition \ref{prop.general.sg.short} we already have the following estimate
\begin{align}\label{proof.thm.semigroup.1}
\| e^{-\tau \mathbb{L}_{\nu,n}} g \|_{L^2(\Omega_\nu)}  \leq e^{(\| \partial_y U^E \|_{L^\infty(\R_+)} + 2 n \| Y \partial_Y U \|_{L^\infty(\R_+)} )\nu^\frac12 \tau}  \| g \|_{L^2 (\Omega_\nu)} \,, \qquad g \in \mathcal{P}_n^{(\nu)} L^2_\sigma (\Omega_\nu)\,,
\end{align} 
which is useful in the short time interval $0< \tau \leq \nu^{-\frac12} n^{-1}=\alpha^{-1}$.
Hence it suffices to consider the case $\tau \geq \nu^{-\frac12} n^{-1}=\alpha^{-1}$.
Let $S_{\nu,n} (\theta)$ be the set defined in Corollary \ref{cor.prop.general'}.
From the proof of Corollary \ref{cor.prop.general'} we have already known that the set 
$\cup_{\mu\in E_{\nu,n}} B_{r_\mu} (\mu)$ with $E_{\nu,n} = \big \{ \mu\in \C~|~ \Re \mu \geq \delta_1^{-1} n^\gamma \nu^{\frac12}\,, ~ |\mu | \geq \delta_1^{-1} n \nu^\frac12 \big \}$ is included in the resolvent set of $-\mathbb{L}_{\nu,n}$, and $S_{\nu,n} (\theta)\subset \cup_{\mu\in E_{\nu,n}} B_{r_\mu} (\mu)$ holds.
On the other hand, in virtue of Corollary \ref{cor.prop.resolvent} the set  $O_{\nu,n} = \big \{ \mu \in \C~ \big |  ~ \Re \mu \geq \delta^{-1} n^\gamma \nu^\frac12\,, ~|\mu|\leq \delta_1^{-1} n\nu^\frac12 \big \}$ is also included in the resolvent set of $-\mathbb{L}_{\nu,n}$.
Hence we conclude that the set 
\begin{align}\label{proof.thm.semigroup.2}
\Sigma_{\nu,\gamma}(\theta) \, = \, S_{\nu,n} (\theta) \, \cup \, \big \{ \mu \in \C~|~\Re \mu \geq \frac{n^\gamma\nu^\frac12}{\delta} \big \} 
\end{align}
is included in the resolvent set of $-\mathbb{L}_{\nu,n}$.
The estimates in Corollary \ref{cor.prop.general'} for $|\Im \mu|\gg 1$ ensures the representation of the semigroup such as 
\begin{align}\label{proof.thm.semigroup.3}
e^{-\tau \mathbb{L}_{\nu,n}} \, = \, \frac{1}{2\pi i} \int_\Gamma e^{\tau \mu} (\mu + \mathbb{L}_{\nu,n} )^{-1} \dd \mu\,,
\end{align}
where the curve $\Gamma$ is oriented counterclockwise and is taken as 
\begin{align*}
\Gamma & \, = \, \Gamma_+ \, + \, \Gamma_- \, + \, l_+ \, + \, l_- \, + \, l_0 
\end{align*}
with
\begin{align}\label{proof.thm.semigroup.4}
\begin{split}
\Gamma_{\pm} & \, = \, \big \{ \mu\in \C ~|~ \pm \Im \mu  =   (\tan \theta )\Re \mu +  \delta_1^{-1} (\alpha + |\tan\theta| n^\gamma\nu^\frac12 )\,, ~\Re \mu \leq 0  \big \}\,,\\
l_{\pm} & \, = \, \big \{ \mu \in \C ~|~ \pm \Im \mu = \delta_1^{-1} (\alpha + |\tan\theta| n^\gamma\nu^\frac12 ) \,, ~ 0\leq \Re \mu \leq \frac{n^\gamma \nu^\frac12}{\delta} \big \}\,,\\
l_0 & \, = \,  \big \{ \mu \in \C ~|~ 0 \leq |\Im \mu | \leq  \delta_1^{-1} (\alpha  + |\tan\theta| n^\gamma\nu^\frac12 )\,, ~ \Re \mu = \frac{n^\gamma\nu^\frac12}{\delta} \big \}\,.
\end{split}
\end{align}
The estimate of the resolvent on $\Gamma_\pm \cup l_\pm$ follows from Corollary \ref{cor.prop.general'}. 
Let $g \in \mathcal{P}_n L^2 (\Omega_\nu)$. Then we have from \eqref{est.cor.prop.general'.2} and $|\Im \mu | = |\Re (\mu) | \, | \tan\theta| + \delta_1^{-1} (\alpha + |\tan\theta| n^\gamma\nu^\frac12 )$ on $\Gamma_\pm$, 
\begin{align*}
& \| \frac{1}{2\pi i} \int_{\Gamma_\pm} e^{\tau \mu} (\mu + \mathbb{L}_{\nu,n})^{-1} g \dd \mu \|_{L^2 (\Omega_\nu)} \\
& \leq C\int_{\Gamma_\pm} e^{\tau \Re (\mu)} |\mu|^{-1} |\dd \mu | \, \| g \|_{L^2 (\Omega_\nu)}\\
& \leq C \int_0^\infty \frac{e^{-\tau s}}{s + |\tan\theta|  s +  \delta_1^{-1} (\alpha + |\tan\theta| n^\gamma\nu^\frac12 )}\dd s \, \| g  \|_{L^2 (\Omega_\nu)}\,.
\end{align*}
Recall that $\alpha=n \nu^\frac12$ and $0<\gamma\leq 1$. 
Then, for any $\kappa\in (0,1]$ there is $C_\kappa>0$ such that 
\begin{align}\label{proof.thm.semigroup.5}
 \int_0^\infty \frac{e^{-\tau s}}{s + |\tan\theta|  s +  \delta_1^{-1} (\alpha + |\tan\theta| n^\gamma\nu^\frac12 )}\dd s \leq \frac{C_\kappa}{(\alpha \tau)^\kappa} \leq C_\kappa \qquad {\rm for}~~\alpha \tau\geq 1\,.
\end{align}
Hence we have, by taking $\kappa=\frac12$ for example,
\begin{align}\label{proof.thm.semigroup.6}
\| \frac{1}{2\pi i} \int_{\Gamma_\pm} e^{\tau \mu} (\mu + \mathbb{L}_{\nu,n})^{-1} g  \dd \mu \|_{L^2 (\Omega_\nu)} \leq C \| g \|_{L^2 (\Omega_\nu)}\,, \qquad \tau\geq \alpha^{-1}\,.
\end{align}
Next we see 
\begin{align}\label{proof.thm.semigroup.7}
& \| \frac{1}{2\pi i} \int_{l_\pm} e^{\tau \mu} (\mu + \mathbb{L}_{\nu,n})^{-1} g \dd \mu \|_{L^2 (\Omega_\nu)}\nonumber \\
& \leq C \int_0^{\frac{n^\gamma \nu^\frac12}{\delta}} \frac{e^{\tau s}}{s +  \delta_1^{-1} (\alpha + |\tan\theta| n^\gamma\nu^\frac12 )} \dd s \, \| g \|_{L^2 (\Omega_\nu)}\nonumber \\
& \leq C n^{\gamma-1} e^{\frac{n^\gamma \nu^\frac12}{\delta}  \tau} \| g \|_{L^2 (\Omega_\nu)} \,.
\end{align}
Finally, on $l_0$ we apply the result of Corollary \ref{cor.prop.resolvent}. 
To this end we consider the following two cases by taking Corollary \ref{cor.prop.resolvent} into account:

\vspace{0.3cm}

\noindent 
{\rm (i')} $n^{\gamma} \nu^\frac12 +\delta n^2 \nu^\frac32\leq \delta \delta_2^{-1}$ ~~~ ~~~~~ {\rm (ii')} $n^{\gamma} \nu^\frac12 +\delta n^2 \nu^\frac32\geq \delta \delta_2^{-1}$ 

\vspace{0.3cm}

\noindent 
In the case (i)' we have from \eqref{est.cor.prop.resolvent.2},
\begin{align}\label{proof.thm.semigroup.8}
\| \frac{1}{2\pi i} \int_{l_0} e^{\tau \mu} (\mu + \mathbb{L}_{\nu,n})^{-1}  g \dd \mu \|_{L^2 (\Omega_\nu)} & \leq C \frac{\delta}{n^\gamma \nu^\frac12}  n^{\frac32(1-\gamma)}  e^{\frac{n^\gamma \nu^\frac12}{\delta}\tau} \int_0^{\frac{C\alpha}{\delta_1}}  \dd s \, \| g \|_{L^2 (\Omega_\nu)} \nonumber \\
& \leq C\delta \delta_1^{-1}   n^{\frac52(1-\gamma)}   e^{\frac{n^\gamma \nu^\frac12}{\delta}\tau}  \| g\|_{L^2 (\Omega_\nu)}\,.
\end{align}
Collecting \eqref{proof.thm.semigroup.1}, \eqref{proof.thm.semigroup.6}, \eqref{proof.thm.semigroup.7}, and \eqref{proof.thm.semigroup.8}, we have arrived at the estimate 
\begin{align}\label{proof.thm.semigroup.9}
\| e^{-\tau \mathbb{L}_{\nu,n}}  g \|_{L^2 (\Omega_\nu)} \leq C n^{\frac52(1-\gamma)}  e^{\frac{n^\gamma \nu^\frac12}{\delta}\tau}  \| g \|_{L^2 (\Omega_\nu)}\,,\qquad \tau >0 \,,
\end{align}
where $C$ depends only on $U^E$ and $U$. In virtue of (i) in Corollary \ref{cor.prop.resolvent}, if \eqref{concave.strong} holds then the above estimate holds for $\gamma\in [\frac23,1]$ with the factor $n^{\frac52(1-\gamma)}$ replaced by $n^{2(1-\gamma)}$.  
The estimate \eqref{proof.thm.semigroup.9} implies \eqref{est.thm.semigroup.1} in the case (i') by returning to the original variables. The estimate \eqref{est.thm.semigroup.3} for the case (ii') is obtained in the same manner by using \eqref{est.cor.prop.resolvent.4} instead of \eqref{est.cor.prop.resolvent.2}. Indeed, in this case \eqref{proof.thm.semigroup.8} is replaced by 
\begin{align}\label{proof.thm.semigroup.10}
\| \frac{1}{2\pi i} \int_{l_0} e^{\tau \mu} (\mu + \mathbb{L}_{\nu,n})^{-1}  g \dd \mu \|_{L^2 (\Omega_\nu)} & \leq C \frac{\delta}{n^\gamma \nu^\frac12} \big ( 1+  \frac{\delta}{n^\gamma \nu^\frac12} \big )  e^{\frac{n^\gamma \nu^\frac12}{\delta}\tau} \int_0^{\frac{C\alpha}{\delta_1}}  \dd s \, \| g \|_{L^2 (\Omega_\nu)} \nonumber \\
& \leq C\delta \delta_1^{-1} n^{1-\gamma}  e^{\frac{n^\gamma \nu^\frac12}{\delta}\tau}  \| g\|_{L^2 (\Omega_\nu)}\,.
\end{align}
Thus we obtain \eqref{est.thm.semigroup.3}.
Note that the dependence on $n$ in the case (ii') is milder than in the case (i'). 
Hence, we conclude that \eqref{est.thm.semigroup.1} holds for $n^{\gamma}\nu^\frac12 \leq 1$, 
while \eqref{est.thm.semigroup.3} follows in the case $n^{\gamma}\nu^\frac12 \geq 1$, as desired.

Next we consider the derivative estimate.
Let us go back to the representation \eqref{proof.thm.semigroup.3} with the curve $\Gamma$ as mentioned in \eqref{proof.thm.semigroup.4}.
For the integral on $\Gamma_\pm$ we have from \eqref{est.cor.prop.general'.3},
\begin{align}
& \| \nabla_{X,Y} \frac{1}{2\pi i} \int_{\Gamma_\pm} e^{\tau \mu} (\mu + \mathbb{L}_{\nu,n})^{-1} g \dd \mu \|_{L^2 (\Omega_\nu)} \nonumber \\
& \leq \frac{C}{\nu^\frac14} \int_{\Gamma_\pm} e^{\tau \Re (\mu)} |\mu|^{-\frac12} |\dd \mu | \, \| g \|_{L^2 (\Omega_\nu)} \nonumber \\
& \leq \frac{C}{\nu^\frac14} \int_0^\infty \frac{e^{-\tau s}}{\big (s + |\tan\theta|  s +  \delta_1^{-1} (\alpha + |\tan\theta| n^\gamma\nu^\frac12 ) \big )^\frac12}\dd s \, \| g  \|_{L^2 (\Omega_\nu)}\nonumber \\
&  \leq  \frac{C}{\nu^\frac14 \tau^\frac12} \| g\|_{L^2 (\Omega_\nu)}\,.\label{proof.thm.semigroup.11}
\end{align}
Similarly, on $l_\pm$ we have 
\begin{align}
& \| \nabla_{X,Y} \frac{1}{2\pi i} \int_{l_\pm} e^{\tau \mu} (\mu + \mathbb{L}_{\nu,n})^{-1} g \dd \mu \|_{L^2 (\Omega_\nu)} \nonumber \\
& \leq \frac{C}{\nu^\frac14} \int_0^{\frac{n^\gamma \nu^\frac12}{\delta}} \frac{e^{\tau s}}{\big (s  +  \delta_1^{-1} (\alpha + |\tan\theta| n^\gamma\nu^\frac12 ) \big )^\frac12}\dd s \, \| g  \|_{L^2 (\Omega_\nu)}\nonumber \\
& \leq \frac{C}{\nu^\frac14 \alpha^\frac12} \int_0^{\frac{n^\gamma \nu^\frac12}{\delta}}  \dd s \, e^{\frac{n^\gamma \nu^\frac12}{\delta}\tau} \| g\|_{L^2 (\Omega_\nu)} \leq C n^{\gamma-\frac12}  e^{\frac{n^\gamma \nu^\frac12}{\delta}\tau} \| g\|_{L^2 (\Omega_\nu)}\,.\label{proof.thm.semigroup.12}
\end{align}
The estimate on $l_0$ is obtained in the same manner as in \eqref{proof.thm.semigroup.8} and \eqref{proof.thm.semigroup.10}. 
Indeed, in the case (i') the estimate \eqref{est.cor.prop.resolvent.3} yields
\begin{align}\label{proof.thm.semigroup.13}
\begin{split}
& \| \nabla_{X,Y} \frac{1}{2\pi i} \int_{l_0} e^{\tau \mu} (\mu + \mathbb{L}_{\nu,n})^{-1}  g \dd \mu \|_{L^2 (\Omega_\nu)} \\
& \qquad \leq C n^{\frac52(1-\gamma)} \big (n^\frac14 + n^{\frac12 - (1-\gamma)}  \big )  e^{\frac{n^\gamma \nu^\frac12}{\delta}\tau}  \| g\|_{L^2 (\Omega_\nu)}\,,
\end{split}
\end{align}
while in the case (ii') the estimate \eqref{est.cor.prop.resolvent.5} implies
\begin{align}\label{proof.thm.semigroup.14}
\| \nabla_{X,Y} \frac{1}{2\pi i} \int_{l_0} e^{\tau \mu} (\mu + \mathbb{L}_{\nu,n})^{-1}  g \dd \mu \|_{L^2 (\Omega_\nu)} & \leq C  n^{1-\frac{\gamma}{2}} e^{\frac{n^\gamma \nu^\frac12}{\delta}\tau}  \| g\|_{L^2 (\Omega_\nu)}\,.
\end{align}
Collecting \eqref{proof.thm.semigroup.11} - \eqref{proof.thm.semigroup.14}, we obtain the estimates for the spatial derivative of $e^{-\tau \mathbb{L}_{\nu,n}} g$, which lead to \eqref{est.thm.semigroup.2} and \eqref{est.thm.semigroup.4} by returning to the original variables.
The proof is complete.
\end{proofx}

\section{Linear evolution operator in middle frequency: the case of time-dependent shear flow}\label{sec.middle.time}

In this section we consider the case when $U^E$ and $U^P$ depend on the time variable.
Our strategy in achieving the estimate of the evolution operator $\mathbb{T}_{\nu,n} (t,s)$ with $\delta_0^{-1}\leq |n| \leq \delta_0^{-1} \nu^{-\frac34}$, is to split the time interval  depending on $n$ and to expand the profile $U^P$ around a fixed time in each short time interval,  in which we can apply the perturbation argument based on the result of Theorem \ref{thm.semigroup} in the previous section.
To verify this idea we need enough regularity of $U^E$ and $U^P$ as well as a concave shape of $U^P(t)$ for each $t$ which should be uniform in time; see \eqref{concave.weak.t} in Section \ref{sec.result}. In fact, due to the factor $|n|^{\frac52(1-\gamma)}$ in \eqref{est.thm.semigroup.1} and the underlying derivative loss structure of \eqref{eq.perturb.intro}, even if we assume the strong concave condition \eqref{concave.strong} on $U^P(t)$ uniformly in time, the present approach does not provide a growth estimate of the order $e^{C|n|^\gamma (t-s)}$ with $\gamma=\frac23$ for $\mathbb{T}_{\nu,n} (t,s)$ in the the parameter regime $|n|^{1+\gamma}\nu \leq 1$. 
Note that, as stated in Remark \ref{rem.thm.semigroup},  the value $\gamma=\frac23$ in the growth estimate is known to be optimal at least for the semigroup $e^{-t\mathbb{A}_{\nu,n}}$, and we have established  this optimal growth bound in Theorem \ref{thm.semigroup} under the condition \eqref{concave.strong} for the time-independent profile $U^P(t)=U^P$.  
The condition \eqref{concave.weak.t} is always satisfied when $U^P$ is the solution to \eqref{eq.heat} with $U^E(0)>0$, if the initial data $U_s\in BC^3 (\R_+)$ satisfies $\| U_s\|<\infty$, $Y\pa_Y^3 U\in L^\infty (\R_+)$, compatibility conditions on $Y=0,\infty$,  and (ii) of \eqref{concave.strong}.

\begin{thm}\label{thm.evolution.middle} Assume that \eqref{concave.weak.t}, stated in Section \ref{sec.result}, holds for some $T>0$.
Then there exist $C, K_0>0$ such that the following estimates hold for all $f\in \mathcal{P}_n L^2_\sigma (\Omega)$ and $0\leq s<t\leq T$. Set 
\begin{align}\label{est.thm.evolution.0}
\theta_{\gamma,n} \, = \, |n|^\gamma \big ( 1+ (1-\gamma)\log (1+|n|)\big )\,, \qquad \gamma \in [\frac23,1]\,.
\end{align}

\noindent 
{\rm (i)} Let $\delta_0^{-1}\leq |n|\leq \delta_0^{-1}\nu^{-\frac34}$. If in addition $\gamma\in [\frac79,1]$ and $|n|^\gamma \nu^\frac12 <1$, then
\begin{align}
\| \mathbb{T}_{\nu,n} (t,s) f \|_{L^2(\Omega)} & \leq C  n^{\frac52(1-\gamma)} e^{K_0 \theta_{\gamma,n} (t-s)}  \| f\|_{L^2 (\Omega)}\,, \label{est.thm.evolution.1}\\
\| \nabla \mathbb{T}_{\nu,n} (t,s) f \|_{L^2(\Omega)} & \leq \frac{C}{\nu^\frac12} \big ( \frac{1}{(t-s)^\frac12} + |n|^{\frac12 + 3 (1-\gamma)} e^{K_0 \theta_{\gamma,n} (t-s)} \big ) \| f\|_{L^2 (\Omega)}\,.\label{est.thm.evolution.2}
\end{align}

\noindent 
{\rm (ii)} If $\gamma\in [\frac23,1]$ and $|n|^\gamma \nu^\frac12 \geq 1$ then  
\begin{align}
\| \mathbb{T}_{\nu,n} (t,s) f \|_{L^2(\Omega)} & \leq C  n^{1-\gamma} e^{K_0 \theta_{\gamma,n} (t-s)}  \| f\|_{L^2 (\Omega)}\,, \label{est.thm.evolution.3}\\
\| \nabla \mathbb{T}_{\nu,n} (t,s) f \|_{L^2(\Omega)} & \leq \frac{C}{\nu^\frac12} \big ( \frac{1}{(t-s)^\frac12} + |n|^{\frac12+\frac32 (1-\gamma)} e^{K_0 \theta_{\gamma,n} (t-s)} \big ) \| f\|_{L^2 (\Omega)}\,.\label{est.thm.evolution.4}
\end{align}
\end{thm}

\begin{proofx}{Theorem \ref{thm.evolution.middle}} It suffices to show the estimate for $\mathbb{T}_{\nu,n} (T_0,0)$ with $0<T_0\leq T$.  We split the time interval $[0,T_0]$ as $t_0=0<t_1<\cdots<t_{N-1}<t_N=T_0$, where $t_l=\frac{l}{N}T_0$ with $N$ determined later. Set $u_n(t) = \mathbb{T}_{\nu,n} (t,0) f$. 
In the time interval $[t_l,t_{l+1}]$ the operator $\mathbb{A}_{\nu,n} (t)$ defined in \eqref{decompose.A} (recall also \eqref{def.A_t}) is expanded around $\mathbb{A}_{\nu,n} (t_l)$. Then by the Duhamel formula we have for $t\in (t_l,t_{l+1}]$,
\begin{align}\label{proof.thm.evolution.1}
\begin{split}
u_n (t)  = \mathbb{T}_{\nu,n} (t,t_l) u_n (t_l) & = e^{-(t-t_l) \mathbb{A}_{\nu,n} (t_l)} u_n (t_l) \\
&  - \int_{t_l}^t e^{-(t-s) \mathbb{A}_{\nu,n}(t_l)}\mathbb{P} \bigg ( in \big ( U^E (s) - U^E (t_l) \big ) u_n (s) \bigg ) \dd s\\
&   - \int_{t_l}^t e^{-(t-s) \mathbb{A}_{\nu,n}(t_l)}\mathbb{P} \bigg ( u_{n,2} (s) \partial_y \big ( U^E (s) - U^E (t_l) \big ) {\bf e}_1 \bigg ) \dd s \\
&  - \int_{t_l}^t e^{-(t-s) \mathbb{A}_{\nu,n}(t_l)}\mathbb{P} \bigg ( in \big ( U^P (s, \frac{\cdot}{\sqrt{\nu}}) - U^P (t_l, \frac{\cdot}{\sqrt{\nu}}) \big ) u_n (s) \bigg ) \dd s\\
&   - \int_{t_l}^t e^{-(t-s) \mathbb{A}_{\nu,n}(t_l)}\mathbb{P} \bigg ( u_{n,2} (s) \partial_y \big ( U^P (s, \frac{\cdot}{\sqrt{\nu}}) - U^P (t_l, \frac{\cdot}{\sqrt{\nu}}) \big ) {\bf e}_1 \bigg ) \dd s\,.
\end{split}
\end{align}
It is easy to see that
\begin{align*}
\| in \big ( U^E (s) - U^E (t_l) \big ) u_n (s) \|_{L^2(\Omega)} & \leq |n| (s-t_l) \|  \partial_t U^E \|_{L^\infty_{t,y}} \| u_n (s) \|_{L^2 (\Omega)}\,,\\
\| in \big ( U^P (s, \frac{\cdot}{\sqrt{\nu}}) - U^P (t_l, \frac{\cdot}{\sqrt{\nu}}) \big ) u_n (s) \|_{L^2 (\Omega)} &\leq |n| (s-t_l) \| \partial_t U^P \|_{L^\infty_{t,Y}} \| u_n (s) \|_{L^2 (\Omega)}\,. 
\end{align*} 
On the other hand, the interpolation inequality yields
\begin{align*}
\|  u_{n,2} (s) \partial_y \big ( U^E (s) - U^E (t_l) \big ) {\bf e}_1\|_{L^2(\Omega)} \leq C (s-t_l)^\frac12 \big ( \|\partial_t U^E \|_{L^\infty_{t,y}} + \| U^E \|_{L^\infty_t C^2_y}\big ) \| u_{n,2} (s) \|_{L^2(\Omega)}\,.
\end{align*}  
Finally, the same computation as in \eqref{proof.prop.general.low.1} using the Hardy inequality implies that 
\begin{align*}
\| u_{n,2} (s) \partial_y \big ( U^P (s, \frac{\cdot}{\sqrt{\nu}}) - U^P (t_l, \frac{\cdot}{\sqrt{\nu}}) \big ) {\bf e}_1\|_{L^2(\Omega)} \leq 2 |n| (s-t_l)\| Y \partial_t \pa_Y U^P \|_{L^\infty_{t,Y}} \| u_{n,1} (s) \|_{L^2 (\Omega)}\,.
\end{align*}

\paragraph{Case (i) $\delta_0^{-1}\leq |n|\leq \delta_0^{-1}\nu^{-\frac34}$, $\gamma\in [\frac79,1]$, and $|n|^\gamma \nu^\frac12<1$.} 
In this case we apply Theorem \ref{thm.semigroup} (i) for the semigroup $e^{-t\mathbb{A}_{\nu,n} (t_l)}$,
which gives for $\gamma\in [\frac79,1]\subset [\frac57,1]$,
\begin{align*}
\begin{split}
\| u_n (t) \|_{L^2(\Omega)} & \leq C  |n|^{\frac52(1-\gamma)}  e^{\frac{|n|^\gamma}{\delta}(t-t_l)}  \| u_n (t_l) \|_{L^2 (\Omega)}\\
& \quad + C |n|^{1+\frac52 (1-\gamma)}  \int_{t_l}^t    e^{\frac{|n|^\gamma}{\delta}(t-s)}  (s-t_l) \| u_n (s) \|_{L^2 (\Omega)} \dd s\\
& \qquad + C |n|^{\frac52(1-\gamma)} \int_{t_l}^t e^{\frac{|n|^\gamma}{\delta} (t-s)} (s-t_l)^\frac12 \| u_n (s) \|_{L^2 (\Omega)} \dd s\,.
\end{split}
\end{align*}
Since $t_{l+1}-t_l=\frac{T_0}{N}$ this estimate implies 
\begin{align*}
\begin{split}
\sup_{t_l < t\leq t_{l+1} }  \| u_n (t) \|_{L^2 (\Omega)} & \leq  C  |n|^{\frac52(1-\gamma)}  e^{\frac{|n|^\gamma}{\delta}\frac{T_0}{N}}  \| u_n (t_l) \|_{L^2 (\Omega)}\\
&\quad  + C_1  |n|^{1+\frac52 (1-\gamma)} (\frac{T_0}{N})^2 e^{\frac{|n|^\gamma}{\delta} \frac{T_0}{N}}  \sup_{t_l < t\leq t_{l+1} }  \| u_n (t) \|_{L^2 (\Omega)}\\
& \qquad + C_1 |n|^{\frac52(1-\gamma)} (\frac{T_0}{N})^\frac32 e^{\frac{|n|^\gamma}{\delta} \frac{T_0}{N}}  \sup_{t_l < t\leq t_{l+1} }  \| u_n (t) \|_{L^2 (\Omega)} \,.
\end{split}
\end{align*}
Let us take 
\begin{align}
N=\max \{ \tilde \delta^{-1} |n|^\gamma T_0\,, 1\} \,, \qquad \gamma\in [\frac79,1]\,,\label{proof.thm.evolution.2}
\end{align}
where $\tilde \delta\in (0,\delta]$ is a small number  depending only on $C_1$ chosen so that, for $\gamma\in [\frac79,1]$, 
\begin{align*}
C_1 |n|^{1+\frac52 (1-\gamma)} (\frac{T_0}{N})^2 e^{\frac{|n|^\gamma}{\delta} \frac{T_0}{N}} 
+ C_1 |n|^{\frac52(1-\gamma)} (\frac{T_0}{N})^\frac32 e^{\frac{|n|^\gamma}{\delta} \frac{T_0}{N}} 
&\leq C_1  \tilde \delta^2  |n|^{\frac72-\frac92\gamma} e^{\frac{\tilde \delta}{\delta}} + C_1 \tilde \delta^\frac32 |n|^{\frac52-4\gamma}  e^{\frac{\tilde \delta}{\delta}} \\
& \leq\frac12\,.
\end{align*}
Thus we have 
\begin{align*}
\sup_{t_l < t\leq t_{l+1} }  \| u_n (t) \|_{L^2 (\Omega)} \leq C |n|^{\frac52(1-\gamma)}  e^{\frac{|n|^\gamma}{\delta}\frac{T_0}{N}}  \| u_n (t_l) \|_{L^2 (\Omega)}\,,
\end{align*}
and in particular,
\begin{align}
\| u_n (t_{l+1}) \|_{L^2 (\Omega)} \leq C_2 |n|^{\frac52(1-\gamma)}  e^{\frac{|n|^\gamma}{\delta}\frac{T_0}{N}}  \| u_n (t_l) \|_{L^2 (\Omega)}\,.\label{proof.thm.evolution.3}
\end{align}
Here $C_2$ is independent of $n$, $N$, and $T_0$. 
If $N=1$ then we obtain the desired estimate for $T_0\leq \tilde \delta_0 |n|^{-\gamma}$.
If $N\geq 2$ then the iteration leads to the estimate
\begin{align}
\| u_n (T_0)\|_{L^2 (\Omega)} = \| u_n (t_N) \|_{L^2(\Omega)} & \leq \bigg ( C_2  |n|^{\frac52(1-\gamma)}  e^{\frac{|n|^\gamma}{\delta}\frac{T_0}{N}} \bigg )^N \| u_n (t_0) \|_{L^2 (\Omega)}\nonumber \\
& = \big( C_2 |n|^{\frac52(1-\gamma)} \big )^{\frac{|n|^\gamma T_0}{\tilde \delta}} e^{\frac{|n|^\gamma}{\delta} T_0} \| f\|_{L^2 (\Omega)}\nonumber \\
& \leq Ce^{K|n|^{\gamma} \big (1+ (1-\gamma)\log |n| \big ) T_0} \| f\|_{L^2 (\Omega)}\,,\label{proof.thm.evolution.4}
\end{align}
for some $K>0$ independent of $n$, $\gamma$, and $T_0$. Hence, \eqref{est.thm.evolution.1} holds.
Next we show the derivative estimate. 
The above iteration argument does not work well due to the appearance of the singularity in a short time.
To overcome this difficulty, recalling \eqref{def.A_t}, we write $u_n(t)=\mathbb{T}_{\nu,n} (t,0) f$ as, instead of \eqref{proof.thm.evolution.1},
\begin{align}\label{proof.thm.evolution.5}
\begin{split}
u_n (t) & = e^{-\nu t \mathbb{A}_n} f \\
&- \int_0^t e^{-\nu (t-s) \mathbb{A}_n} \mathbb{P} \bigg ( i n \big ( U^E (s) + U^{BL} (s,\frac{\cdot}{\sqrt{\nu}}) \big ) u_n (s) + u_{n,2} \partial_y \big ( U^E (s) + U^{BL} (s, \frac{\cdot}{\sqrt{\nu}}) \big ) {\bf e}_1 \bigg ) \dd s
\end{split}
\end{align}
Here $\mathbb{A}_n$ is the Stokes operator $\mathbb{A}=-\mathbb{P}\Delta$ restricted on the invariant space $\mathcal{P}_n L^2_\sigma (\Omega)$. It is well known that 
\begin{align*}
\| \nabla e^{-\nu t\mathbb{A}} f\|_{L^2 (\Omega)} \leq \frac{C}{\nu^\frac12 t^\frac12} \| f\|_{L^2 (\Omega)}\,, \qquad t>0\,.
\end{align*}
Then we have by using the estimate as in \eqref{proof.prop.general.low.1} and by applying \eqref{est.thm.evolution.1},
\begin{align}
\| \nabla u_n (T_0) \|_{L^2(\Omega)} & \leq \frac{C}{\nu^\frac12 T_0^\frac12} \| f\|_{L^2 (\Omega)} + \frac{C|n|}{\nu^\frac12} \int_0^{T_0} \frac{\| u_n (s) \|_{L^2 (\Omega)}}{(T_0-s)^{\frac12}} \dd s\nonumber \\
& \leq  \frac{C}{\nu^\frac12 T_0^\frac12} \| f\|_{L^2 (\Omega)} + \frac{C|n|^{1+\frac52(1-\gamma)}}{\nu^\frac12} \int_0^{T_0}  \frac{e^{K|n|^\gamma \big (1+ (1-\gamma)\log |n|\big ) s}}{(T_0-s)^{\frac12}}    \| f\|_{L^2 (\Omega)} \dd s\nonumber\\
& \leq \frac{C}{\nu^\frac12 T_0^\frac12} \| f\|_{L^2 (\Omega)}  + \frac{C|n|^{\frac12 + 3 (1-\gamma)} e^{K|n|^\gamma \big (1+ (1-\gamma)\log |n|\big ) T_0}}{\nu^\frac12}  \| f\|_{L^2 (\Omega)}\,.\label{proof.thm.evolution.6}
\end{align}
This implies \eqref{est.thm.evolution.2}.

\paragraph{Case (ii) $|n|^\gamma \nu^\frac12 \geq 1$, $\gamma \in [\frac23,1]$.}
The strategy is the same as in the case (i), and we apply Theorem \ref{thm.semigroup} (ii) instead of (i) in this case, which yields from \eqref{proof.thm.evolution.1},
\begin{align*}
\begin{split}
\| u_n (t) \|_{L^2(\Omega)} & \leq C  |n|^{1-\gamma}  e^{\frac{|n|^\gamma}{\delta}(t-t_l)}  \| u_n (t_l) \|_{L^2 (\Omega)}\\
& \quad + C |n|^{1+ 1-\gamma}  \int_{t_l}^t    e^{\frac{|n|^\gamma}{\delta}(t-s)}  (s-t_l) \| u_n (s) \|_{L^2 (\Omega)} \dd s\\
& \qquad + C |n|^{1-\gamma} \int_{t_l}^t e^{\frac{|n|^\gamma}{\delta} (t-s)} (s-t_l)^\frac12 \| u_n (s) \|_{L^2 (\Omega)} \dd s\,.
\end{split}
\end{align*}
Thus we have 
\begin{align*}
\begin{split}
\sup_{t_l < t\leq t_{l+1} }  \| u_n (t) \|_{L^2 (\Omega)} & \leq  C  |n|^{1-\gamma}  e^{\frac{|n|^\gamma}{\delta}\frac{T_0}{N}}  \| u_n (t_l) \|_{L^2 (\Omega)}\\
&\quad  + C_1  |n|^{1+1-\gamma} (\frac{T_0}{N})^2 e^{\frac{|n|^\gamma}{\delta} \frac{T_0}{N}}  \sup_{t_l < t\leq t_{l+1} }  \| u_n (t) \|_{L^2 (\Omega)}\\
& \qquad + C_1 |n|^{1-\gamma} (\frac{T_0}{N})^\frac32 e^{\frac{|n|^\gamma}{\delta} \frac{T_0}{N}}  \sup_{t_l < t\leq t_{l+1} }  \| u_n (t) \|_{L^2 (\Omega)} \,.
\end{split}
\end{align*}
Setting $N$ as in \eqref{proof.thm.evolution.2}, and we see for $\gamma\in [\frac23,1]$,
\begin{align*}
C_1 |n|^{1+1-\gamma} (\frac{T_0}{N})^2 e^{\frac{|n|^\gamma}{\delta} \frac{T_0}{N}} 
+ C_1 |n|^{1-\gamma} (\frac{T_0}{N})^\frac32 e^{\frac{|n|^\gamma}{\delta} \frac{T_0}{N}} 
&\leq C_1  \tilde \delta^2  |n|^{2-3\gamma} e^{\frac{\tilde \delta}{\delta}} + C_1 \tilde \delta^\frac32 |n|^{1-\frac52\gamma}  e^{\frac{\tilde \delta}{\delta}} \\
& \leq\frac12\,,
\end{align*}
if $\tilde \delta$ is sufficiently small depending only on $C_1$. Hence we obtain 
\begin{align}
\| u_n (t_{l+1}) \|_{L^2 (\Omega)} \leq C_2 |n|^{1-\gamma}  e^{\frac{|n|^\gamma}{\delta}\frac{T_0}{N}}  \| u_n (t_l) \|_{L^2 (\Omega)}\,.\label{proof.thm.evolution.7}
\end{align}
Then \eqref{est.thm.evolution.3} and \eqref{est.thm.evolution.4} follow from \eqref{proof.thm.evolution.7} by the same argument as in the case (i).
The details are omitted here. The proof is complete.
\end{proofx}

\section{Nonlinear stability in Gevrey class}\label{sec.nonlinear}

In this section we consider the full nonlinear problem \eqref{eq.perturb.intro}.
By the Duhamel formula the associated integral equations are given by 
\begin{align}\label{eq.ns.integral}
u(t) = \mathbb{T}_\nu (t,0) a - \int_0^t \mathbb{T}_\nu (t,s) \mathbb{P} \big ( u \cdot \nabla u \big ) \dd s\,, \qquad t>0\,.
\end{align}
For $\gamma\in (0,1]$, $d\geq 0$, and $K>0$ let us introduce the Banach space $X_{d,\gamma,K}$ as 
\begin{align}
X_{d,\gamma, K} = \{ f\in L^2_\sigma (\Omega) ~|~ \| f \|_{X_{d,\gamma,K}} = \sup_{n\in \mathbb{Z}} \, (1+|n|^d) e^{K\theta_{\gamma,n}} \| \mathcal{P}_n f \|_{L^2 (\Omega)} <\infty \}\,,
\end{align}
where $\theta_{\gamma,n}=|n|^\gamma \big ( 1+ (1-\gamma) \log (1+|n|) \big )$ as in \eqref{est.thm.evolution.0}.

\begin{thm}\label{thm.nonlinear} Assume that \eqref{concave.weak.t} holds for some $T>0$. For any $\gamma \in [\frac79,1)$, $d>\frac92-\frac72\gamma$, and $K>0$, there exist $T'\in (0,T]$ and $K'\in (0,K)$ such that the following statement holds for any sufficiently small $\nu>0$.
If $\| a\|_{X_{d,\gamma,K}} \leq \nu^{\frac12+\beta}$ with $\beta=\frac{2(1-\gamma)}{\gamma}$ then the system \eqref{eq.ns.integral} admits a unique solution $u\in C([0,T']; L^2_\sigma (\Omega))\cap L^2 (0,T'; W^{1,2}_0 (\Omega)^2)$ satisfying the estimate
\begin{align}\label{est.thm.nonlinear.1}
\sup_{0<t\leq T'} \big ( \| u (t) \|_{X_{d,\gamma,K'}} + (\nu t)^\frac14 \| u(t) \|_{L^\infty (\Omega)} +  (\nu t)^\frac12 \| \nabla u (t) \|_{L^2 (\Omega)} \big )  \leq C  \| a\|_{X_{d,\gamma,K}} \,.
\end{align}
Here $C>0$ is independent of $\nu$.
\end{thm}

\begin{rem}{\rm (i) Since $a\in L^2_\sigma (\Omega)$ and the problem is a two-dimensional one, the unique existence of global solutions to \eqref{eq.ns.integral} in $C([0,\infty); L^2_\sigma (\Omega))\cap L^2_{loc} (0,\infty; W^{1,2}_0 (\Omega)^2)$ is classical for any $\nu>0$. The nontrivial part of Theorem \ref{thm.nonlinear} is the estimate \eqref{est.thm.nonlinear.1}, which is uniform with respect to sufficiently small $\nu>0$.

\noindent (ii) Theorem \ref{thm.nonlinear} holds also for $\gamma=1$ if $\|a\|_{X_{g,1,K}} \leq \kappa \nu^\frac12$ with sufficiently small (but independent of $\nu$) $\kappa>0$.

\noindent (iii) The requirement $\gamma\in [\frac79,1)$ comes form Theorem \ref{thm.evolution.middle}; the stability estimate for the evolution operator is obtained only for the exponent $\gamma\in [\frac79,1]$. 
In other words, once Theorem \ref{thm.evolution.middle} is obtained for $\gamma\in [\frac23,1]$ (with the same estimates as in \eqref{est.thm.evolution.1} - \eqref{est.thm.evolution.4} ) then Theorem \ref{thm.nonlinear} holds for $\gamma\in [\frac23,1)$ without any change of the statement.
}
\end{rem}

\begin{proofx}{Theorem \ref{thm.nonlinear}} Set 
\begin{align}
q=d-\frac72(1-\gamma) \in (1,d) \,, \qquad K(t) = K - 2 K_0 t\,,\label{proof.thm.nonlinear.1}
\end{align}
where $K_0>0$ is the number in Theorem \ref{thm.evolution.middle}.
We establish the a priori estimate of the solution to \eqref{eq.ns.integral} in the space 
\begin{align}\label{proof.thm.nonlinear.2}
\begin{split}
Y_{\gamma, K, T'} & = \{ f\in C([0, T']; L^2_\sigma (\Omega))~|~\\
& \qquad \|f\|_{Y_{\gamma,K,T'}} = \sup_{0<t\leq T'}  \big ( \| f (t) \|_{X_{q,\gamma, K(t)}} +  \| (\nu t)^\frac12 \nabla f (t) \|_{X_{q,\gamma, K(t)}} \big )  <\infty \}\,.
\end{split}
\end{align}  
For each $n\in \mathbb{Z}$ we have 
\begin{align}\label{proof.thm.nonlinear.3}
\mathcal{P}_n u(t) = \mathbb{T}_{\nu,n} (t,0) \mathcal{P}_n a - \int_0^t \mathbb{T}_{\nu,n} (t,s) \mathbb{P} \mathcal{P}_n \big (  u \cdot \nabla u  \big ) \dd s\,, \qquad t>0\,.
\end{align}
For $1\leq |n|\leq \delta_0^{-1}\nu^{-\frac34}$ the evolution operator $\mathbb{T}_{\nu,n} (t,s)$ is estimated as in \eqref{est.thm.evolution.1} - \eqref{est.thm.evolution.4} in virtue of Proposition \ref{prop.general.low} for $|n|\leq \delta_0^{-1}$ and Theorem \ref{thm.evolution.middle} for $\delta_0^{-1}<|n|\leq \delta_0^{-1}\nu^{-\frac34}$. The estimate for the case $n=0$ follows from Proposition \ref{prop.general.low}  with $m_1=1$.

\paragraph{Case (i) $|n|^\gamma \nu^\frac12 < 1$.}
In this case we have from \eqref{est.thm.evolution.1},
\begin{align}\label{proof.thm.nonlinear.4}
\begin{split}
\| \mathcal{P}_n u (t) \|_{L^2 (\Omega)} & \leq C (1+ |n|^{\frac52(1-\gamma)} ) e^{K_0 \theta_{\gamma,n} t} \| \mathcal{P}_n a\|_{L^2 (\Omega)} \\
&\quad  + C (1+|n|^{\frac52(1-\gamma)} ) \int_0^t e^{K_0 \theta_{\gamma,n} (t-s)}  \| \mathcal{P}_n (u \cdot \nabla u) \|_{L^2 (\Omega)} \dd s \,.
\end{split}
\end{align}
The nonlinear term is estimated as 
\begin{align*}
\| \mathcal{P}_n ( u \cdot \nabla u )  \|_{L^2 (\Omega)} & \leq \| \| \mathcal{P}_n ( u_1 \pa_x u )  \|_{L^2 (\Omega)}   + \| \mathcal{P}_n ( u_2 \pa_y u )  \|_{L^2 (\Omega)} \nonumber  \\
& \leq \| \sum_{j\in \mathbb{Z}} (e^{-i j x} \mathcal{P}_j u_1) \cdot (e^{-i (n-j) x} \pa_x \mathcal{P}_{n-j} u ) \|_{L^2_y (\R_+)} \nonumber \\
& \quad + \| \sum_{j\in \mathbb{Z}} (e^{-i j x} \mathcal{P}_j u_2) \cdot (e^{-i (n-j) x} \pa_y \mathcal{P}_{n-j} u ) \|_{L^2_y (\R_+)}\,.
\end{align*}
From the Gagliardo-Nirenberg inequality we have 
\begin{align*}
& \| \sum_{j\in \mathbb{Z}} (e^{-i j x} \mathcal{P}_j u_1) \cdot (e^{-i (n-j) x} \pa_x \mathcal{P}_{n-j} u ) \|_{L^2_y (\R_+)}\\
& \leq \sum_{j\in \mathbb{Z}} \| e^{-i j x} \mathcal{P}_j u_1 \|_{L^\infty_y (\R_+)} \| e^{-i (n-j) x} \mathcal{P}_{n-j} \pa_x u\|_{L^2_y (\R_+)}\\
& \leq C \sum_{j\in \mathbb{Z}} \| e^{-ijx} \mathcal{P}_j u_1 \|_{L^2(\R_+)}^\frac12 \| \partial_y e^{-ijx} \mathcal{P}_j u_1 \|_{L^2 (\R_+)}^\frac12 |n-j| ^\frac12 \| \mathcal{P}_{n-j} u \|_{L^2 (\Omega)}^\frac12 \| \mathcal{P}_{n-j} \pa_x u \|_{L^2 (\Omega)}^\frac12  \nonumber \\
& \leq  C \sum_{j\in \mathbb{Z}} \| \mathcal{P}_j u _1\|_{L^2(\Omega)}^\frac12 \| \nabla \mathcal{P}_j u_1 \|_{L^2(\Omega)}^\frac12  |n-j|^\frac12 \| \mathcal{P}_{n-j} u \|_{L^2 (\Omega)}^\frac12 \| \nabla \mathcal{P}_{n-j} u \|_{L^2 (\Omega)}^\frac12\,.
\end{align*}
On the other hand, the divergence free condition implies 
\begin{align*}
& \| \sum_{j\in \mathbb{Z}} (e^{-i j x} \mathcal{P}_j u_2) \cdot (e^{-i (n-j) x} \pa_y \mathcal{P}_{n-j} u ) \|_{L^2_y (\R_+)}\\
& \leq \sum_{j\in \mathbb{Z}} \| e^{-i j x} \mathcal{P}_j u_2 \|_{L^\infty_y (\R_+)} \| e^{-i (n-j) x} \mathcal{P}_{n-j} \pa_y u\|_{L^2_y (\R_+)}\\
& \leq C \sum_{j\in \mathbb{Z}} \| e^{-ijx} \mathcal{P}_j u_2 \|_{L^2(\R_+)}^\frac12 \| \partial_y e^{-ijx} \mathcal{P}_j u_2 \|_{L^2 (\R_+)}^\frac12  \| \mathcal{P}_{n-j} \pa_y u \|_{L^2 (\Omega)}  \nonumber \\
& \leq  C \sum_{j\in \mathbb{Z}} \| \mathcal{P}_j u _1\|_{L^2(\Omega)}^\frac12 \|\mathcal{P}_j \pa_x u_1 \|_{L^2(\Omega)}^\frac12  \| \nabla \mathcal{P}_{n-j} u \|_{L^2 (\Omega)}\\
& \leq C \sum_{j\in \mathbb{Z}} |j|^\frac12 \| \mathcal{P}_j u _1\|_{L^2(\Omega)} \| \nabla \mathcal{P}_{n-j} u \|_{L^2 (\Omega)}\,.
\end{align*}
Thus, for $u\in Y_{\gamma,K,T'}$ we have 
\begin{align*}
& \| \mathcal{P}_n ( u \cdot \nabla u )  (s) \|_{L^2 (\Omega)} \\
& \leq \frac{C}{(\nu s)^\frac12} \sum_{j\in \mathbb{Z}} \big ( \frac{1}{(1+|j|^q)(1+|n-j|^{q-\frac12})}  + \frac{1}{(1+|j|^{q-\frac12})(1+|n-j|^q)}  \big ) \\
& \qquad \qquad \qquad \times e^{-K(s) \theta_{\gamma,j}-K(s) \theta_{\gamma,n-j}} \| u\|_{Y_{\gamma,K,T'}}^2  
\end{align*}
Since the function $h(\tau) = \tau^\gamma \big ( 1+ (1-\gamma) \log (1+\tau)\big)$, $\gamma\in (0,1)$, is monotone increasing and concave for $\tau>0$ we have $\theta_{\gamma,j} + \theta_{\gamma,n-j}\geq \theta_{\gamma,n}$. Then we finally obtain  
\begin{align}
 \| \mathcal{P}_n ( u \cdot \nabla u )  (s) \|_{L^2 (\Omega)} & \leq \frac{C}{(\nu s)^\frac12} \frac{e^{-K(s) \theta_{\gamma,n}}}{1+|n|^{q-\frac12}}   \| u\|_{Y_{\gamma,K,T'}}^2\,.\label{proof.thm.nonlinear.5}
\end{align}
Here we have also used the condition $q>1$. Note that \eqref{proof.thm.nonlinear.5} itself is valid for all $n\in \mathbb{Z}$. By the definition of $K(t)$ in \eqref{proof.thm.nonlinear.1}, we have for $|n|\geq 1$,
\begin{align}
\int_0^t  e^{K_0 \theta_{\gamma,n} (t-s)} e^{-K(s)\theta_{\gamma,n}} s^{-\frac12} \dd s 
& \, = \, e^{-(K -K_0 t) \theta_{\gamma,n}} \int_0^t e^{K_0 \theta_{\gamma,n} s}  s^{-\frac12} \dd s \nonumber \\
& \leq C e^{-(K-2K_0t)\theta_{\gamma,n}} \min\{ \frac{1}{(K_0 \theta_{\gamma,n})^\frac12}, t^\frac12 \} \,,\label{proof.thm.nonlinear.6}
\end{align}
Therefore, from \eqref{proof.thm.nonlinear.4}, \eqref{proof.thm.nonlinear.5}, and \eqref{proof.thm.nonlinear.6}, we obtain  
\begin{align*}
\| \mathcal{P}_n u (t) \|_{L^2 (\Omega)} & \leq C \frac{e^{-(K-K_0 t) \theta_{\gamma,n}}}{1+|n|^{d-\frac52(1-\gamma)}} \| a \|_{X_{d,\gamma, K}} \\
& \quad + \frac{C (1+ |n|)^{\frac12 + \frac52 (1-\gamma)} e^{-K(t) \theta_{\gamma,n}}}{\nu^\frac12 (1+|n|^q)} \min\{ \frac{1}{(K_0\theta_{\gamma,n})^\frac12}, t^\frac12\} \| u \|_{Y_{\gamma,K,T'}}^2\,,
\end{align*}
and hence, if $T' < \frac{K}{2K_0}$ then for $\beta=\frac{2(1-\gamma)}{\gamma}$,
\begin{align}\label{proof.thm.nonlinear.8}
& \sup_{0<t\leq T'} \sup_{|n|^\gamma \nu^\frac12 < 1} (1+|n|^q) e^{K(t)\theta_{\gamma,n}} \| \mathcal{P}_n u(t) \|_{L^2 (\Omega)} \nonumber \\
&  \qquad \leq C \bigg (  \| a\|_{X_{d,\gamma,K}} +   \nu^{-\frac12} \| u \|_{Y_{\gamma,K,T'}}^2 \sup_{|n|^\gamma \nu^\frac12 <1} (1+|n|)^{\frac12 + \frac52(1-\gamma)} \min\{ \theta_{\gamma,n}^{-\frac12}, {T'}^\frac12\} \bigg ) \nonumber \\
& \qquad \leq C \bigg (  \| a\|_{X_{d,\gamma,K}} +   \nu^{-\frac12-\beta} \| u \|_{Y_{\gamma,K,T'}}^2 \sup_{|n|^\gamma \nu^\frac12 <1} (1+|n|)^{\frac12 + \frac52(1-\gamma)-2\beta\gamma} \min\{ \theta_{\gamma,n}^{-\frac12}, {T'}^\frac12\} \bigg ) \nonumber \\
& \qquad \leq C \bigg (  \| a\|_{X_{d,\gamma,K}} +   \nu^{-\frac12-\beta} \| u \|_{Y_{\gamma,K,T'}}^2 \sup_{n\in \mathbb{Z}} \, (1+|n|)^{\frac12 - \frac32(1-\gamma)} \min\{ \theta_{\gamma,n}^{-\frac12}, {T'}^\frac12\} \bigg ) \,.
\end{align}
Since $\theta_{\gamma,n}=|n|^\gamma \big (1+ (1-\gamma) \log (1+|n|)\big )$ and $\gamma<1$ we see 
\begin{align}\label{proof.thm.nonlinear.8'}
\lim_{T'\rightarrow 0}  \sup_{n\in \mathbb{Z}} \, (1+|n|)^{\frac12 - \frac32(1-\gamma)} \min\{ \theta_{\gamma,n}^{-\frac12}, {T'}^\frac12\} \, = \, 0 \,.
\end{align}
The derivative estimate is obtained similarly for $|n|^\gamma \nu^\frac12 < 1$. 
Indeed, from Proposition \ref{prop.general.low} for $|n|\leq \delta_0^{-1}$ and Theorem \ref{thm.evolution.middle} for $\delta_0^{-1}<|n| < \nu^{-\frac{1}{2\gamma}}$, we have, instead of \eqref{proof.thm.nonlinear.4},
\begin{align}\label{proof.thm.nonlinear.9}
\begin{split}
\| \nabla \mathcal{P}_n u (t) \|_{L^2 (\Omega)} & \leq \frac{C}{\nu^\frac12} \big ( \frac{1}{t^\frac12} + |n|^{\frac12+3(1-\gamma)}  e^{K_0 \theta_{\gamma,n} t} \big ) \| \mathcal{P}_n a\|_{L^2 (\Omega)} \\
& \quad + \frac{C}{\nu^\frac12} \int_0^t \big ( \frac{1}{(t-s)^\frac12} + |n|^{\frac12+3(1-\gamma)}  e^{K_0 \theta_{\gamma,n} (t-s)} \big )  \| \mathcal{P}_n (u \cdot \nabla u) \|_{L^2 (\Omega)} \dd s\,.
\end{split}
\end{align}
Noe that the inequality $e^{K_0 \theta_{\gamma,n} t} \leq C (K_0 \theta_{\gamma,n} t)^{-\frac12}e^{2 K_0|n|^\gamma t}$ holds. Then the first term of the right-hand side of \eqref{proof.thm.nonlinear.9} is bounded from above by 
\begin{align*}
\frac{C}{(\nu t)^\frac12} \big (1 + |n|^{\frac72(1-\gamma)}  e^{2 K_0 \theta_{\gamma,n} t} \big ) \| \mathcal{P}_n a\|_{L^2 (\Omega)} \,.
\end{align*}
On the other hand, by using \eqref{proof.thm.nonlinear.5} the second term of the right-hand side of \eqref{proof.thm.nonlinear.9} is estimated as 
\begin{align*}
& \frac{C(1+|n|^\frac12)}{\nu (1+|n|^q)} \int_0^t \big ( \frac{1}{(t-s)^\frac12} + |n|^{\frac12 + 3 (1-\gamma)}  e^{K_0 \theta_{\gamma,n} (t-s)} \big ) \frac{e^{- K(s)\theta_{\gamma,n}}}{s^\frac12}  \dd s \, \| u \|_{Y_{\gamma,K,T'}}^2 \\
& \, = :\,  I_{\nu,n} (t) \| u \|_{Y_{\gamma,K,T'}}^2\,.
\end{align*}
Here we have to be careful about the derivative loss in $I_{\nu,n} (t)$ for large $n$.
We observe that, for $l\in (-1,1]$,
\begin{align}\label{proof.thm.nonlinear.24}
\int_0^t (t-s)^{l} e^{K_0 \theta_{\gamma,n} (t-s) -K(s)\theta_{\gamma,n}} s^{-\frac12} \dd s & \leq Ce^{-K(t) \theta_{\gamma,n}} \min\{ \frac{1}{\theta_{\gamma,n}^{1+l} t^\frac12}\,, ~t^{\frac{1+l}{2}} \}\,.
\end{align}
Then the term $I_{\nu,n}(t)$ is estimated as 
\begin{align}\label{proof.thm.nonlinear.11}
\begin{split}
I_{\nu,n} (t) & \leq \frac{C (1+|n|^\frac12) e^{-K(t) \theta_{\gamma,n}}}{\nu (1+|n^q)} \min\{ \frac{1}{(\theta_{\gamma,n}t)^\frac12}\,, ~ t^{\frac14}\} \\
& \quad + \frac{C(1+|n|)^{1+3(1-\gamma)} e^{-K (t) \theta_{\gamma,n}}}{\nu (1+|n|^q)} \min\{\frac{1}{ \theta_{\gamma,n} t^\frac12}\,, ~t^\frac12 \} \,.
\end{split}
\end{align} 
Thus we obtain for $T' < \frac{K}{2K_0}$,
\begin{align*}
& (\nu t )^\frac12\| \nabla \mathcal{P}_n u (t) \|_{L^2 (\Omega)} \\
& \leq \frac{Ce^{-K(t) \theta_{\gamma,n} \gamma} }{(1+|n|)^{d-\frac72(1-\gamma)}}  \| a\|_{X_{d,\gamma,K_0}}   + \frac{C (1+|n|)^\frac12 e^{-K(t)\theta_{\gamma,n}}}{\nu^\frac12 (1+|n|^q)} \min\{ \theta_{\gamma,n}^{-\frac12} \,, ~ t^\frac34 \} \| u\|_{Y_{\gamma,K,T'}}^2 \nonumber \\
& \quad + \frac{C (1+|n|)^{1+3(1-\gamma)} e^{-K (t)\theta_{\gamma,n}}}{\nu^\frac12 (1+|n|^q)} \min \{ \theta_{\gamma,n}^{-1}\,, ~ t \} \, \| u\|_{Y_{\gamma,K,T'}}^2\,.
\end{align*}
Recall that  $q= d-\frac72 (1-\gamma)$ and $\beta = \frac{2(1-\gamma)}{\gamma}$.
Then we have 
\begin{align}\label{proof.thm.nonlinear.13}
& \sup_{0<t\leq T'} \sup_{|n|^\gamma \nu^\frac12 <1} (1+|n|^q) e^{K(t)|n|^\gamma} (\nu t )^\frac12\| \nabla \mathcal{P}_n u (t) \|_{L^2 (\Omega)} \nonumber \\
& \leq  C\bigg (  \| a\|_{X_{d,\gamma,K}}   + \nu^{-\frac12} \| u\|_{Y_{\gamma,K,T'}}^2 \sup_{|n|^\gamma\nu^\frac12<1} (1+|n|)^\frac12 \min\{ \theta_{\gamma,n}^{-\frac12}\,, ~{T'}^\frac34\} \nonumber \\
& \quad + \nu^{-\frac12} \| u\|_{Y_{\gamma,K,T'}}^2  \sup_{|n|^\gamma\nu^\frac12<1} \, (1+|n|)^{1+3 (1-\gamma)}\min \{ \theta_{\gamma,n}^{-1}\,, ~ {T'} \} \big ) \bigg ) \nonumber \\
& \leq  C\bigg (  \| a\|_{X_{d,\gamma,K}}   +  \nu^{-\frac12-\beta} \| u \|_{Y_{\gamma,K,T'}}^2 \sup_{|n|^\gamma\nu^\frac12<1} (1+|n|)^{\frac12-2\beta\gamma} \min\{ \theta_{\gamma,n}^{-\frac12}\,, ~{T'}^\frac34\} \nonumber \\
& \quad + \nu^{-\frac12-\beta} \| u\|_{Y_{\gamma,K,T'}}^2  \sup_{|n|^\gamma\nu^\frac12<1} \, (1+|n|)^{1+3 (1-\gamma)-2\beta \gamma}\min \{ \theta_{\gamma,n}^{-1}\,, ~ {T'} \}  \bigg ) \nonumber \\
\begin{split}
& \leq  C\bigg (  \| a\|_{X_{d,\gamma,K}}   +  \nu^{-\frac12-\beta} \| u \|_{Y_{\gamma,K,T'}}^2 \sup_{n\in \mathbb{Z}} (1+|n|)^{\frac12-4(1-\gamma)} \min\{ \theta_{\gamma,n}^{-\frac12}\,, ~{T'}^\frac34\} \\
& \quad + \nu^{-\frac12-\beta} \| u\|_{Y_{\gamma,K,T'}}^2  \sup_{n\in \mathbb{Z}} \, (1+|n|)^{1 - (1-\gamma)}\min \{ \theta_{\gamma,n}^{-1}\,, ~ {T'} \}  \bigg ) \,.
\end{split}
\end{align}
Note that, since $\theta_{\gamma,n} = |n|^\gamma \big ( 1+ (1-\gamma) \log (1+|n|)\big)$ and $\gamma <1$, we see that 
\begin{align}\label{proof.thm.nonlinear.14}
\begin{split}
\lim_{T'\rightarrow 0}  \sup_{n\in \mathbb{Z}} (1+|n|)^{\frac12-4(1-\gamma)} \min\{ \theta_{\gamma,n}^{-\frac12}\,, ~{T'}^\frac12\} & \, = \, 0\,,\\ 
\lim_{T'\rightarrow 0} \sup_{n\in \mathbb{Z}} \, (1+|n|)^{1 - (1-\gamma)}\min \{ \theta_{\gamma,n}^{-1}\,, ~ {T'} \}  & \,  =\, 0\,.
\end{split}
\end{align}

\paragraph{Case (ii) $|n|^\gamma \nu^\frac12 \geq 1$ and $|n|\leq \delta_0^{-1} \nu^{-\frac34}$.}
The argument is the same as in the case (i), we simply apply the result of (ii) in Theorem \ref{thm.evolution.middle} in this case. From \eqref{proof.thm.nonlinear.3} combined with \eqref{est.thm.evolution.3} and \eqref{proof.thm.nonlinear.5} we have 
\begin{align}\label{proof.thm.nonlinear.15}
\| \mathcal{P}_n u(t) \|_{L^2 (\Omega)} & \leq C |n|^{1-\gamma} e^{K_0 \theta_{\gamma,n} t} \| \mathcal{P}_n a \|_{L^2 (\Omega)} \nonumber \\
& \quad + \frac{C(1+|n|)^\frac12}{\nu^\frac12 (1+|n|^q)}  \int_0^t |n|^{1-\gamma} e^{K_0 \theta_{\gamma,n} (t-s)} e^{- K(s)\theta_{\gamma,n}} s^{-\frac12}\dd s \, \| u \|_{Y_{\gamma,K,T'}}^2 \nonumber \\
\begin{split}
& \leq \frac{C |n|^{1-\gamma} e^{-K(t) \theta_{\gamma,n}}}{(1+|n|^d)} \| \mathcal{P}_n a \|_{X_{d,\gamma,K}} \\
& \quad + \frac{C|n|^{\frac12+1-\gamma}}{\nu^\frac12 (1+|n|^q)}  \int_0^t  e^{K_0 \theta_{\gamma,n} (t-s)} e^{- K(s)\theta_{\gamma,n}} s^{-\frac12}\dd s \, \| u \|_{Y_{\gamma,K,T'}}^2\,.
\end{split}
\end{align}
Using \eqref{proof.thm.nonlinear.6} and \eqref{proof.thm.nonlinear.6}, we obtain from the similar calculation as in \eqref{proof.thm.nonlinear.8},
\begin{align}\label{proof.thm.nonlinear.16}
& \sup_{0<t\leq T'} \sup_{\nu^{-\frac{1}{2\gamma}} \leq |n| \leq \delta_0^{-1}\nu^{-\frac34}} (1+|n|^q) e^{K(t)\theta_{\gamma,n}} \| \mathcal{P}_n u(t) \|_{L^2 (\Omega)} \nonumber \\
&  \qquad \leq C \bigg (  \| a\|_{X_{d,\gamma,K}} +   \nu^{-\frac12} \| u \|_{Y_{\gamma,K,T'}}^2 \sup_{\nu^{-\frac{1}{2\gamma}} \leq |n| \leq \delta_0^{-1}\nu^{-\frac34}} (1+|n|)^{\frac12 + 1-\gamma} \min\{ \theta_{\gamma,n}^{-\frac12}, {T'}^\frac12\} \bigg ) \nonumber \\
&  \qquad \leq C \bigg (  \| a\|_{X_{d,\gamma,K}} +   \nu^{-\frac12} \| u \|_{Y_{\gamma,K,T'}}^2 \sup_{\nu^{-\frac{1}{2\gamma}} \leq |n| \leq \delta_0^{-1}\nu^{-\frac34}} (1+|n|)^{\frac32(1-\gamma)} \bigg ) \nonumber \\
& \qquad \leq C \bigg (  \| a\|_{X_{d,\gamma,K}} +   \nu^{-\frac12-\frac98(1-\gamma)} \| u \|_{Y_{\gamma,K,T'}}^2 \bigg ) \,.
\end{align}
In the first line we have used the fact $d-q-1+\gamma\geq 0$ by the choice of $q$.
As for the derivative estimate, we have from \eqref{est.thm.evolution.4},
\begin{align}\label{proof.thm.nonlinear.17}
\begin{split}
& \| \nabla \mathcal{P}_n u(t) \|_{L^2 (\Omega)} \\
& \leq \frac{C}{\nu^\frac12} \big ( \frac{1}{t^\frac12} +  |n|^{\frac12+\frac32(1-\gamma)} e^{K_0 \theta_{\gamma,n} (t-s)} \big )  \| \mathcal{P}_n a \|_{L^2 (\Omega)} \\
& \quad + \frac{C(1+|n|)^\frac12}{\nu (1+|n|^q)}  \int_0^t \big ( \frac{1}{(t-s)^\frac12} +  |n|^{\frac12 + \frac32 (1-\gamma)}  e^{K_0 \theta_{\gamma,n} (t-s)} \big ) e^{- K(s) \theta_{\gamma,n}} s^{-\frac12}\dd s \, \| u \|_{Y_{\gamma,K,T'}}^2\,.
\end{split}
\end{align}
Therefore, the similar computation as in  \eqref{proof.thm.nonlinear.16} using \eqref{proof.thm.nonlinear.24} yields
\begin{align}\label{proof.thm.nonlinear.18}
& \sup_{0<t\leq T'} \sup_{|n|^\gamma \nu^\frac12 <1} (1+|n|^q) e^{K(t)|n|^\gamma} (\nu t )^\frac12\| \nabla \mathcal{P}_n u (t) \|_{L^2 (\Omega)} \nonumber \\
& \leq  C\bigg (  \| a\|_{X_{d,\gamma,K}}   + \nu^{^\frac12} \| u\|_{Y_{\gamma,K,T'}}^2 \sup_{\nu^{-\frac{1}{2\gamma}}\leq |n| \leq \delta_0^{-1} \nu^{-\frac34}} (1+|n|)^\frac12 \min\{ \theta_{\gamma,n}^{-\frac12}\,, ~{T'}^\frac12\} \nonumber \\
& \quad + \nu^{-\frac12} \| u\|_{Y_{\gamma,K,T'}}^2  \sup_{\nu^{-\frac{1}{2\gamma}}\leq |n| \leq \delta_0^{-1} \nu^{-\frac34}} \, (1+|n|)^{1+\frac32 (1-\gamma)}\min \{ \theta_{\gamma,n}^{-1}\,, ~ {T'} \} \big ) \bigg ) \nonumber \\
\begin{split}
& \leq  C\bigg (  \| a\|_{X_{d,\gamma,K}}   +  \nu^{-\frac12-\frac{15}{8} (1-\gamma)} \| u \|_{Y_{\gamma,K,T'}}^2  \bigg ) \,.
\end{split}
\end{align}
Note that $\frac{15}{8} (1-\gamma)<\beta=\frac{2}{\gamma}(1-\gamma)$ for $\gamma<1$. 

\paragraph{Case (iii) $|n|>\delta_0^{-1}\nu^{-\frac34}$.}
In this case we apply Theorem \ref{prop.general.high}. 
For the estimate of $\| \mathcal{P}_n u(t) \|_{L^2(\Omega)}$
we have from \eqref{proof.thm.nonlinear.3} combined with \eqref{est.prop.general.high.1} and\eqref{proof.thm.nonlinear.5}, 
\begin{align}\label{proof.thm.nonlinear.19}
\| \mathcal{P}_n u(t) \|_{L^2 (\Omega)} & \leq C \| \mathcal{P}_n a \|_{L^2 (\Omega)} + \frac{C(1+|n|)^\frac12}{\nu^\frac12 (1+|n|^q)}  \int_0^t e^{-\frac14\nu |n|^2 (t-s)} e^{- K(s)\theta_{\gamma,n}} s^{-\frac12}\dd s \, \| u \|_{Y_{\gamma,K,T'}}^2\,.
\end{align}
Then we see 
\begin{align}\label{proof.thm.nonlinear.20}
I = (1+|n|)^\frac12 \int_0^t e^{-\frac14\nu |n|^2 (t-s)} e^{- K(s)\theta_{\gamma,n}} s^{-\frac12}\dd s\leq \frac{C e^{-K(t)\theta_{\gamma,n}}}{\nu^{\frac12 (1-\gamma)}}
\end{align}
for $|n|>\delta_0^{-1}\nu^{-\frac34}$. Indeed, when $\delta_0^{-1}\nu^{-\frac34}\leq |n|\leq \nu^{-1}$ we have 
\begin{align*}
I \leq C|n|^\frac12 \int_0^t e^{-K(s) \theta_{\gamma,n}} s^{-\frac12} \dd s \leq C |n|^\frac12 e^{-K(t)\theta_{\gamma,n}} \theta_{\gamma,n}^{-\frac12}\leq C \nu^{-\frac12 (1-\gamma)} e^{-K(t)\theta_{\gamma,n}}\,,
\end{align*}
while when $|n|\geq \nu^{-1}$ we use the factor $e^{-\frac14\nu n^2 (t-s)}$ in the integral of $I$, and then
\begin{align*}
I \leq C \int_0^t \frac{1}{\nu^\frac14 (t-s)^\frac14} e^{-K(s) \theta_{\gamma,n}} s^{-\frac12} \dd s \leq \frac{C e^{-K(t) \theta_{\gamma,n}}}{\nu^{\frac14} \theta_{\gamma,n}^\frac14} \leq C e^{-K(t)\theta_{\gamma,n}} \nu^{-\frac14 (1-\gamma)}\,,
\end{align*} 
which proves \eqref{proof.thm.nonlinear.20}. Then \eqref{proof.thm.nonlinear.19} and \eqref{proof.thm.nonlinear.20} give
\begin{align}\label{proof.thm.nonlinear.21}
\sup_{0<t\leq T'} \sup_{|n|\geq \delta_0^{-1}\nu^{-\frac34}} (1+|n|^q) e^{K(t) \theta_{\gamma,n}} \| \mathcal{P}_n u(t) \|_{L^2(\Omega)} \leq C \bigg ( \| a\|_{X_{d,\gamma,K}} + \nu^{-\frac12-\frac12 (1-\gamma)} \| u\|_{Y_{\gamma, K,T'}}^2\bigg ) \,.
\end{align}
The derivative of $\mathcal{P}_n u$ is estimated by using \eqref{est.prop.general.high.2}, and we obtain
\begin{align}\label{proof.thm.nonlinear.22}
\begin{split}
& \| \nabla \mathcal{P}_n u(t) \|_{L^2 (\Omega)} \\
& \leq \frac{C}{(\nu t)^\frac12} \big ( 1 +  |n|t \big ) e^{-\frac14 \nu n^2 t} \| \mathcal{P}_n a \|_{L^2 (\Omega)} \\
& \quad + \frac{C(1+|n|)^\frac12}{\nu (1+|n|^q)}  \int_0^t \big ( \frac{1}{(t-s)^\frac12} +  |n|(t-s)^\frac12 \big ) e^{-\frac14\nu |n|^2 (t-s)} e^{- K(s) \theta_{\gamma,n}} s^{-\frac12}\dd s \, \| u \|_{Y_{\gamma,K,T'}}^2\,.
\end{split}
\end{align}
The last integral has to be computed carefully as in \eqref{proof.thm.nonlinear.20}.
Our aim is to show 
\begin{align}
II & \, := \, (1+|n|)^\frac12 \int_0^t \big ( \frac{1}{(t-s)^\frac12} +  |n|(t-s)^\frac12 \big ) e^{-\frac14\nu |n|^2 (t-s) - K(s) \theta_{\gamma,n}} s^{-\frac12}\dd s \nonumber \\
& \leq \frac{Ce^{-K(t)\theta_{\gamma,n}}}{\nu^{\frac32(1-\gamma)} t^\frac12} \,,\label{proof.thm.nonlinear.23}
\end{align}
for $|n|\geq \delta_0^{-1}\nu^{-\frac34}$. We first consider the case $\delta_0^{-1}\nu^{-\frac34}\leq |n|\leq \nu^{-1}$. In this case we have from \eqref{proof.thm.nonlinear.24} and  $\theta_{\gamma,n}=|n|^\gamma \big (1+ (1-\gamma)\log (1+|n|) \big )$, 
\begin{align}\label{proof.thm.nonlinear.25}
II \leq \frac{C e^{-K(t)\theta_{\gamma,n}} }{t^\frac12} \big ( |n|^{\frac{1-\gamma}{2}} + |n|^{\frac32 (1-\gamma)} \big ) \leq \frac{Ce^{-K(t) \theta_{\gamma,n}} }{\nu^{\frac32 (1-\gamma)} t^\frac12} \,.
\end{align}
Next we consider the case $|n|\geq \nu^{-1}$. In this case we use the factor $e^{-\frac14\nu n^2 (t-s)}$ in the integral of $II$, and then 
\begin{align*}
II \leq C\int_0^t \big ( \frac{1}{\nu^{\frac14} (t-s)^{\frac34}} + \frac{1}{\nu^\frac34 (t-s)^\frac14} \big )  e^{-K(s)\theta_{\gamma,n}} s^{-\frac12} \dd s,
\end{align*}
which is bounded from above by, again from \eqref{proof.thm.nonlinear.24} and $|n|\geq \nu^{-1}$, 
\begin{align}\label{proof.thm.nonlinear.26}
C \big ( \frac{1}{\nu^\frac14 \theta_{\gamma,n}^\frac14} + \frac{1}{\nu^\frac34 \theta_{\gamma,n}^\frac34} \big ) t^{-\frac12} e^{-K(t) \theta_{\gamma,n}} \leq \frac{C e^{-K(t)\theta_{\gamma,n}}}{\nu^{\frac34 (1-\gamma)} t^\frac12}\,.  
\end{align}
Collecting \eqref{proof.thm.nonlinear.25} and \eqref{proof.thm.nonlinear.26}, we obtain \eqref{proof.thm.nonlinear.23}. Then \eqref{proof.thm.nonlinear.19} and \eqref{proof.thm.nonlinear.23} yield
\begin{align*}
& (1+|n|^q) e^{K(t)|n|^\gamma}(\nu t)^\frac12 \| \nabla \mathcal{P}_n u(t) \|_{L^2 (\Omega)} \nonumber \\
& \leq C \frac{(1+|n|t)}{1+|n|^{d-q}} e^{-2 K_0 \theta_{\gamma,n} t -\frac14\nu n^2 t} \| a\|_{X_{d,\gamma,K}}
+ C\nu^{-\frac12-\frac32(1-\gamma)} \| u \|_{Y_{\gamma,K,T'}}^2\nonumber \\
& \leq C \frac{(1+|n|^{1-\gamma})}{1+|n|^{d-q}}  \| a\|_{X_{d,\gamma,K}}
+ C\nu^{-\frac12-\frac32(1-\gamma)} \| u \|_{Y_{\gamma,K,T'}}^2 \,.\nonumber
\end{align*}
Since $d-q-1+\gamma\geq 0$ by the choice of $q$, we obtain 
\begin{align}
\sup_{|n|\geq \delta_0^{-1}\nu^{-\frac34}}  (1+|n|^q) e^{K(t)|n|^\gamma}(\nu t)^\frac12 \| \nabla \mathcal{P}_n u(t) \|_{L^2 (\Omega)} \leq C \big ( \| a\|_{X_{d,\gamma,K}} + \nu^{-\frac12-\frac32(1-\gamma)} \| u \|_{Y_{\gamma,K,T'}}^2 \big ) \,.\label{proof.thm.nonlinear.27}
\end{align}
Since $\beta=\frac{2(1-\gamma)}{\gamma}$ and $0<\gamma<1$ we have $\nu^{\beta-\frac32(1-\gamma)} \rightarrow 0$ as $\nu\rightarrow 0$.
Collecting \eqref{proof.thm.nonlinear.8}, \eqref{proof.thm.nonlinear.8'}, \eqref{proof.thm.nonlinear.13}, \eqref{proof.thm.nonlinear.14}, \eqref{proof.thm.nonlinear.16}, \eqref{proof.thm.nonlinear.18}, \eqref{proof.thm.nonlinear.21}, and \eqref{proof.thm.nonlinear.27}, we have arrived at, for $\gamma\in [\frac79,1)$ and $T'< \frac{K}{2K_0}$,
\begin{align}\label{proof.thm.nonlinear.28}
\begin{split}
\| u \|_{Y_{\gamma,K,T'}} &\leq C' \bigg ( \| a\|_{X_{d,\gamma,K_0}}  + \kappa \nu^{-\frac12-\beta} \| u\|_{Y_{\gamma,K,T'}}^2 \bigg )\,.
\end{split}
\end{align}
Here $C'>0$ is independent of small $T'$ and $\nu$, while $\kappa=\kappa (T',\nu)>0$ is taken as small enough  when $T'$ and $\nu$ are sufficiently small. Hence,  we can close the estimate if $T'$ and $\nu$ are small enough as long as $\| u \|_{Y_{\gamma,K,T'}}\leq C \nu^{\frac12+\beta}$, which is consistent with the condition $\| a\|_{X_{d,\gamma,K}}\leq \nu^{\frac12+\beta}$. Note that the number $T'$ is taken uniformly with respect to sufficiently small $\nu$. 
In particular, we obtain 
\begin{align}
\| u\|_{Y_{\gamma,K,T'}} \leq 2 C'  \| a\|_{X_{d,\gamma,K}} \,.\label{proof.thm.nonlinear.29}
\end{align}
The existence of solutions satisfying \eqref{proof.thm.nonlinear.29} is proved by the Banach fixed point theorem in the closed ball
$\{f \in Y_{\gamma,K,T'}~|~\| f \|_{Y_{\gamma,K,T'}} \leq 2 C'  \| a\|_{X_{d,\gamma,K}} \}$,  based on exactly the same calculation as above. 
Since the argument is standard we omit the details here. 
Note that if $u\in Y_{\gamma,K,T'}$ and $T'\in (0, \frac{K}{4K_0})$ then $u(t), (\nu t)^\frac12 \nabla u(t) \in X_{q,\gamma,2K'}\subset X_{d,\gamma,K'}$ with $K'=\frac{K}{4}$ for all $t\in (0,T']$, and $\displaystyle \sup_{0 < t\leq T'} \big ( \| u(t)\|_{X_{d,\gamma,K'}} + (\nu t)^\frac12 \| \nabla u(t) \|_{X_{d,\gamma,K'}} \big ) \leq C \| u \|_{Y_{\gamma,K,T'}}$ holds.  Hence, by using the embedding 
$$(\nu t)^\frac14 \| u(t) \|_{L^\infty(\Omega)} + (\nu t)^\frac12 \| \nabla u (t) \|_{L^2 (\Omega)} \leq C \big ( \| u(t) \|_{X_{d,\gamma, K'}} + (\nu t)^\frac12 \|\nabla u(t) \|_{X_{d,\gamma,K'}}\big )\,,$$
the estimate \eqref{est.thm.nonlinear.1} follows from \eqref{proof.thm.nonlinear.29}. The proof is complete.
\end{proofx}

\section*{Acknowledgements}
This work is supported by JSPS Program for Advancing Strategic International Networks
to Accelerate the Circulation of Talented Researchers,
'Development of Concentrated Mathematical Center Linking to Wisdom of the Next Generation',
which is organized by the Mathematical Institute of Tohoku University.
The first author also acknowledges the support of ANR project Dyficolty ANR-13-BS01-0003-01 and the support of program ANR-11-IDEX-005 of Universit\'e Sorbonne Paris Cit\'e.
The second author is grateful to Universit\'e Paris Diderot
and Courant Institute of Mathematical Sciences for their kind hospitality during his stay
from November 2015 to February 2016.

\appendix
\section{Proof of Lemma \ref{lem.psi_Ai}}

To prove Lemma \ref{lem.psi_Ai} we first show 
\begin{lem}\label{lem.Ai} Let $k=0,1,2$. There exists $C,M >0$ such that if $|\alpha \tilde \eps^{\frac13}| \le \frac12$, then   for all $z$ satisfying $|\arg (z)| \le \frac{5\pi}{6}$,  $|z| \ge M$,  
\begin{align}
\frac{1}{C} |z^{-\frac{5-2k}{4}} e^{-\frac23 z^\frac32} | \le     | \pa^k_z  Ai_\alpha (z) | & \leq C |z^{-\frac{5-2k}{4}} e^{-\frac23 z^\frac32} |.
\end{align}
\end{lem}

\begin{proof}
We first make the change of variables $u = |z|^{\frac12}t$ in  \eqref{defAialpha},  so that 
$$ Ai_\alpha(z) = \frac{1}{2i\pi} |z|^{-\frac12} \int_{|z|^{-\frac12} L} \frac{\exp\left(|z|^{\frac32} \big(e^{i \arg(z)} u   - \frac{u^3}{3}\big) \right)}{u^2 - |z|^{-1}(\tilde \eps)^{\frac23} \alpha^2} du. $$
By Cauchy's theorem, for $|z| \ge 1$, we can replace the contour $|z|^{-\frac12} L$ by $L$, so that 
$$ Ai_\alpha(z) = \frac{1}{2i\pi} |z|^{-\frac12} \int_{L} \frac{\exp\left(|z|^{\frac32}  \big( e^{i \arg(z)} u   - \frac{u^3}{3} \big)\right)}{u^2 - |z|^{-1}(\tilde \eps)^{\frac23} \alpha^2} du.
 $$
Set 
$$ {\cal A}(\rho, \theta, \eta)  = \frac{1}{2i \pi} \rho^{-1/2} \int_{L} \frac{\exp\left(\rho^{\frac32} \big( e^{i \theta} u   - \frac{u^3}{3}\big)\right)}{u^2 -  \rho^{-1} \eta^2} du.
 $$
where $\rho \ge 1$, $|\eta| \le \frac12$, $\theta \in [-\frac{5\pi}{6}, \frac{5\pi}{6}]$. We shall prove that:
\begin{equation}  \label{asympt_A}
{\cal A}(\rho, \theta, \eta) \sim \frac{1}{2 i \sqrt{\pi}}  (\rho e^{i \theta})^{-\frac54}\exp(-\frac23 (\rho e^{i \theta})^{\frac32}), \quad \rho \rightarrow +\infty
\end{equation}
uniformly in $\eta$ and $\theta$. Case $k=0$ in Lemma \ref{lem.Ai} follows immediately. Cases $k=1,2$ are treated similarly. 

\medskip
To prove \eqref{asympt_A}, we use the method of steepest descent. This method is exposed with much details in \cite{Die}, where an equivalent of the classical Airy function 
$$ A_i(\rho) =  \frac{1}{2i \pi} \rho^{1/2} \int_{L} \exp\left(\rho^{\frac32} (u   - \frac{u^3}{3})\right) du $$ 
is derived as $\rho \rightarrow +\infty$. In our setting, the idea is to replace the curve $L$ by a curve $L_\theta$ that goes  through a critical point of the phase $h_\theta(u) = e^{i \theta} u - \frac{u^3}{3}$. More precisely, one must choose at least locally  the curve of steepest descent for $\Re h_\theta$ through the critical point. By Cauchy-Riemann theorem, the gradients of $\Re h_\theta$ and  $\Im h_\theta$ are orthogonal, so that the imaginary part of the phase  remains constant along that curve. In this way, there is no oscillation of the phase along that curve, which allows application of the Laplace method to the integral, replacing $h_\theta$ and the factor $(u^2 - \rho^{-1} \eta^2)^{-1}$ by their leading part at the critical point. The difficulty lies in the choice of the good curve $L_\theta$, in the range $\theta \in [-\frac{5\pi}{6}, \frac{5\pi}{6}]$.
Before describing $L_\theta$, we collect information on the phase $h_\theta$. It has two critical points: $h'_\theta(u) = 0$ if and only if $u = \pm e^{i \frac{\theta}{2}}$. We denote $u_\theta = - e^{i\frac{\theta}{2}}$. The curves of steepest descent and ascent near $u = u_\theta$ satisfy $\Im h(u)  =\Im h(u_\theta)$. With $u = x+ i y$, this reads
\begin{equation} \label{Imh}
\sin \theta \, x + \cos \theta \, y - y x^2 + \frac{1}{3} y^3 + \frac{2}{3} \sin(\frac{3 \theta}{2}) = 0. 
\end{equation}
Note that this equation is invariant under the transformation $(y, \theta)  \rightarrow (- y, -\theta)$. 
{\em From now on, we therefore focus on the case $\theta \in [-\frac{5\pi}{6},0]$}, all arguments being transposable to $[0, \frac{5\pi}{6}]$. 

\medskip
\noindent
Let us denote by $\Gamma_\theta$ the curve of steepest descent through $u_\theta = x_\theta + i y_\theta$, which is at least locally defined. 
One computes $h_\theta''(u_\theta) = 2 e^{i \frac{\theta}{2}}$, so that for $u$ close to $u_\theta$, 
$$ h_\theta(u) - h_\theta(u_\theta) \sim \frac{h_\theta''(u_\theta)}{2} (u - u_\theta)^2 =  r^2 e^{i (\frac{\theta}{2} + 2 \varphi)}, \quad \text{with } u - u_\theta = r e^{i \varphi}. $$
The curve of steepest descent corresponds to the conditions
$$ \sin\big(\frac{\theta}{2} + 2 \varphi \big) = 0, \quad \cos\big(\frac{\theta}{2} + 2 \varphi\big) < 0. $$
Hence, it is tangent to the line $\phi = - \frac{\theta}{4} \pm \frac{\pi}{2}$. 

\medskip
\noindent
For $\theta = 0$, the computation is explicit and performed in \cite{Die}. The curve of steepest descent is given by  
$$ \Gamma_0 = \{ (x,y), 1 - x^2 + \frac{1}{3} y^2 = 0 \}. $$
This is a branch of hyperbola, see \cite{Die} for details.  For $\theta \in [-\frac{5\pi}{6},0)$, as the tangent line is $\phi = - \frac{\theta}{4} \pm \frac{\pi}{2}$, one has locally near $(x_\theta, y_\theta)$ that: 
\begin{itemize}
\item the upper branch of $\Gamma_\theta$ (corresponding to $y > y_\theta$)  goes in  the upper left direction: the points $(x,y)$ on this upper branch satisfy $y > y_\theta$, $x < x_\theta$. Moreover, from the expression of $\Im h$, 
$$ \pa_x \Im h_\theta(x,y) = \sin \theta + 2 xy, \quad \pa_y  \Im h_\theta(x,y)  = (\cos\theta - x^2)  + y^2. $$
It yields $\pa_x \Im h(x,y) > 0$, $\pa_y \Im h(x,y) >0$ on the upper branch, near $(x_\theta, y_\theta)$.
\item the lower branch of $\Gamma_\theta$ (corresponding to $y < y_\theta$)  goes in  the lower right direction: the points $(x,y)$ on this  lower branch satisfy $y < y_\theta$, $x > x_\theta$. Moreover $\pa_x \Im h(x,y) < 0$, $\pa_y \Im h(x,y) < 0$, near $(x_\theta, y_\theta)$.
\end{itemize}
Note  that $\pa_y \Im h$ is increasing with $y$ over $\{ y > 0\}$, as $\pa^2_y \Im h = 2y$. Also,  $\pa_x \Im h$ does not vanish  on  $\{x < x_\theta, y > y_\theta\}$, nor on $\{x_\theta<  x < 0, 0 < y < y_\theta\}$. From these remarks, it follows that the upper branch keeps pointing in the upper left direction and extends to infinity in $y$. Taking the limit $y \rightarrow +\infty$ in \eqref{Imh}, we see that $\Gamma_\theta$ is asymptotic to $-yx^2 + \frac13 y^3 = 0$, which means that
$$ \lim_{\substack{z \in \Gamma_\theta, \\ \Im z \rightarrow +\infty}} \arg(z) = \frac{2\pi}{3}. $$

\medskip
\noindent
From the same remarks, it follows that the lower branch of $\Gamma_\theta$ must escape the upper left quadrant. It can be either by the imaginary axis or by the real axis (depending on the value of $\theta$). In the special case $\theta = -\frac{2\pi}{3}$, direct verification shows that 
$$\Gamma_{-\frac{2\pi}{3}} \cap \{ x \le 0, y \ge 0\} = \exp(\frac{2i\pi}{3}) \R_+,$$
 so that it escapes by the origin. Nevertheless, for $\theta \in [-\frac\pi2,0[$, $\Gamma_\theta$ can never  cross the imaginary axis, as $\pa_y \Im h\vert_{x=0} = \cos \theta + y^2 > 0$. Moreover, still for 
$\theta \in [-\frac\pi2,0[$, $\pa_x \Im h = \sin \theta - 2 x y < 0$ for all $y < 0$ and $x < 0$. Hence,  the lower branch extends to infinity in $y$, and taking  the limit $y \rightarrow -\infty$ in \eqref{Imh}, we see that $\Gamma_\theta$ is asymptotic to $-yx^2 + \frac13 y^3 = 0$, which means that $$ \lim_{\substack{z \in \Gamma_\theta, \\ \Im z \rightarrow -\infty}} \arg(z) = -\frac{2\pi}{3}. $$

\medskip 
\noindent
We  now describe our choice for the contour $L_\theta$. We need to pay attention to the poles $\pm \rho^{-\frac12} \eta$ in the integrand, due to the denominator $u^2 - \rho^{-1} \eta^2$. We want to choose $L_\theta$ so that it is homotopic to $L$ without crossing the poles.  \begin{itemize}
\item If $\theta \in [-\frac\pi2,0]$,  $\Gamma_\theta$ exits the upper left quadrant through the negative real axis,  away from the origin (which is the exit point for $\theta = -\frac{2\pi}{3}$). Hence, for $\rho$ large enough, uniformly in $\theta \in [-\frac\pi2,0]$ and $|\eta| \le \frac12$, $\Gamma_\theta$ (oriented from bottom to top) is homotopic to $L$ without crossing the poles. We take $L_\theta = \Gamma_\theta$. 
\item If $\theta \in  [-\frac{5\pi}{6}, -\frac\pi2[$, we note that $\Re h(u_\theta) = -\frac23 \cos(\frac{3\theta}{2}) \ge \frac{\sqrt{2}}{3}$. Let $R > 0$ such that $R + \frac{R^3}{3} < \frac{\sqrt{2}}{6}$.  For $\rho$ large enough, the poles $\pm \rho^{-\frac12} \eta$ are in $D(0,\frac{R}{2})$. 
Let us denote by $z_\theta$ the point (either on the real or imaginary axis) at  which $\Gamma_\theta$ exits the upper left quadrant. There are three possibilities. 
\begin{itemize}
\item either $z_\theta$ is real negative and less than $-R$. In this case, we take 
$$ L_\theta = L_{\theta,-} \cup L_{\theta, 0} \cup L_{\theta,+}   $$
where 
$$L_{\theta,-} = \Big\{ r e^{-2i\pi/3}, r \in (R, +\infty) \Big\}, \quad L_0 = \Big\{ R e^{i \theta},  \theta \in [-\pi, -\frac{2\pi}{3}] \Big\}  \cup [z_\theta,-R],$$
and $L_{\theta, +} = \Gamma_\theta \cap \{ x < 0, y > 0 \}$ (oriented from bottom to top). 
\item or $z_\theta$ is on the imaginary axis, above $i R$. In this case, we take 
$$ L_\theta = L_{\theta,-} \cup L_{\theta, 0} \cup L_{\theta,+}   $$
where 
$$L_{\theta,-} = \Big\{ r e^{-2i\pi/3}, r \in (R,+\infty) \Big\}, \quad L_0 = \Big\{ R e^{i \theta},  \theta \in  [-\frac{3\pi}{2}, -\frac{2\pi}{3}] \Big\} \cup [iR, z_\theta],$$ 
and $L_{\theta, +} = \Gamma_\theta \cap \{ x < 0, y > 0 \}$ (oriented from bottom to top). 
\item or $\Gamma_\theta$ crosses the circle $\{ z = |R|\}$ before exiting the upper left quadrant, at some $z = R e^{i \alpha}$, $\alpha \in [-3\pi/2, -\pi]$. We set in this case 
$$ L_\theta = L_{\theta,-} \cup L_{\theta, 0} \cup L_{\theta,+}   $$
where 
$$L_{\theta,-} = \Big\{ r e^{-2i\pi/3}, r \in (R, +\infty) \Big\}, \quad L_0 =  \Big\{ R e^{i \theta},  \theta \in [\alpha, -\frac{2\pi}{3}] \Big\},$$ 
$L_{\theta, +} = \Gamma_\theta \cap \{ x < 0, y > 0, |x+i y| > R \}$ (oriented from bottom to top). 
\end{itemize}
In all three cases, the path $L_\theta$ is homotopic to $L$ away from the poles. Moreover, the important point is that that there exists $\mu > 0$ such that for all $z \in L_{\theta,-} \cup L_{\theta,0}$, 
\begin{equation} \label{realh}
\Re h(z) \le \Re h(u_\theta) - \mu.
\end{equation}
(this condition is the one needed to apply the method of steepest descent, see \cite{Die}). Indeed, from the choice of $R$, this condition is obviously satisfied over $|z| =R$.  If $z_\theta \in (-\infty, -R)$, we find for all $z \in [z_\theta,-R]$, 
$$  \Re h(z) = \cos \theta z - \frac{1}{3} z^3 \le \Re h(z_\theta) \le   \Re h(u_\theta) - \mu  $$
for some $\mu > 0$ (as $z_\theta$ belongs to the line of steepest descent).  
On the other hand, if $z_\theta \in [iR, i\infty]$, we find for all $z \in [iR, z_\theta]$, 
$$ \Re h(z) = - \sin\theta |z| + \frac{1}{3}  |z|^3 \le \Re h(z_\theta) \le   \Re h(u_\theta) - \mu  $$
for some $\mu > 0$.  Finally, for all $r \in [R,+\infty)$, 
$$ \Re h(r e^{-\frac{2i\pi}{3}}) = \cos(\theta - \frac{2\pi}{3}) r - \frac{r^3}{3} $$
which reaches its maximal value at $r = R$ when $\theta \in [-\frac{5\pi}{6}, 0]$. 
\end{itemize}
By Cauchy's theorem, we write 
%
$$ {\cal A}(\rho,\theta,\eta) =   \frac{1}{2i \pi} \rho^{-1/2} \int_{L_\theta} \frac{\exp(\rho^{\frac32} \left( e^{i \theta} u   - \frac{u^3}{3}\right)}{u^2 -  \rho^{-1} \eta^2} du. $$
The steepest method applies to the right-hand side: one can replace the phase $h_\theta(u)$, the factor $\frac{1}{u^2 - \rho^{-1} \eta^2}$, and the contour parametrization  by their leading order terms near $u = u_\theta$. We find, for any $a > 0$ small enough
\begin{align*}
  {\cal A}(\rho,\theta,\eta) & \sim \frac{1}{2i \pi} \rho^{-\frac12}\frac{\exp(-\frac23 (\rho e^{i\theta})^{\frac32})}{u_\theta - \rho^{-1} \eta^2} \int_{[u_\theta - i a e^{-i\frac\theta4}, u_\theta + i a e^{-i\frac\theta4} ]}    \exp(\rho^{\frac32} (u - u_\theta)^2 e^{i \frac\theta2})  du \\
   & \sim  \frac{1}{2 i \sqrt{\pi}}  (\rho e^{i \theta})^{-\frac54}\exp(-\frac23 (\rho e^{i \theta})^{\frac32})
   \end{align*}
as expected. 
 \end{proof}

\noindent 
{\it Proof of \eqref{est.lem.psi_Ai.3}.} Recall that $Z_c = Y_c + i \frac{\Im c_\epsilon}{\pa_Y V|_{Y=Y_c}}$ and $\tilde \epsilon = \frac{\epsilon}{\pa_Y V|_{Y= Y_c}}$, and it is assumed that $|Y_c| + \Im c_\epsilon \leq C|\epsilon|^{\frac{1-\theta}{3}}\ll 1$ for some small $\theta \in (0,\frac{1}{10})$.
Taking into account the pointwise estimate \eqref{est.lem.psi_Ai.1} we set 
\begin{align}\label{proof.est.lem.psi.Ai.3.0}
\begin{split}
\frac23 \big ( \frac{Y-Z_c}{\tilde \epsilon^\frac13} \big )^\frac32 - \frac23 \big (-\frac{Z_c}{\tilde \epsilon^\frac13}\big )^\frac32 & \, = \, \frac23 (n \pa_Y V|_{Y=Y_c} )^\frac12 \big ( f (Y)^\frac32 - f (0) ^\frac32 \big )\,, \\
f (Y) & \, = \,  e^{\frac{\pi}{6}i} \big ( Y- Y_c - i \frac{\Im c_\epsilon}{\pa_Y V|_{Y=Y_c}} \big ) \,.
\end{split}
\end{align} 
We first give the lower bound of $\Re (f(Y)^\frac32 - f(0)^\frac32)$. To this end we observe that the condition $\Im c_\epsilon>0$ implies 
\begin{align*}
-\pi <\arg ( Y-Y_c - i \frac{\Im c_\epsilon}{\pa_Y V|_{Y=Y_c}} ) <0\,, \qquad Y\geq 0\,,
\end{align*}
and therefore, 
\begin{align*}
-\frac{5\pi}{6} < \arg f(Y) < \frac{\pi}{6}\,, \qquad Y\geq 0\,.
\end{align*}
Set $z_1 = f(Y)$ and $z_0=f(0)$. Then the segment $\rho: \rho (t) = t z_1 + (1-t) z_0$, $0\leq t\leq 1$, lies in the half plane $-\frac{5\pi}{6}<\arg \rho (t) <\frac{\pi}{6}$, and hence, we see 
\begin{align*}
f(Y)^\frac32-f(0)^\frac32  \, = \, \int_\rho (z^\frac32)' \dd z  \, = \, \frac32 (z_1-z_0) \int_0^1 \big ( \rho (t) \big )^\frac12 \dd t & \, = \, \frac32 Y \int_0^1 \big \{ e^{\frac{\pi}{3}i} \rho (t)  \big \}^\frac12 \dd t\,.
\end{align*}
Set $\lambda (t) = \big \{ e^{\frac{\pi}{3}i} \rho (t)  \big \}^\frac12$. Since $-\frac{\pi}{2} <\arg \big ( e^{\frac{\pi}{3}i} \rho (t) \big ) <\frac{\pi}{2}$, we have 
\begin{align*}
\Re (\lambda (t) ) \geq \frac{1}{\sqrt{2}} |\rho (t)|^\frac12 \geq \frac18 \big ( |Y_c-t Y|^\frac12 + (\frac{\Im c_\epsilon}{\pa_Y V|_{Y=Y_c}})^\frac12 \big )\,,
\end{align*}
which implies 
\begin{align}\label{proof.est.lem.psi.Ai.3.1}
\Re \big (f(Y)^\frac32-f(0)^\frac32 \big ) \, = \, \frac32 Y \int_0^1 \Re (\lambda (t) ) \dd t \geq \frac{3Y}{16} \big ( \int_0^1 |Y_c - t Y|^\frac12 \dd t +  (\frac{\Im c_\epsilon}{\pa_Y V|_{Y=Y_c}})^\frac12 \big )\,.
\end{align}
It is not difficult to find a constant $C>0$ such that $\int_0^1 |Y_c - t Y|^\frac12 \dd t\geq \frac{1}{C} |Y_c|^\frac12$ for all  $Y\geq 0$.  Hence, there exists $C>0$ such that 
\begin{align}\label{proof.est.lem.psi.Ai.3.2}
\frac23 (n \pa_Y V|_{Y=Y_c})^\frac12 \Re \big (f(Y)^\frac32-f(0)^\frac32 \big )\geq \frac{1}{C} |\frac{Z_c}{\tilde \epsilon}|^\frac12 Y\,, \qquad Y\geq 0\,.
\end{align}
Then \eqref{est.lem.psi_Ai.1} combined with \eqref{proof.est.lem.psi.Ai.3.0} and \eqref{proof.est.lem.psi.Ai.3.2} gives the estimate 
\begin{align}\label{proof.est.lem.psi.Ai.3.3}
|\partial_Y^k \psi_{Ai} (Y)| \leq C |\tilde \epsilon|^{-\frac{k}{3}}  | \frac{Z_c}{\tilde \epsilon^\frac13} |^\frac54 | \frac{Y-Z_c}{\tilde \epsilon^\frac13} |^{-\frac{5-2k}{4}} e^{-\frac{1}{C} |\frac{Z_c}{\tilde \epsilon}|^\frac12 Y}\,, \qquad Y\geq 0\,,\quad k=0,1,2\,.
\end{align}
By using \eqref{proof.est.lem.psi.Ai.3.3}, $\|\partial_Y^k \psi_{Ai} \|_{L^2}$ is estimated in the following three cases:
(i) $Y_c\leq 0$ (ii) $0<Y_c\leq \Im c_\epsilon$ (iii) $Y_c \geq \Im c_\epsilon$.
Here we give the proof only for the cases (ii) and (iii). The case (i) is computed in the same manner.
We see
\begin{align*}
\| \partial_Y^k \psi_{Ai} \|_{L^2}^2 &  \leq   C |\tilde \epsilon|^{-\frac{2k}{3}}  | \frac{Z_c}{\tilde \epsilon^\frac13} |^\frac52 \int_0^{\frac{Y_c}{2}} | \frac{Y-Z_c}{\tilde \epsilon^\frac13} |^{-\frac{5-2k}{2}} e^{-\frac{2}{C} |\frac{Z_c}{\tilde \epsilon}|^\frac12 Y} \dd Y\\
& ~~ +  C |\tilde \epsilon|^{-\frac{2k}{3}}  | \frac{Z_c}{\tilde \epsilon^\frac13} |^\frac52 \int_{\frac{Y_c}{2}}^\infty  | \frac{Y-Z_c}{\tilde \epsilon^\frac13} |^{-\frac{5-2k}{2}} e^{-\frac{2}{C} |\frac{Z_c}{\tilde \epsilon}|^\frac12 Y} \dd Y\\
& =: I + II\,.
\end{align*}
The term $I$ is estimated as 
\begin{align*}
I \leq C |\tilde \epsilon|^{-\frac{2k}{3}}  | \frac{Z_c}{\tilde \epsilon^\frac13} |^\frac52 |\frac{Z_c}{\tilde \epsilon^\frac13} |^{-\frac{5-2k}{2}} \int_0^{\frac{Y_c}{2}}  e^{-\frac{2}{C} |\frac{Z_c}{\tilde \epsilon}|^\frac12 Y} \dd Y
& \leq C |\tilde \epsilon|^{-\frac{2k}{3}}  | \frac{Z_c}{\tilde \epsilon^\frac13} |^k |\tilde \epsilon|^\frac12 |Z_c|^{-\frac12}\\
& \leq C |\tilde \epsilon|^{\frac13-\frac{2k}{3}} |\frac{Z_c}{\tilde \epsilon^\frac13}|^{k-\frac12} \,.
\end{align*}
As for the term $II$, distinguishing the regions where $Y$ is away from $Y_c$ and $Y$ close to $Y_c$, we find
\begin{align*}
II & \leq C |\tilde \epsilon|^{-\frac{2k}{3}} |\frac{Z_c}{\tilde \epsilon^\frac13}|^\frac52
|\tilde \epsilon^\frac13|^{\frac{5-2k}{2}} |\frac{\tilde \epsilon}{Z_c} |^\frac12 e^{-\frac{1}{C} |\frac{Z_c}{\tilde \epsilon}|^\frac12 Y_c} \big ( |Z_c|^{-\frac{5-2k}{2}} + (\Im c_\epsilon)^{-\frac{5-2k}{2}} \big )\\
& \leq  C  |\tilde \epsilon|^{\frac13-\frac{2k}{3}} |\frac{Z_c}{\tilde \epsilon^\frac13}|^2 e^{-\frac{1}{C} |\frac{Z_c}{\tilde \epsilon}|^\frac12 Y_c} |\frac{\tilde \epsilon^\frac13}{\Im c_\epsilon}|^{\frac{5-2k}{2}}\,.
\end{align*}
If $0<Y_c\leq \Im c_\epsilon$ then $|Z_c|\leq C \Im c_\epsilon$, while if $Y_c\geq \Im c_\epsilon$ then 
$|\frac{Z_c}{\tilde \epsilon^\frac13}|^2 e^{-\frac{1}{C} |\frac{Z_c}{\tilde \epsilon}|^\frac12 Y_c} \leq C'$
since $Y_c \simeq|Z_c|$ in this case. Hence, both in the cases (ii) and (iii) we conclude that 
\begin{align*}
II \leq C |\tilde \epsilon|^{\frac13-\frac{2k}{3}} |\frac{Z_c}{\tilde \epsilon^\frac13}|^{k-\frac12}\,.
\end{align*}
Collecting the estimates for $I$ and $II$, we obtain \eqref{est.lem.psi_Ai.3}, for $|Z_c| \simeq |c_\eps|$ and $|\epsilon|\simeq |\tilde \epsilon|$. The proof is complete.

\section{Method of continuity in abstract framework}

\begin{prop}\label{prop.continuity} Let $L: D(L)\subset X \rightarrow X$ be a closed linear operator in a Banach space $X$. Let $\Sigma$ be an arcwise connected set in $\mathbb{C}$. Assume that 

\vspace{0.1cm}

\noindent
{\rm (i)} there exists $\mu_0\in \Sigma$ such that $\mu_0\in \rho (-L)$,

\noindent 
{\rm (ii)} there exists $C_{\Sigma}>0$ such that for any $\mu\in \Sigma$ and for any $f\in X$ there exists $u\in D(L)$ such that $(\mu + L) u =f$ and $\|u\|_X \leq C_{\Sigma} \| f\|_X$.

\vspace{0.1cm}

\noindent 
Then $\Sigma\subset \rho (-L)$.
\end{prop}

\begin{proof} Fix any $\mu\in \Sigma$ and take the continuous curve $\lambda :[0,1] \rightarrow \Sigma$ connecting $\mu_0$ and $\mu$: $\lambda (0)=\mu_0$, $\lambda (1)=\mu$, and $\lambda (t) \in \Sigma$ for all $t\in [0,1]$.  
Since $\mu_0\in \rho (-L)$ and $\rho(-L)$ is open, there exists $\delta>0$ such that $\lambda([0,\delta]) \subset  \rho (-L)$. In virtue of the condition (ii) we see $\|( \lambda (t) + L)^{-1}\|_{X\rightarrow X} \leq C_\Sigma$ for any $t\in [0,\delta]$. Then, by considering the Neumann series, we can take $\delta_\Sigma>0$ depending only on $C_\Sigma$ such that  $\lambda ([\delta, \delta+\delta_\Sigma]) \subset  \rho (-L)$. 
Again from (ii) we have $\| (\lambda(t) + L)^{-1} \|_{X\rightarrow X}\leq C_\Sigma$ for any $t\in [0,\delta+\delta_\Sigma]$. Repeating this argument we finally achieve $\mu\in \rho (-L)$. The proof is complete.
\end{proof}

\section{Estimate for heat equations}

\begin{prop}\label{prop.heat} Assume that $U \in BC^3(\R_+)$ satisfies $\| U \|  + \| Y \pa_Y^3 U \|_{L^\infty (\R_+)} <\infty$. Assume in addition that $U|_{Y=0} =0$ and $\displaystyle \lim_{Y\rightarrow \infty} U (Y)=U^E(0)$. Then there exists a unique solution to \eqref{eq.heat} such that 
\begin{align}\label{est.prop.heat}
\begin{split}
\sup_{0<t<T} \big ( \| U^P (t)\| + \| \partial_t U^P (t)  \|_{L^\infty (\R_+)} + \| Y \partial_t \pa_Y U^P (t) \|_{L^\infty (\R_+)} \big ) & \leq C \| U \|\,.
\end{split}
\end{align}
Moreover, if $U$ satisfies $\pa_Y U|_{Y=0}>0$ and $-M \pa_Y^2 U \geq (\pa_Y U)^2$ for all $Y\geq 0$, then for any $T>0$ and $\sigma\in (0,1]$ there exists $M_\sigma>0$ such that $-M_\sigma \pa_Y^2 U \geq (\pa_Y U)^2$ for all $Y\geq \sigma$ and $t\in [0,T]$.
\end{prop}

\begin{proof} Set 
\begin{align}\label{proof.pro.heat.1}
G_D (t,Y,Z) \, = \, G(t,Y-Z) - G(t,Y+Z)\,, \qquad G_N (t,Y,Z) \, = \, G (t,Y-Z) + G(t,Y+Z)\,,
\end{align}
where $G(t,Y) = (4\pi t)^{-\frac12} e^{-\frac{|Y|^2}{4t}}$ is the one-dimensional Gaussian.
Then the solution $U^P$ to \eqref{eq.heat} and its derivatives are represented by 
\begin{align}\label{proof.pro.heat.2}
\begin{split}
U^P(t,Y )  & \, = \, \int_0^\infty G_D (t,Y,Z) \, U(Z) \dd Z\,,\\
\pa_Y U^P (t,Y) & \, = \, \int_0^\infty G_N (t,Y,Z) \, \pa_Z U (Z) \dd Z\,,\\
\pa_t U^P (t,Y) & \, = \, \pa_Y^2 U^P (t,Y) \, = \, \int_0^\infty G_D (t,Y,Z) \, \pa_Z^2 U(Z) \dd Z\,,\\
\pa_t\pa_Y U^P (t,Y) & \, = \, 2 G(t,Y) \, \pa_Z^2 U|_{Z=0} + \int_0^\infty G_N (t,Y,Z) \, \pa_Z^2 U(Z) \dd Z \,.
\end{split}
\end{align}
From \eqref{proof.pro.heat.2} and the inequality $(1+Y^k) \leq C (|Y-Z|^k + 1+Z^k)$ for $Y, Z\geq 0$ it is straightforward to show \eqref{est.prop.heat} with a constant $C$ depending only on $T>0$. 
Next we assume that $\pa_Y U|_{Y=0}>0$ and $-M\pa_Y^2 U \geq (\pa_Y U)^2$ for all $Y\geq 0$.
Fix $T>0$ and $\sigma \in (0,1]$ and let $Y\geq \sigma$.  
From \eqref{proof.pro.heat.2} and the H${\rm \ddot{o}}$lder inequality we have 
\begin{align*}
(\pa_Y U^P (t,Y) )^2 & \leq \int_0^\infty G_N (t,Y,Z) \dd Z \int_0^\infty G_N (t,Y,Z) (\pa_Z U )^2 \dd Z = \int_0^\infty G_N (t,Y,Z) (\pa_Z U)^2 \dd Z\,,\\
-\pa_Y^2 U^P  (t,Y) &\geq \int_0^\infty M^{-1} G_D (t,Y,Z) (\pa_Z U)^2 \dd Z\,.
\end{align*}
Note that, since $(\pa_Z U)^2 \in BC(\overline{\R_+})$ there is $r\in (0,\frac{\sigma}{2}]$ such that
\begin{align*} 
\int_0^\infty G (t,Y-Z) (\pa_Z U)^2 \dd Z \leq 2 \int_{r}^\infty G (t,Y-Z) (\pa_Z U)^2 \dd Z\,,\qquad Y\geq \sigma\,, \quad  t\in (0,T]\,,
\end{align*}
which implies from $G_N (t,Y,Z) \leq 2G(t,Y-Z)$ for $Y,Z\geq 0$,
\begin{align*}
\int_0^\infty G_N (t,Y,Z) (\pa_Z U)^2 \dd Z \leq 4 \int_r^\infty G(t,Y-Z) (\pa_Z U)^2 \dd Z\,.
\end{align*}
Take large $M_\sigma>0$, which will be determined later.
Then we have 
\begin{align}\label{proof.pro.heat.3}
& -M_\sigma \pa_Y^2 U^P (t,Y) - (\pa_Y U^P (t,Y))^2 \nonumber \\
& \geq \int_r^\infty \bigg ( \frac{M_\sigma}{M} G_D (t,Y,Z) - 4 G (t,Y-Z) \bigg ) (\pa_Z U )^2 \dd Z \nonumber \\
& = \frac{M_\sigma}{M} \int_r^\infty \bigg ( (1-\frac{4 M}{M_\sigma}) G (t, Y-Z) - G(t, Y+Z) \bigg ) (\pa_Z U)^2 \dd Z\,.
\end{align}
Then we observe that, for $Y\geq \sigma$, $Z\geq r$, and $t\in (0,T]$,
\begin{align*}
(1-\frac{4M}{M_\sigma} )G(t,Y-Z)  - G (t,Y+Z) &  =  (4\pi t)^{-\frac12} e^{-\frac{Y^2+Z^2}{4\pi t}} \bigg ( (1-\frac{4M}{M_\sigma}) e^{\frac{Y Z}{2t}} - e^{-\frac{Y Z}{2t}} \bigg ) \nonumber \\
& \geq (4\pi t)^{-\frac12} e^{-\frac{Y^2+Z^2}{4\pi t}} \bigg ( (1-\frac{4M}{M_\sigma}) e^{\frac{\sigma r}{2T}} - e^{-\frac{\sigma r}{2T}} \bigg ) >0\,,
\end{align*}
if $M_\sigma \in (4M,\infty)$ is sufficiently large. Note that $M_\sigma$ depends only on $\sigma$, $r$, and $T$.
This shows that the positivity of the right-hand side of \eqref{proof.pro.heat.3}. The proof is complete.
\end{proof}

\bibliography{biblio_Prandtl_Gevrey}
\bibliographystyle{abbrv}


\end{document}